\documentclass[12pt,a4paper]{amsart}
\usepackage[top=30truemm,bottom=30truemm,left=25truemm,right=25truemm]{geometry}
\usepackage{amsfonts,amssymb,amscd,amsmath,latexsym,amsbsy, mathrsfs, amsxtra, mathdots, fancyhdr, color, tikz}
\usepackage{eucal}
\usetikzlibrary{arrows.meta, positioning}
\usepackage[all]{xy}
\usepackage{float}

\numberwithin{equation}{section}
\theoremstyle{plain}
\newtheorem{thm}{Theorem}[section]

\newtheorem{prop}[thm]{Proposition}

\newtheorem{defi}[thm]{Definition}
\newtheorem{lem}[thm]{Lemma}
\newtheorem{cor}[thm]{Corollary}

\theoremstyle{remark}
\newtheorem{rema}[thm]{Remark}

\newcommand{\ad}{{\mbox{\upshape{ad}}}}

\newcommand{\bc}{{\mathbf{c}}}

\newcommand{\bs}{{\mathbf{s}}}

\newcommand{\barx}{\overline{\phantom{x}}}

\newcommand{\N}{{\mathbb N}}

\newcommand{\cA}{{\mathcal A}}

\newcommand{\cB}{{\mathcal B}}

\newcommand{\cF}{{\mathcal F}}

\newcommand{\cG}{{\mathcal G}}
\newcommand{\cH}{{\mathcal H}}
\newcommand{\cI}{{\mathcal I}}

\newcommand{\cJ}{{\mathcal J}}

\newcommand{\cL}{{\mathcal L}}

\newcommand{\cO}{{\mathcal O}}
\newcommand{\cP}{{\mathcal P}}
\newcommand{\cR}{{\mathcal R}}

\newcommand{\cW}{{\mathcal W}}

\newcommand{\cop}{\mathrm{cop}}

\newcommand{\End}{\mbox{End}}
\newcommand{\Etil}{\widetilde{E}}

\newcommand{\field}{k}

\newcommand{\gfrak}{{\mathfrak g}}

\newcommand{\gr}{{\mathrm{gr}}}

\newcommand{\Hom}{{\mathrm{Hom}}}

\newcommand{\id}{{\mathrm{id}}}

\newcommand{\kfrak}{{\mathfrak k}}
\newcommand{\kow}{{\varDelta}}
\newcommand{\lact}{{\triangleright}}

\newcommand{\mutil}{\tilde{\mu}}

\newcommand{\lfrak}{{\mathfrak l}}

\newcommand{\obar}{\overline{\phantom{o}}}

\newcommand{\Oint}{\mathcal{O}_{\mathrm{int}}}

\newcommand{\ot}{\otimes}

\newcommand{\poly}{{\mathrm{poly}}}

\newcommand{\qfield}{\mathbb{K}}

\newcommand{\rfrak}{{\mathfrak r}}

\newcommand{\ract}{\triangleleft}

\newcommand{\scrU}{\mathscr{U}}

\newcommand{\sfrak}{\mathfrak{s}}

\newcommand{\ubar}{\overline{u}}

\newcommand{\Uq}{U}

\newcommand{\uqg}{{U_q(\mathfrak{g})}}

\newcommand{\uqgX}{U_q(\mathfrak{g}_X)}

\newcommand{\vep}{\varepsilon}

\newcommand{\Xfrak}{\mathfrak{X}}

\newcommand{\Z}{{\mathbb Z}}

\title{Short star products for quantum symmetric pairs and applications}

\dedicatory{To the memory of Gail Letzter}

\author[Stefan Kolb]{Stefan Kolb}
\address{School of Mathematics, Statistics and Physics,
Newcastle University, Newcastle upon Tyne NE1 7RU, United Kingdom}
\email{stefan.kolb@newcastle.ac.uk}
\author[Milen Yakimov]{Milen Yakimov}
\address{
Department of Mathematics, Northeastern University, Boston, MA 02115, U.S.A. and International Center for Mathematical Sciences, Institute of Mathematics and Informatics \\
Bulgarian Academy of Sciences \\ 
Acad. G. Bonchev Str., Bl. 8 \\
Sofia 1113, Bulgaria
}
\email{m.yakimov@northeastern.edu}


\begin{document}
%%%%%%%%%%%%%%%%%%%%%%%%%%%%%%%%%%%%%%
\begin{abstract}
    We prove that the star product for quantum symmetric pair coideal subalgebras is short. We apply this result to obtain new conceptual proofs, from first principles, of several fundamental facts about quantum symmetric pairs. In particular, we establish the existence of the algebra anti-automorphism $\sigma_\tau$ and of the bar involution, without making use of the quasi $K$-matrix. We give a new elementary proof of a conjecture by Balagovi\'c and Kolb, sometimes referred to as the fundamental lemma for quantum symmetric pairs. We obtain a conceptual formula expressing the tensor quasi $K$-matrix in terms of the much studied quasi $R$-matrix and the Letzter map. This also allows for a new independent proof of the intertwiner property of the quasi $K$-matrix.  
\end{abstract}
\maketitle
%%%%%%%%%%%%%%%%%%%%%%%%%%%%%%%%%%%%%%
\section{Introduction}
%%%%%%%%%%%%%%%%%%%%%%%%%%%%%%%%%%%%%%
%%%%%%%%%%%%%%%%%%%%%%%%%%%%%%%%%%%%%%
\subsection{Short star products} The concept of short star products for commutative graded Poisson algebras originated in the context of conformal field theory in the work of Beem, Peelaers and Rastelli \cite{a-BPR17} and was further developed mathematically by Etingof and Stryker \cite{a-ES}. Any $\N$-graded algebra $A=\bigoplus_{k\in \N}A_n$ is filtered with $\cF_n(A)=\bigoplus_{k=0}^n A_k$. A star product on $A$ is an additional associative algebra structure $\ast:A\ot A\rightarrow A$ which turns $(A,\cF,\ast)$ into a filtered algebra such that the associated graded algebra $\gr_\cF(A,\ast)$ is isomorphic to $A$ under the canonical isomorphism of vector spaces $\gr_\cF(A)\cong A$. 
Recall that if $(B,\cG)$ is a filtered algebra then the associated graded algebra $\gr_\cG(B)$, if commutative, naturally carries a Poisson structure. In this case $(B,\cG)$ is called a filtered quantization of the commutative graded Poisson algebra $\gr_\cG(B)$. Star products on commutative graded algebras $A$ are hence particular examples of filtered quantizations of $A$.

The papers \cite{a-BPR17} and \cite{a-ES} identify a crucial concept for star products: A star product $\ast$ on $A$ is called short if
\begin{align}\label{eq:short}
  A_m \ast A_n \subseteq \bigoplus_{k=|m-n|}^{m+n} A_k \qquad \mbox{for all $m,n\in \N$.}
\end{align}
The nonzero lower bound in the direct sum in \eqref{eq:short} poses a strong restriction on the possible filtered deformations. In \cite{a-ES} short star products on commutative, graded Poisson algebras were identified as a new, important structure in representation theory.

Star products on $A$ can be obtained via quantization maps. Let $(B,\cG)$ be a filtered associative algebra. A quantization map is an isomorphism of filtered vector spaces $\phi:A\rightarrow B$ such that the associated graded map $\gr(\phi):\gr_\cF(A)\cong A \rightarrow \gr_\cG(B)$ is an isomorphism of graded algebras. If $\phi$ is a quantization map then
\begin{align*}
  a\ast b := \phi^{-1}(\phi(a) \phi(b)) \qquad \mbox{for all $a,b\in A$}
\end{align*}
defines a star product on $A$, see \cite[Section 2.2]{a-ES}.
%%%%%%%%%%%%%%%%%%%%%%%%%%%%%%%%%%%%%%
\subsection{Quantum symmetric pairs.}
%%%%%%%%%%%%%%%%%%%%%%%%%%%%%%%%%%%%%%
A comprehensive theory of quantum symmetric pairs was developed for semisimple complex Lie algebras by Letzter \cite{a-Letzter99a}, \cite{MSRI-Letzter} and extended to the Kac-Moody case in \cite{a-Kolb14}. Let $\gfrak$ be a symmetrizable Kac-Moody algebra with generalized Cartan matrix $(a_{ij})_{i,j\in I}$ for some index set $I$. Recall that a Satake diagram $(X,\tau)$ consists of a subset $X\subset I$ of finite type and an involutive diagram automorphism $\tau:I\rightarrow I$, satisfying certain compatibility conditions, see e.g.~\cite[Definition 2.3]{a-Kolb14}. Each Satake diagram determines an involutive Lie algebra automorphism $\theta=\theta(X,\tau):\gfrak \rightarrow \gfrak$. Let $\kfrak=\{x\in \gfrak\,|\,\theta(x)=x\}$ be the corresponding fixed Lie subalgebra.

We consider the quantized enveloping algebra $\Uq=\uqg$ of $\gfrak$ with generators $E_i, F_i, K_i^{\pm 1}$ for $i\in I$ defined over the field $\qfield=k(q)$ where $k$ is a field of characteristic zero. To each Satake diagram $(X,\tau)$ and suitable parameters $\bc=(c_i)_{i\in I\setminus X}\in (\qfield^{\times})^{I\setminus X}$ the theory of quantum symmetric pairs associates a subalgebra $\cB_\bc\subset \uqg$ which is a quantum group analogue of the universal enveloping algebra $U(\kfrak)$. Crucially, $\cB_\bc$ is a right coideal subalgebra, that is, the coproduct $\kow$ of $\Uq$ satisfies the relation
\begin{align*}
  \kow(\cB_\bc)\subset \cB_\bc\ot \Uq.
\end{align*}
We will refer to the algebra $\cB_\bc$ as a quantum symmetric pair coideal subalgebra (or QSP-subalgebra) of $\Uq$. By definition, the QSP-subalgebra is generated by a Hopf subalgebra $\cH_{X,\tau}$ of $\Uq$ and additional generators $B_i$ for $i\in I\setminus X$, see Section \ref{sec:QSP} for a precise definition. Hence, the QSP-subalgebra $\cB_\bc$ is contained in the subbialgebra $U^\poly\subset \Uq$ generated by $\cH_{X,\tau}$ and the elements $F_i, K_i^{-1}, E_iK_i^{-1}$ for all $i\in I\setminus X$. We call the subalgebra $\cA_{X,\tau}^-\subset U^\poly$ generated by $\cH_{X,\tau}$ and $F_i$ for all $i\in I\setminus X$ the quantum horospherical subalgebra. As shown in \cite{a-KY21}, the vector space $U^\poly$ decomposes into a direct sum
\begin{align}\label{eq:Upoly=AJ}
  U^\poly =\cA_{X,\tau}^- \oplus \cJ_{X,\tau}
\end{align}
where $\cJ_{X,\tau}$ denotes the left ideal of $U^\poly$ generated by $\ad_r(\cH_{X,\tau})(E_iK_i^{-1})$ and $K_i^{-1}$ for all $i\in I\setminus X$. The Letzter map $\psi:\cB_\bc\rightarrow \cA_{X,\tau}^-$ is defined as the composition of the inclusion $\cB_\bc\hookrightarrow U^\poly$ with the projection onto $\cA_{X,\tau}^-$ with respect to the direct sum decomposition \eqref{eq:Upoly=AJ}. We showed in \cite{a-KY21} that the Letzter map is a filtered isomorphism of vector spaces and the associated graded map is an isomorphism of graded algebras. This means that the inverse of the Letzter map is a quantization map in the sense of \cite{a-ES} and hence gives rise to a star product on $\cA_{X,\tau}^-$ defined by
\begin{align}\label{eq:Bstar}
  a \ast b =\psi(\psi^{-1}(a) \psi^{-1}(b)) \qquad \mbox{for all $a,b\in \cA_{X,\tau}^-$.}
\end{align}
As the algebra $\cA_{X,\tau}^-$ is generated in degrees $0$ and $1$, all the information about the star product on $\cA_{X,\tau}^-$ is contained in either family of maps $\mu_i^L, \mu_i^R:\cA_{X,\tau}^-\rightarrow \cA_{X,\tau}^-$ for $i\in I\setminus X$ defined by
\begin{align*}
   \mu_i^L(a)=F_i\ast a - F_ia, \qquad  \mu_i^R(a)=a\ast F_i - aF_i \qquad \mbox{for all $a\in \cA_{X,\tau}^-$.}
\end{align*}
In \cite[Lemma 3.3]{a-KY21} we obtained an explicit formula for the maps $\mu_i^L$, see also Equation \eqref{eq:muiL}. However, a similar formula for $\mu_i^R$ remained so far elusive if $X\neq \emptyset$.
%%%%%%%%%%%%%%%%%%%%%%%%%%%%%%%%%%%%%
\subsection{Goal}
%%%%%%%%%%%%%%%%%%%%%%%%%%%%%%%%%%%%%
With the present paper we argue that short star products are a powerful tool also for non-commutative, graded algebras. While we lose the underlying Poisson structure, the concepts of filtered deformations, quantization maps and shortness equally make sense. We demonstrate the pivotal role of the shortness property in the case of quantum symmetric pairs. Hence, the first main goal of the present paper is to show that the star product \eqref{eq:Bstar} on the quantum horospherical subalgebra $\cA_{X,\tau}^-$ is indeed short.

In the second half of the paper we show that the perspective of quantum symmetric pairs as short filtered deformations allows us to conceptualize various fundamental notions and results in this theory and to simplify their proofs. A similar perspective emerged previously in \cite{a-KY20} in a more general context of Nichols algebras of diagonal type. However, that paper only considered quantum symmetric pairs of quasi split type,  i.e.~$X=\emptyset$. The quasi-split case is substantially easier and hence the shortness property did not feature explicitly in \cite{a-KY20}. In the present paper, we develop the theory for arbitrary (generalized) Satake diagrams $(X,\tau)$ in the setting of Drinfeld-Jimbo quantum groups.

As applications of the shortness of the star product, we obtain a new, conceptual construction of the anti-automorphism $\sigma_\tau:\cB_\bc \rightarrow \cB_\bc$ first observed in \cite{a-WangZhang23}, and of the bar involution $\obar^B:\cB_\bc \rightarrow \cB_\bc$ as formulated in \cite{a-Kolb22}. Moreover, we give a new, elementary proof of a conjecture by Balagovi\'c and Kolb \cite[Conjecture 2.7]{a-BalaKolb15}, \cite{a-BaoWang21}. Most importantly, and as the second main goal of the paper, we obtain a new, conceptual formula for the tensor quasi $K$-matrix for quantum symmetric pairs. In this approach, the intertwiner property of the quasi $K$-matrix is obtained by rewriting the intertwiner property of the quasi $R$-matrix in terms of the star product.
%%%%%%%%%%%%%%%%%%%%%%%%%%%%%%%%%%%%%%
\subsection{Results}
%%%%%%%%%%%%%%%%%%%%%%%%%%%%%%%%%%%%%%
Let $\cW_{X,\tau}$ be the quotient of $U^\poly$ by the two sided ideal generated by all $K_i^{-1}$ for $i\in I\setminus X$. We call $\cW_{X,\tau}$ the horospherical Heisenberg double corresponding to the (generalized) Satake diagram $(X,\tau)$. We obtain the commutative diagram
\begin{align*}
  \xymatrix{
    \cB_\bc \ar@{->}_{\psi|_{\cB_\bc}}[ddr] \ar@{^{(}->}[rr] \ar@{->}^{\pi|_{\cB_\bc}}[dr]& & U^\poly \ar@{->}^{\psi}[ddl] \ar@{->}_{\pi}[dl]\\
    & \cW_{X,\tau} \ar@{->}^{\eta}[d]\\
    &\cA_{X,\tau}^-&
    }
\end{align*}
where $\pi:U^\poly\rightarrow \cW_{X,\tau}$ and $\eta:\cW_{X,\tau}\rightarrow \cA_{X,\tau}^-$ denote the canonical projections. Both $U^\poly$ and $\cW_{X,\tau}$ have natural vector space gradings such that the projection $\pi$ is a graded map, see \eqref{eq:gradings}. The QSP-subalgebra $\cB_\bc$ is not a graded subspace of $U^\poly$. However, we have the following crucial result.

\medskip

\noindent{\bf Theorem A.} [Theorem \ref{thm:Bc-graded}] \textit{The image $\pi(\cB_\bc)$ is a graded subspace of $\cW_{X,\tau}$.}

\medskip
The projection $\eta$ turns $\cA_{X,\tau}^-$ into a left $\cW_{X,\tau}$-module which can be thought of as a Fock space for the horospherical Heisenberg double. The algebra $\cW_{X,\tau}$ has a triangular decomposition
\begin{align*}
  \cW_{X,\tau}\cong \cR_X^{-,r}\ot \cH_{X,\tau}\ot \cR_X^{+,r}
\end{align*}
where $\cR_X^{-,r}$ and $\cR_X^{+,r}$ are quantum nilradicals of the negative and positive standard parabolics corresponding to the subset $X$. The elements of $\cR_X^{-,r}$ and $\cR_X^{+,r}$ act as lowering and raising operators, respectively, on the $\N$-graded vector space $\cA_{X,\tau}^-$. This perspective, together with Theorem A, allows us the prove the central result of this paper.

\medskip

\noindent{\bf Theorem B.} [Theorem \ref{thm:short}] \textit{The star product on $\cA_{X,\tau}^-$ defined by Equation \eqref{eq:Bstar} is short.}

\medskip

In the sequel, we will give several major applications of Theorem B.
%%%%%%%%%%%%%%%%%%%%%%%%%%%%%%
\subsubsection{Application 1: The algebra automorphism $\sigma_\tau$}
%%%%%%%%%%%%%%%%%%%%%%%%%%%%%%
Theorem B implies that the lower order term maps $\mu_i^L$ and $\mu_i^R$ for $i\in I\setminus X$ are both homogeneous of degree $-1$. This allows us to compare graded components of iterated star products, see for example Proposition \ref{prop:muR-muL}. By an inductive argument, we obtain an explicit formula for the lower order term map $\mu_i^R$. Recall the algebra anti-isomorphism $\sigma:\Uq\rightarrow \Uq$ defined in \cite[3.1.3]{b-Lusztig94}, see also \eqref{eq:sigma-def}. See Section \ref{sec:background} for the notation in Equation \eqref{eq:muiR-explicit}.

\medskip

\noindent{\bf Theorem C.} [Lemma \ref{lem:mutilRi}, Theorem \ref{thm:muiR}] \textit{For all $i\in I\setminus X$ we have
\begin{align*}
  \mu_i^R = \sigma\circ \tau \circ \mu_i^L \circ \sigma \circ \tau,
\end{align*}
or explicitly, for any $u\in \sigma(\cR_X^{-,r})$, we have
\begin{align}\label{eq:muiR-explicit}
  \mu_i^R(u) = -c_{\tau(i)} \frac{q^{-(\alpha_i,\theta(\alpha_i))}}{q_i-q_i^{-1}}  T^{-1}_{w_X} \circ \partial^R_{\tau(i)}\circ T_{w_X}(u)K_{\alpha_i+\theta(\alpha_i)}.
\end{align}
}

\medskip

Theorem C has the following immediate consequence.

\medskip

\noindent{\bf Theorem D. } [Theorem \ref{thm:sigmatau-anti}, Corollary \ref{cor:sigmatauB}] \textit{The map $\sigma\circ \tau:\cA_{X,\tau}^-\rightarrow \cA_{X,\tau}^-$ is an algebra anti-automorphism of the star product algebra $(\cA_{X,\tau}^-,\ast)$. Hence, the linear map $\sigma_\tau:\cB_\bc\rightarrow \cB_\bc$ defined by $\sigma_\tau=\psi^{-1}\circ \sigma\circ \tau \circ \psi$ is an algebra anti-automorphism which satisfies $\sigma_\tau|_{\cH_{X,\tau}}=\sigma\circ \tau|_{\cH_{X,\tau}}$ and $\sigma_\tau(B_i)=B_{\tau(i)}$ for all $i\in I\setminus X$.}

\medskip

 The existence of the algebra anti-automorphism $\sigma_\tau:\cB_\bc\rightarrow \cB_\bc$ was first observed by Wang and Zhang in \cite[Proposition 3.14]{a-WangZhang23}. Their construction relies on Appel and Vlaar's improvement \cite[Section 7]{a-AppelVlaar22} of the construction of the quasi $K$-matrix in \cite[Section 6]{a-BalaKolb19}. The construction of the anti-automorphism $\sigma_\tau$ in the present paper does not rely on the prior construction of the quasi $K$-matrix. Moreover, the star-product interpretation of quantum symmetric pairs clarifies the conceptual origin of the map $\sigma_\tau$.
%%%%%%%%%%%%%%%%%%%%%%%%%%%%%%

%%%%%%%%%%%%%%%%%%%%%%%%%%%%%%
\subsubsection{Application 2: The bar involution for $\cB_\bc$}
%%%%%%%%%%%%%%%%%%%%%%%%%%%%%%
Theorem D provides a new proof of the existence of the bar involution for $\cB_\bc$ for suitable parameters. Indeed, let $\obar:\Uq\rightarrow \Uq$ denote the bar involution for $\Uq$, see \eqref{eq:obar-def}. For parameters $\bc=(c_i)_{i\in I\setminus X}$ satisfying the condition
\begin{align}\label{eq:ci-bar-condition}
  \overline{c_i}=(-1)^{2\alpha_i(\rho_X^\vee)} q^{(2\rho_X-\theta(\alpha_i),\alpha_i)} c_{\tau(i)} \qquad \mbox{for all $i\in I\setminus X$}
\end{align}
one sees directly that the map $\sigma\circ\tau\circ\obar:\Uq\rightarrow \Uq$ restricts to an anti-automorphism of $\cB_\bc$ which maps $B_i$ to $B_{\tau(i)}$, see Lemma \ref{lem:BBt}. Now one obtains the bar involution for $\cB_\bc$ by composition with $\sigma_\tau$.

\medskip

\noindent{\bf Theorem E.} [Corollary \ref{cor:barB}] \textit{If \eqref{eq:ci-bar-condition} holds then the map
\begin{align*}
  \obar^B:=\sigma_\tau\circ \sigma \circ \tau \circ \obar:\cB_\bc\rightarrow \cB_\bc, \qquad b\mapsto \overline{b}^B
\end{align*}
defines a $k$-algebra automorphism of $\cB_\bc$ such that $\obar^B|_{\cH_{X,\tau}}=\obar|_{\cH_{X,\tau}}$ and $\overline{B_i}^B=B_i$ for all $i\in I\setminus X$.
}
\medskip  

The existence of the bar involution of $\cB_\bc$ for parameters satisfying \eqref{eq:ci-bar-condition} had previously been proved in \cite{a-Kolb22} as a consequence of the improved construction of the quasi $K$-matrix in \cite{a-AppelVlaar22}. Again, the proof of Theorem E in the present paper is much simplified and does not rely on the prior construction of the quasi $K$-matrix.
%%%%%%%%%%%%%%%%%%%%%%%%%%%%%%%%%%%%%%%%%
\subsubsection{Application 3: The fundamental lemma for quantum symmetric pairs}\label{sec:intro-fundamental}
%%%%%%%%%%%%%%%%%%%%%%%%%%%%%%%%%%%%%%%%%
The following formula was conjectured in \cite[Conjecture 2.7]{a-BalaKolb15}:
\begin{align}\label{eq:fund-intro}
   \sigma\circ \tau(\partial^R_i(T_{w_X}(E_i)))=\partial^R_i(T_{w_X}(E_i)) \qquad \mbox{for all $i\in I\setminus X$,}
\end{align}
see Section \ref{sec:background} for notation. Up to a sign $\nu_i\in \{\pm 1\}$, relation \eqref{eq:fund-intro} was proved in \cite[Proposition 2.5]{a-BalaKolb15} for generalized Satake diagrams in the Kac-Moody setting, see also \cite[Remark 2.6]{a-BalaKolb15}. For $\gfrak$ of finite type, relation \eqref{eq:fund2} was proved in \cite[Proposition 2.3]{a-BalaKolb15} by casework. In \cite[Section 4]{a-BaoWang21} a sophisticated argument involving canonical bases was used to prove that $\nu_i=1$ also in the Kac-Moody setting.
In \cite[1.3]{a-BaoWang21}, formula \eqref{eq:fund-intro} was identified as a `\textit{major technical assumption ... toward the existence of bar involution, quasi-$K$-matrix and universal $K$-matrix, and as well as for the constructions of $\imath$-canonical bases in}'  \cite{a-BaoWang21}. For this reason, Bao and Wang referred to Equation \eqref{eq:fund-intro} as the \textit{fundamental lemma of QSP}. Relation \eqref{eq:fund-intro} also features in the improved construction of the quasi $K$-matrix in \cite{a-AppelVlaar22}.

In the present paper (Proposition \ref{prop:fundamental}, Corollary \ref{cor:fundamental}) we show that relation \eqref{eq:fund-intro} follows directly from the relation
\begin{align*}
  \mu^L_{\tau(i)}(F_i)=\mu^R_{i}(F_{\tau(i)})
\end{align*}
which we can write down explicitly thanks to Theorem C.
This shows that the star-product interpretation of quantum symmetric pairs provides a natural setting for an independent, elementary proof of relation \eqref{eq:fund-intro}. The present proof does not rely on results of either \cite{a-BalaKolb15} or \cite{a-BaoWang21}.
%%%%%%%%%%%%%%%%%%%%%%%%%%%%%%
\subsubsection{Application 4: The quasi $K$-matrix for $\cB_\bc$}
%%%%%%%%%%%%%%%%%%%%%%%%%%%%%%
The concept of the ordinary quasi $K$-matrix $\Xfrak$ and the tensor quasi $K$-matrix $\Theta^B$ was introduced into the theory of quantum symmetric pairs, ingeniously, by Bao and Wang in \cite{a-BaoWang18a}. The element $\Xfrak$ was defined by the intertwiner relation
\begin{align}\label{eq:intertwiner-intro}
  \overline{b}^B \Xfrak =\Xfrak \overline{b} \qquad \mbox{for all $b\in \cB_\bc$}
\end{align}
which is natural in the context of canonical bases for $\cB_\bc$-modules obtained by restriction from $\Uq$-modules. However, then hard work is required to prove the existence of $\Xfrak$, see \cite{a-BaoWang18a}, \cite{a-BalaKolb19} and the improvement in \cite{a-AppelVlaar22}. In \cite{a-BaoWang18a} the tensor quasi $K$-matrix $\Theta^B$ was introduced as a secondary object. However, $\Theta^B$ reflects the coideal property which may be considered more fundamental than the fact that $\cB_\bc$ is a subalgebra. For this reason it is desirable to have a closed formula for the tensor quasi $K$-matrix. Such a formula can be obtained via the Letzter map. Let $\Theta$ be the quasi $R$-matrix for $\Uq$ and define
\begin{align}\label{eq:quasiK-intro}
  \Theta^B=(\psi^{-1}\ot \id)(\Theta).
\end{align}  
The following Theorem shows that $\Theta^B$ satisfies the intertwiner property of the tensor quasi $K$-matrix.

\medskip
\noindent {\bf Theorem F.} [Theorem \ref{thm:intertwiner2}]
  The element $\Theta^B$ defined by \eqref{eq:quasiK-intro} satisfies the relation
  \begin{align}\label{eq:ThetaB-intertwiner-intro}
    \kow(\sigma_\tau(b)) \cdot\Theta^B = \Theta^B\cdot (\sigma_\tau\ot(\sigma\circ\tau))\circ \kow(b) \qquad \mbox{for all $b\in \cB_\bc$.}
  \end{align}
  
The conceptual construction \eqref{eq:quasiK-intro} of the quasi $K$-matrix was already observed in the quasi-split case in \cite{a-KY20}. However, the general case is technically substantially more involved. In particular, we will need to revisit the coproduct of the generators $B_i$ of $\cB_\bc$. This will lead us to the elements
\begin{align*}
  M_i = q^{-(\theta(\alpha_i),\alpha_i)}(K_i^{-1}\ot 1)\big(\kow(T_{w_X}(E_{\tau(i)}))- T_{w_X}(E_{\tau(i)})\ot 1\big) \qquad \mbox{for $i\in I\setminus X$.}
\end{align*}  
We will show that the elements $M_i$ can be used to compactly rewrite the lower order term maps $\mu_i^L$ and $\mu_i^R$ in terms of the quasi $R$-matrix.

\medskip

\noindent{\bf Theorem G.} [Corollary \ref{cor:muiTheta}] 
   \textit{For all $i\in I\setminus X$ the following relations hold:
  \begin{align*}
    (\mu_i^L\ot \id)(\Theta) &= c_i M_i\cdot\Theta,&
    (\mu_i^R\ot \id)(\Theta) &=c_{\tau(i)} \Theta\cdot (\sigma \ot \sigma)(M_i).
  \end{align*}
  }

With these relations, the intertwiner property \eqref{eq:ThetaB-intertwiner-intro} for $\Theta^B$ can be read off directly from the intertwiner property
\begin{align}\label{eq:ThetaF}
 (F_i\ot K_i^{-1}+1 \ot F_i) \cdot \Theta = \Theta \cdot (F_i\ot K_i +1 \ot F_i)
\end{align}
for the quasi $R$-matrix, by translating the ordinary multiplication in the first tensor factor of \eqref{eq:ThetaF} into the star product multiplication.

At the very end of Section \ref{sec:RtoK} we explain how to extend our construction of the tensor quasi $K$-matrix to non-standard quantum symmetric pairs, involving two families of parameters $\bc=(c_i)_{i\in I\setminus X}$ and $\bs=(s_i)_{i\in I\setminus X}$, see \cite[Definition 5.6]{a-Kolb14}. This follows a line of argument first developed in \cite[3.5]{a-DobKol19} and also used in \cite{a-AppelVlaar22}.

Figure \ref{fig:implications} lays out the logical dependencies between the results proved in this paper. It illustrates the pivotal role played by the shortness property of the star product. 

\medskip

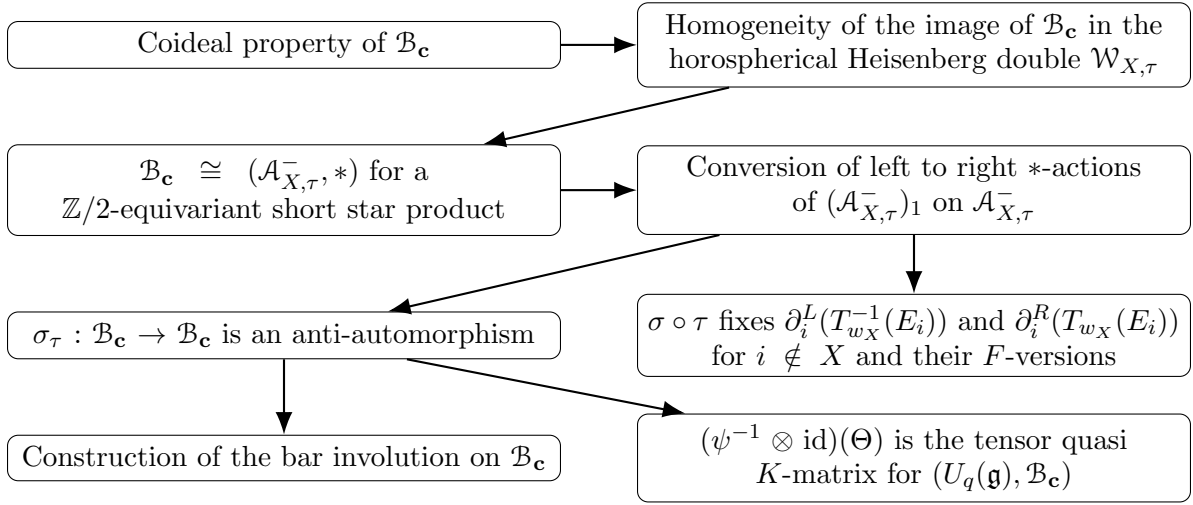
\begin{figure}
\centering

\begin{tikzpicture}[    node distance=1cm and 1cm,    box/.style={        draw,        rounded corners,        align=center,        font=\small,        text width=7cm,        fill=white    },    arrow/.style={        -{Latex[length=3mm]}, thick    }]
% Nodes
\node[box] (A) {Coideal property of $\cB_\bc$};\node[box, right=of A] (B) {Homogeneity of the image of $\cB_\bc$ in the horospherical Heisenberg double $\cW_{X,\tau}$};\node[box, below=of A] (C) {$\cB_\bc \cong (\cA^-_{X,\tau}, *)$ for a \\ $\Z/2$-equivariant short star product};\node[box, right=of C] (D) {Conversion of left to right $*$-actions \\ of $(\cA^-_{X,\tau})_1$ on $\cA^-_{X,\tau}$};\node[box, below=of C] (E) {$\sigma_\tau : \cB_\bc \rightarrow \cB_\bc$ is an anti-automorphism};\node[box, right=of E] (F) {$\sigma \circ \tau$ fixes $\partial^L_i(T_{w_X}^{-1}(E_i))$ and 
$\partial^R_i(T_{w_X}(E_i))$ for $i \notin X$ and their $F$-versions};\node[box, below=of E] (G) {Construction of the bar involution on $\cB_\bc$};\node[box, right=of G] (H) {$(\psi^{-1} \otimes \mathrm{id})(\Theta)$ is the tensor quasi \\$K$-matrix for $(\uqg,\cB_\bc)$};
% Arrows
\draw[arrow] (A) -- (B);
\draw[arrow] (B) -- (C);
\draw[arrow] (C) -- (D);
\draw[arrow] (D) -- (E);
\draw[arrow] (D) -- (F);
\draw[arrow] (E) -- (G);
\draw[arrow] (E) -- (H);
\end{tikzpicture}

\caption{Implications between main results}
\label{fig:implications}
\end{figure}

%%%%%%%%%%%%%%%%%%%%%%%%%%%%%%%%%%%%%%
\subsection{Organization}
%%%%%%%%%%%%%%%%%%%%%%%%%%%%%%%%%%%%%%
The main body of this paper is organized in four sections. In Section \ref{sec:letzter} we fix notation for quantized enveloping algebras and quantum symmetric pairs. We spell out the properties of various versions $\cR_X^{\pm,r}$, $\cR_X^{\pm,l}$ of quantum parabolic nilradicals which are obtained as coinvariants for suitable coactions on quantum standard parabolics. The interplay between the different versions of quantum parabolic nilradicals will be crucial in later sections. We discuss the relation between QSP-subalgebras and quantum horospherical subalgebras and revisit the Letzter map from \cite{a-KY21}. We then define the horospherical Heisenberg double $\cW_{X,\tau}$ and prove Theorem A as a consequence of the coideal property of the QSP-subalgebra $\cB_\bc$.

In Section \ref{sec:shortstar}, we revisit the general concept of short star products for (non-commutative) graded algebras and prove Theorem B. The remaining two sections are devoted to applications. Section \ref{sec:sigmatau} begins with a detailed discussion of the lower order term maps $\mu^L_i$ and $\mu^R_i$ for $i\in I\setminus X$, which eventually leads to a proof of Theorem C. As a consequence, we obtain the algebra anti-automorphisms $\sigma\circ \tau:(\cA_{X,\tau}^-,\ast)\rightarrow (\cA_{X,\tau}^-,\ast)$ and $\sigma_\tau:\cB_\bc\rightarrow \cB_\bc$, hence proving Theorem D. In Section \ref{sec:Bbar} we will revisit the bar involution for $\cB_\bc$, hence proving Theorem E, and in Section \ref{sec:noteworthy} we will give the elementary proof the the fundamental lemma for QSP.

The final Section \ref{sec:quasiK} is devoted to the discussion of the tensor quasi $K$-matrix. We first revisit relative versions of the skew Hopf pairing between the negative and the positive quantum Borel subalgebra. This leads to a non-degenerate pairing between the quantum nilradicals $\cR_X^{-,r}$ and $\cR_X^{+,l}$ in Proposition \ref{prop:rhrh}. This pairing allows us to give an alternative description of the lower order term maps $\mu_i^L$ in Theorem \ref{thm:muiLa-new}, which leads to a proof of Theorem G. Finally, Theorem F will be obtained in Theorem \ref{thm:intertwiner2} as a consequence of Theorem G and the intertwiner property of the quasi $R$-matrix.

\medskip

\noindent{\bf Acknowledgements.} The authors are grateful to Maarten van Pruijssen for valuable comments on horospherical subgroups and varieties. This project started out when S.K.~was visiting Radboud Universiteit Nijmegen in April 2024 and progressed steeply during a three months visit to MPIM Bonn in spring 2025. S.K.~is grateful to both institutions for the hospitality and the excellent environment. The research of M.Y. was supported by the Bulgarian Science Fund grant KP-06-N92/5 and the Ministry of Education and Science grant DO1-239/10.12.2024, the Simons Foundation grant SFI-MPS-T-Institutes-00007697 and NSF grant DMS–2200762.
%%%%%%%%%%%%%%%%%%%%%%%%%%%%%%%%%%%%%%
\section{The Letzter map}\label{sec:letzter}
%%%%%%%%%%%%%%%%%%%%%%%%%%%%%%%%%%%%%%
We introduce background and notation for quantum groups $\uqg$ and quantum symmetric pairs $(\uqg, \cB_\bc)$ and recall the Letzter map. We will see that the Letzter map factors over the horospherical Heisenberg double $\cW_{X,\tau}$. The algebra $\cW_{X,\tau}$ has a natural vector space grading. As a first main result we will show in Theorem \ref{thm:Bc-graded} that the image of $\cB_\bc$ in $\cW_{X,\tau}$ is a graded subspace. This fact will be a crucial ingredient when proving the shortness of the star product in Section \ref{sec:short-star} 
%%%%%%%%%%%%%%%%%%%%%%%%%%%%%%%%%%%%%%
\subsection{Background on quantized enveloping algebras}\label{sec:background}
Throughout the paper we fix a symmetrizable Kac--Moody algebra $\gfrak$ and denote by $\gfrak'=[\gfrak,\gfrak]$ its derived subalgebra. Let $(a_{ij})_{i,j\in I}$ be the generalized Cartan matrix of $\gfrak$, where $I$ is a finite set, and let $\{d_i\,|\,i\in I\}$ be a collection of relatively prime positive integers such that the matrix $(d_ia_{ij})$ is symmetric. Denote by $\Pi=\{\alpha_i\,|\,i\in I\}$ and $Q=\Z\Pi$ the set of simple roots for $\gfrak$ and its root lattice, respectively. Set $\N=\{0,1,2,\dots\}$, and let $Q^+=\N\Pi$ be the positive cone of $Q$. For any subset $X \subseteq I$ we write 
\[
Q_X=\sum_{j\in X} \Z \alpha_j \quad \mbox{and} \quad Q_X^+=\sum_{j\in X}\N \alpha_j.
\]
Let $(\cdot,\cdot):Q\times Q \rightarrow \Z$ be the symmetric bilinear form defined by $(\alpha_i,\alpha_j)=d_ia_{ij}$ for all $i,j\in I$. Denote by $W$ the Weyl group of $\gfrak$ with simple reflections $s_i$ for $i\in I$. 

Throughout the paper we fix a field $k$ of characteristic 0 and denote $\qfield=k(q)$. The quantized enveloping algebra $U=U_q(\gfrak')$ is the $\qfield$-algebra with generators $E_i, F_i, K_i^{\pm 1}$ for $i \in I$ and defining relations given in \cite[Section 3.1]{b-Lusztig94}. It is a Hopf algebra with coproduct $\kow$ given by
 \begin{align}\label{eq:kow-def}
    \kow(E_i){=}E_i\ot 1{+} K_i\ot E_i, \; \kow(F_i){=}F_i\ot K_i^{-1}{+}1 \ot F_i, \; \kow(K_i){=}K_i\ot K_i \;\;  
\mbox{for all $i \in I$,}
 \end{align}
antipode $S$ given by 
\begin{equation}
\label{eq:antipode}
S(K_i)= K_i^{-1}, \; \; S(E_i)=-K_i^{-1}E_i, \; \; 
S(F_i) =-F_i K_i \; \;  
\mbox{for all $i \in I$}
\end{equation}
and counit $\vep$. Using Sweedler notation for the coproduct we write $\kow(u)=u_{(1)}\ot u_{(2)}$ for $u\in \Uq$.
For $\beta=\sum_{i\in I}n_i \alpha_i\in Q$ we write $K_\beta=\prod_{i\in I}K_i^{n_i}$. Set $q_i = q^{d_i}$ for all $i \in I$. 

We will make frequent use of the following involutive (anti-)automorphisms of $\Uq$, see \cite[3.1.3, 3.1.12]{b-Lusztig94}. 
\begin{enumerate}
  \item The Chevalley involution is the $\qfield$-algebra automorphism  $\omega:\Uq\rightarrow\Uq$ defined by
  \begin{align}\label{eq:omega-def}
    \omega(E_i)=-F_i, \quad \omega(F_i)=-E_i, \quad \omega(K_i)=K_i^{-1} \qquad\mbox{for all $i\in I$.}
  \end{align}
    The map $\omega$ is a coalgebra anti-automorphism of $\Uq$, i.e.~$\kow(\omega(u))=\omega(u_{(2)})\ot\omega(u_{(1)})$ for all $u\in U$. Our definition of $\omega$ differs from \cite[3.1.3]{b-Lusztig94} by a sign.
    \item The map $\sigma:\Uq\rightarrow \Uq$ is the $\qfield$-algebra anti-automorphism defined by
    \begin{align}\label{eq:sigma-def}
      \sigma(E_i)=E_i, \quad \sigma(F_i)=F_i, \quad \sigma(K_i)=K_i^{-1} \qquad\mbox{for all $i\in I$.}
    \end{align}
Here, saying that $\sigma$ is a $\qfield$-algebra anti-automorphism means that $\sigma$ is a $\qfield$-linear automorphism such that $\sigma(u'u)=\sigma(u)\sigma(u')$ for all $u,u'\in \Uq$.  
  \item The bar involution $\obar:\Uq\rightarrow\Uq$, $u\mapsto \overline{u}$, is the $\field$-algebra automorphism defined by
    \begin{align}\label{eq:obar-def}
    \overline{E_i}=E_i, \quad \overline{F_i}=F_i, \quad \overline{K_i}=K_i^{-1}, \quad \overline{f(q)}=f(q^{-1}) \qquad\mbox{for all $i\in I, f(q) \in k(q)$.}
    \end{align}
 \end{enumerate}   
Observe that $\sigma$ and the bar involution $\obar:\Uq \rightarrow \Uq$ satisfy the following relation
\begin{align}\label{eq:sigmakow}
  (\sigma \ot \sigma)\circ \kow(u)= (\obar \ot \obar)\circ\kow(\sigma(\overline{u}))\qquad \mbox{for all $u\in \Uq$.}
\end{align}
Indeed, this relation is checked on the generators $E_i, F_i, K_i$ and $q$, and then follows from the fact that the maps on both sides of the equation are algebra anti-homomorphisms.

Denote by $U^\pm$ the unital subalgebras of $U$ generated by $\{E_i\,|\,i\in I\}$ and $\{F_i\,|\,i\in I\}$, respectively. Let $U^0$ be the unital subalgebra of $U$ generated by $\{K_i^{\pm 1}\,|\,i\in I\}$; it is a Laurent polynomial ring in those generators. The algebras $U$ and $U^\pm$ are $Q$-graded by setting $\deg(E_i) = \alpha_i$, $\deg(F_i) = -\alpha_i$ and $\deg (K_i^{\pm1})=0$ for all $i \in I$. The corresponding graded components will be denoted by $U_\mu$ and $U^\pm_{\mu}$ for $\mu \in Q$.  

For $i\in I$, let $T_i$ be the algebra automorphism of $U$ denoted by $T''_{i,1}$ in \cite[Section 37.1]{b-Lusztig94}. These automorphisms define an action of the braid group of $W$ on $U$ by setting 
$T_{s_{i_1} \ldots s_{i_l}} = T_{s_{i_1}} \ldots T_{s_{i_l}}$ for all reduced words $s_{i_1} \ldots s_{i_l} \in W$. 

 We recall the  Lusztig--Kashiwara skew derivations $\partial_i^L,\partial_i^R: U^\pm \rightarrow U^\pm$ for $i\in I$ which can be defined by
\begin{align}
  E_i y - y E_i &= \frac{K_i \partial_i^L(y)-\partial_i^R(y)K_i^{-1}}{q_i-q_i^{-1}},\label{eq:EyyE}\\
  F_i x - x F_i &= \frac{K_i^{-1} \partial_i^L(x)-\partial_i^R(x)K_i}{q_i-q_i^{-1}}\label{eq:FxxF}
\end{align}
for $x\in U^+$ and $y\in U^-$,  see \cite[Proposition 3.1.6]{b-Lusztig94}. Up to a sign, the maps $\partial_i^{L,R}:U^-\rightarrow U^-$ and $\partial_i^{L,R}:U^+\rightarrow U^+$ are related by conjugation by the Chevalley involution $\omega:U^\pm\rightarrow U^\mp$ given by \eqref{eq:omega-def}, that is
\begin{align*}
  \omega(\partial_i^{L,R}(x))=-\partial_i^{L,R}(\omega(x)) \qquad \mbox{for all $x\in U^+$}.
\end{align*}
The maps $\partial^R_i,\partial^L_i:U^-\rightarrow U^-$ satisfy the relations
$\partial^R_i(F_j)=\delta_{ij}=\partial^L_i(F_j)$ for all $i,j\in I$. Moreover, Equation \eqref{eq:EyyE} implies that
\begin{equation}\label{eq:skew-der}
\begin{aligned}
  \partial_i^L(f_\mu f_\nu)=\partial^L_i(f_\mu)f_\nu + q^{(\alpha_i,\mu)}f_\mu \partial^L_i(f_\nu),\\
  \partial_i^R(f_\mu f_\nu) = q^{(\alpha_i,\nu)}\partial_i^R(f_\mu) f_\nu + f_\mu \partial_i^R(f_\nu)
\end{aligned}
\end{equation}
for all $f_\mu\in U^-_{-\mu}$, $f_\nu\in U^-_{-\nu}$. Similarly, the maps $\partial^R_i,\partial^L_i:U^+\rightarrow U^+$ satisfy the relations
$\partial^R_i(E_j)=\delta_{ij}=\partial^L_i(E_j)$ for all $i,j\in I$ and Equation \eqref{eq:FxxF} implies that
\begin{equation}\label{eq:skew-der-+}
\begin{aligned}
  \partial_i^L(e_\mu e_\nu)=\partial^L_i(e_\mu)e_\nu + q^{(\alpha_i,\mu)}e_\mu \partial^L_i(e_\nu),\\
  \partial_i^R(e_\mu e_\nu) = q^{(\alpha_i,\nu)}\partial_i^R(e_\mu) e_\nu + e_\mu \partial_i^R(e_\nu)
\end{aligned}
\end{equation}
for all $e_\mu\in U^+_{\mu}$, $e_\nu\in U^+_{\nu}$. The skew derivation properties \eqref{eq:skew-der} and \eqref{eq:skew-der-+} determine the maps $\partial^R_i,\partial^L_i:U^\pm\rightarrow U^\pm$ uniquely from their actions on the generators of $U^\pm$. Hence, the relations
\eqref{eq:skew-der-+} together with the definition of the coproduct \eqref{eq:kow-def} imply that
  \begin{align}
    \kow(x)&=K_\mu\ot x + E_i K_{\mu-\alpha_i}\ot \partial_i^L(x) + \mbox{L.l.o.t}, \label{eq:cop-partial-L}\\
    \kow(x)&=x\ot 1 + \partial_i^R(x)K_i \ot E_i + \mbox{R.l.o.t}\label{eq:cop-partial-R}
\end{align}
for $x\in U^+_\mu$ where $\mbox{L.l.o.t}\in  \sum_{\nu\neq 0,\alpha_i} U^+_\nu K_{\mu-\nu}\ot U^+_{\mu-\nu}$ and $\mbox{R.l.o.t}\in  \sum_{\nu\neq 0,\alpha_i} U^+_{\mu-\nu} K_{\nu}\ot U^+_{\nu}$. Applying $\omega\ot \omega$ to \eqref{eq:cop-partial-L} and \eqref{eq:cop-partial-R} we obtain
\begin{align}
  \kow(y) &= y\ot K_{-\mu} + \partial_i^L(y)\ot F_i K_{-(\mu-\alpha_i)} +  \mbox{L.l.o.t} \label{eq:cop-partial2-L},\\
   \kow(y) &= 1\ot y + F_i\ot \partial_i^R(y) K_i^{-1}  +  \mbox{R.l.o.t} \label{eq:cop-partial2-R}
\end{align}  
for $y\in U^-_{-\mu}$ where $\mbox{L.l.o.t}\in  \sum_{\nu\neq 0,\alpha_i} U^-_{-(\mu-\nu)}\ot U^-_{-\nu} K_{-(\mu-\nu)}$ and $\mbox{R.l.o.t}\in  \sum_{\nu\neq 0,\alpha_i} U^-_{-\nu}\ot U^-_{-(\mu-\nu)} K_{-\nu}^{-1}$.

In the sequel, we will occasionally work with completions $\scrU$ and $\scrU^{(2)}$ of $\Uq$ and $\Uq\ot \Uq$, respectively,
defined as the algebras of functorial endomorphisms of canonical fiber functors. We refer the reader to 
\cite[Section 3]{a-BalaKolb19} for details. In brief, let $\Oint$ denote the category of integrable $\Uq$-modules in category $\cO$, let $\mathcal{Vect}$ be the category of $\qfield$-vector spaces and let $\mathcal{For}:\Oint\rightarrow \mathcal{Vect}$ be the forgetful functor. Then
\begin{align}\label{eq:scrU}
   \scrU = \End(\mathcal{For}) 
\end{align}
is the algebra of natural transformations of $\mathcal{For}$ to itself. The action of $\Uq$ on objects in $\Oint$ allows us to consider $\Uq$ as a subalgebra of $\scrU$. The algebra $\scrU$ contains a homomorphic image of the braid group corresponding to $W$, realized by Lusztig's braid group action on modules. More generally, for any $n\ge 1$ we define
\begin{align}\label{eq:scrUn}
    \scrU^{(n)} = \End(\mathcal{For}^{(n)}) 
\end{align}
where $\mathcal{For}^{(n)}:\Oint^n\rightarrow \mathcal{Vect}$ is the composition of $(n-1)$-fold tensor products and the forgetful functor. Again, $\scrU^{(n)}$ is an algebra with multiplication $\cdot$ given by composition of natural transformations. Any natural transformation $\varphi\in \scrU$ can be restricted to all $\mathcal{For}(M\ot N)$ for $M,N\in \Oint$. In this way we obtain an algebra homomorphism $\kow:\scrU\rightarrow \scrU^{(2)}$ which extends the coproduct on $\Uq$. Similarly, one obtains algebra homomorphisms $\kow\ot \id, \id\ot \kow:\scrU^{(2)}\rightarrow \scrU^{(3)}$, see \cite[Section 3]{a-BalaKolb19}.
%%%%%%%%%%%%%%%%%%%%%%%%%%%%%%%%%%%%%%
\subsection{Radford's biproduct for parabolic subalgebras}
%%%%%%%%%%%%%%%%%%%%%%%%%%%%%%%%%%%%%%
For any proper subset $X\subset I$ let $\cP_X^\pm$ be the corresponding quantum parabolic subalgebras of $\uqg$. More precisely, we set
\begin{align*}
  \cP^+_X&=\qfield\langle  E_i, K_i^{\pm 1}, F_j\,|\,i\in I, j\in X\rangle,\\
  \cP^-_X&=\qfield\langle F_i, K_i^{\pm 1}, E_j\,|\,i\in I, j\in X\rangle 
\end{align*}  
and let
\begin{align}\label{eq:LX-def}
  \cL_X=\qfield\langle  K_i^{\pm 1}, E_j, F_j\,|\,i\in I, j\in X\rangle
\end{align}
be the corresponding Levi factor. Let $U^+_X$ and $U^-_X$ be the unital subalgebras of $\Uq$ generated by $\{E_j\,|\,j\in X\}$ and $\{F_j\,|\,j\in X\}$, respectively. Then $\cL_X$ has a triangular decomposition
\begin{align}\label{eq:triang_LX}
  U_X^-\ot U^0\ot U^+_X\cong\cL_X.
\end{align}  
We need left and right versions of Radford's construction in \cite{a-Radford85} applied to both $\cP^+_X$ and $\cP^-_X$. The left version coincides with \cite{a-Radford85}. The right version was also spelled out in \cite{a-KY21} for $\cP^-_X$. 

Consider the surjective Hopf algebra homomorphisms $\pi_X^\pm:\cP_X^\pm\rightarrow \cL_X$ defined by $\pi_X^\pm|_{\cL_X}=\id_{\cL_X}$ and $\pi_X^+(E_i)=0$, $\pi_X^-(F_i)=0$ for all $i\in I\setminus X$. We define right and left coactions
\begin{align*}
  \kow^{\pm,r}_{\cL_X}&=(\id \ot \pi^\pm_X)\circ \kow:\cP_X^\pm\rightarrow \cP_X^\pm\ot \cL_X,\\
  \kow^{\pm,l}_{\cL_X}&=(\pi^\pm_X\ot \id)\circ \kow:\cP_X^\pm\rightarrow \cL_X\ot \cP_X^\pm.
\end{align*}
Consider the corresponding subalgebras of coinvariants
\begin{align*}
   \cR_X^{\pm,r}=\{a\in \cP_X^\pm\,|\,\kow^{\pm,l}_{\cL_X}(a)=1\ot a\},\quad
  \cR_X^{\pm,l}=\{a\in \cP_X^\pm\,|\,\kow^{\pm,r}_{\cL_X}(a)=a\ot 1\},
\end{align*}
where the notational swap of right and left is intentional. 
We collect known properties of $\cR_X^{\pm,r}$ and $\cR_X^{\pm,l}$. To emphasize that $\cL_X$ is the underlying Hopf algebra, we write $H=\cL_X$.
%%%%%%%%%%%%%%%%%%%%%%%%%%%%%%%%%%%%%%%%%%%%%%%%
\begin{prop}[{\cite[Theorem 3]{a-Radford85}}]
  The algebra $\cR_X^{\pm, l}$ is a left $H$-module under the left adjoint action, and a left $H$-comodule under $\kow^{\pm,l}_H$. With these two structures, $\cR_X^{\pm,l}$ is a Hopf algebra in the category of left Yetter-Drinfeld modules over $H$. We can form Radford's biproduct $\cR_X^{\pm, l}\times H$ which coincides with $\cR_X^{\pm, l}\otimes H$ as a vector space. The multiplication map $\cR_X^{\pm, l}\times H\rightarrow \cP_X^\pm$ is an isomorphism of Hopf algebras.
\end{prop}  
%%%%%%%%%%%%%%%%%%%%%%%%%%%%%%%%%%%%%%%%%%%%%%%%
Similarly, $\cR_X^{\pm, r}$ is a right $H$-module under the right adjoint action, and a right $H$-comodule under $\kow^{\pm, r}_H$. One can form a right version of Radford's biproduct, and the multiplication map $H\otimes \cR_X^{\pm,r}\rightarrow \cP_X^{\pm}$ is an isomorphism.

The four algebras $\cR_X^{\pm,s}$ for $s\in \{r,l\}$ are closely related. Recall the antipode $S$ defined by \eqref{eq:antipode} and the Chevalley involution $\omega$ defined by \eqref{eq:omega-def}.
%%%%%%%%%%%%%%%%%%%%%%%%%%%%%%55
\begin{lem}\label{lem:omS}
  The following relations hold:
  \begin{align}
    \omega(\cR_X^{+,l})&=\cR_X^{-,r}, &  \omega(\cR_X^{+,r})&=\cR_X^{-,l},\label{eq:omegaR}\\
     S(\cR_X^{+,l})&=\cR_X^{+,r}, &  S(\cR_X^{-,r})&=\cR_X^{-,l}.\label{eq:SR}
  \end{align}  
\end{lem}
%%%%%%%%%%%%%%%%%%%%%%%%%%%%%%%%%
\begin{proof}
  Relations \eqref{eq:omegaR} follow from the fact that $\omega$ restricts to coalgebra anti-isomorphisms $\omega:\cP_X^\pm\rightarrow \cP_X^\mp$ such that $\omega\circ \pi^\pm_X=\pi_X^\mp\circ \omega$.
  Relations \eqref{eq:SR} follow from the fact that $S$ restricts to a coalgebra anti-automorphism $S:\cP_X^\pm\rightarrow \cP_X^\pm$ which commutes with $\pi_X^\pm$.
\end{proof}  
%%%%%%%%%%%%%%%%%%%%%%%%%%%%%%%%%
The algebras $\cR_X^{\pm, s}$ for $s\in \{r,l\}$ have several alternative descriptions. Let $G^+$ and $G^-$ be the unital subalgebras of $\Uq$ generated by the elements of the sets $\{E_iK_i^{-1}\,|\,i\in I\}$ and $\{F_iK_i\,|\,i\in I\}$, respectively; they are given by
\[
G^\pm = S(U^{\pm}),
\]
cf. \eqref{eq:antipode}.
The following corollary is proved for $\cR_X^{-,r}$ in \cite[Corollary 2.3]{a-KY21}. The other cases are proved analogously or by application of Lemma \ref{lem:omS}.
%%%%%%%%%%%%%%%%%%%%%%%%%%%%%%%%%%%%%%%%%%%%%%%%
\begin{cor}\label{cor:adH-gen}
  The following hold:
  \begin{enumerate}
  \item The algebra $\cR_X^{+,l}$ is generated by the union of subspaces $\bigcup_{i\in I\setminus X}\ad_l(H)(E_i)$.
  \item The algebra $\cR_X^{+,r}$ is generated by the union of subspaces $\bigcup_{i\in I\setminus X}\ad_r(H)(E_iK_i^{-1})$.
  \item The algebra $\cR_X^{-,l}$ is generated by the union of subspaces $\bigcup_{i\in I\setminus X}\ad_l(H)(F_iK_i)$.
  \item The algebra $\cR_X^{-,r}$ is generated by the union of subspaces $\bigcup_{i\in I\setminus X}\ad_r(H)(F_i)$.
  \end{enumerate}
  In particular, we have  $\cR_X^{+,l}\subset U^+$,  $\cR_X^{+,r}\subset G^+$,  $\cR_X^{-,l}\subset G^-$, $\cR_X^{-,r}\subset U^-$.
\end{cor}
%%%%%%%%%%%%%%%%%%%%%%%%%%%%%%%%%%%%%%%%%%%%%%%%
\begin{proof}
  To verify the final statement of the corollary, note that $\ad_l(F_j)(E_i)=0$ for all $j\in X$, $i\in I\setminus X$. Hence $E_i$ is a lowest weight vector for the left adjoint action of $H=\cL_X$  and therefore
  \begin{align*}
     \ad_l(H)(E_i)=\ad_l(U^+_X)(E_i)\subseteq U^+.
  \end{align*}
  The other inclusions follow from the relations
  \begin{align*}
     \ad_r(F_j)(E_iK_i^{-1})=0, \quad\ad_l(E_j)(F_iK_i)=0,\quad \ad_r(E_j)(F_i)=0 \qquad \mbox{for $j\neq i$}
  \end{align*}
  in a similar way.
\end{proof}  
%%%%%%%%%%%%%%%%%%%%%%%%%%%%%%%%%%%%%%%%%%%%%%%%
The algebra $\cR_X^{+,l}$ is $\N$-graded via the degree function given by
\begin{align}\label{eq:deg-R}
  \deg(\ad_l(h)(E_i))=1 \qquad \mbox{for all $i\in I\setminus X, h\in H$.}
\end{align}
We define similar gradings on $\cR_X^{+,r}$, $\cR_X^{-,r}$, $\cR_X^{-,l}$ by requesting the linear isomorphisms $\omega$ and $S$ in Lemma \ref{lem:omS} to be graded. For any $n\in \N$ and $s\in \{r,l\}$ let $\cR_{X,n}^{\pm,s}$ denote the homogeneous component of degree $n$.

The equation
      \begin{align}\label{eq:kow-adhE}
         \kow(\ad_l(h)(E_i))=\ad_l(h)(E_i)\ot 1 + h_{(1)}K_i S(h_{(3)})\ot \ad_l(h_{(2)})(E_i)
      \end{align}
      for $h\in H$ and $i\in I\setminus X$ implies that $\kow(\cR^{+,l}_X)\subset U^+ U^0 \ot \cR_X^{+,l}$. Hence $\cR^{+,l}_X$ is a left coideal subalgebra of the Hopf algebra $U^+U^0$. Similarly, $\cR_X^{-,l}$ is a left coideal subalgebra of $U^- U^0$, and $\cR_X^{+,r}\subset \Uq^+ U^0$ and $\cR^{-,r}_X\subset \Uq^-U^0$ are right coideal subalgebras.

      The subalgebras $\cR_X^{\pm, s}$ for $s\in \{r,l\}$ can also be described in terms of Lusztig's braid group automorphisms, if $X$ is of finite type. In this case let $w_X\in W$ be the longest element in the parabolic subgroup $W_X\subset W$.
%%%%%%%%%%%%%%%%%%%%%%%%%%%%%%%%%%%%%%%%%%%%%%%%
  \begin{prop}\label{prop:R-UTU}
    Assume that $X\subset I$ is of finite type. The following hold:
    \begin{align}
      \cR_X^{-,r}&=U^-\cap T_{w_X}(U^-), & \cR_X^{-,l}&=G^-\cap T^{-1}_{w_X}(G^-),\label{eq:UTU1}\\
      \cR_X^{+,l}&=U^+\cap T_{w_X}(U^+), &\cR_X^{+,r}&=G^+\cap T^{-1}_{w_X}(G^+).\label{eq:UTU2}
    \end{align}  
  \end{prop}  
%%%%%%%%%%%%%%%%%%%%%%%%%%%%%%%%%%%%%%%%%%%%%%%%
  \begin{proof}
    The first relation in \eqref{eq:UTU1} is proved in \cite[Corollary 2.7]{a-KY21}. The first relation in \eqref{eq:UTU2} follows by application of $\omega$ to the first relation in \eqref{eq:UTU1}, Lemma \ref{lem:omS} and the relations in \cite[37.2.4]{b-Lusztig94}. Finally, the second relations in both \eqref{eq:UTU1} and \eqref{eq:UTU2} follow from the first relations by application of the antipode $S$ using Lemma \ref{lem:omS} and the fact that $S(T_{w_X}(U^\pm))=T^{-1}_{w_X}(S(U^\pm))=T^{-1}_{w_X}(G^\pm)$.
  \end{proof}  
%%%%%%%%%%%%%%%%%%%%%%%%%%%%%%%%%%%%%%%%%%%%%%%%
  The subalgebras $\cR_X^{+,l}\subset U^+$ and $\cR_X^{-,r}\subset U^-$ can also be described in terms of the Lusztig-Kashiwara skew-derivations $\partial_j^R$ on $U^+$ and $U^-$, respectively. The second formula of the following lemma was proved in \cite[Lemma 2.4, (2.10)]{a-KY21} and the first formula follows by application of $\omega$.
%%%%%%%%%%%%%%%%%%%%%%%%%%%%%%%%%%%%%%%%%%%%%%%%
  \begin{lem}
    The following relations hold:
    \begin{align*}
        \cR_X^{+,l}&=\{u\in U^+\,|\,\partial_j^R(u)=0\quad \mbox{for all $j\in X$}\},\\ 
        \cR_X^{-,r}&=\{u\in U^-\,|\,\partial_j^R(u)=0\quad \mbox{for all $j\in X$}\}.
    \end{align*}    
  \end{lem}  
%%%%%%%%%%%%%%%%%%%%%%%%%%%%%%%%%%%%%%%%%%%%%%%%  
  The triangular decomposition of $\Uq$ leads to a triangular decomposition for $\cP^\pm_X$. Combining these triangular decompositions, one obtains that the multiplication map gives a linear isomorphism
  \begin{align}\label{eq:triang_PX}
     \cR_X^{-,s}\ot H \ot \cR_X^{+,t}\cong \Uq \qquad \mbox{for all $s,t\in \{r,l\}$.}
  \end{align}
  In particular, the multiplication map gives linear isomorphisms
  \begin{align}\label{eq:triang_Upm}
    \cR_X^{-,r}\ot U^-_X\cong U^-, \qquad \cR_X^{+,l}\ot U^+_X \cong U^+.
  \end{align}
%%%%%%%%%%%%%%%%%%%%%%%%%%%%%%%%%%%%%%%%%%%%%%%
  \subsection{Quantum symmetric pairs}\label{sec:QSP}
%%%%%%%%%%%%%%%%%%%%%%%%%%%%%%%%%%%%%%%%%%%%%%%
  We recall the construction of quantum symmetric pairs, as originally defined for finite dimensional $\gfrak$ by G.~Letzter \cite{a-Letzter99a} and as extended to the Kac-Moody case in \cite{a-Kolb14}. We will work in the slightly more general setting of generalized Satake diagrams proposed in \cite{a-RV20}. Recall that a generalized Satake diagram is a pair $(X,\tau)$ consisting of a subset $X\subset I$ of finite type and a diagram automorphism $\tau:I\rightarrow I$ satisfying the following three properties:
  \begin{enumerate}
    \item $\tau^2=\id_I$;
    \item\label{eq:GS2} $\alpha_{\tau(j)}=-w_X(\alpha_j)$ for all $j\in X$, in particular $\tau(X)=X$;
    \item for $i\in I\setminus X$ and $j\in X$ the following implication holds:  
\begin{align}\label{eq:generalizedSatake}
  \tau(i)=i \mbox{ and } a_{ji}=-1 \quad \Longrightarrow \quad w_X(\alpha_i)\neq \alpha_i+\alpha_j,
\end{align}
\end{enumerate}
where $w_X$, as before, denotes the longest element in the parabolic subgroup $W_X$. Let $\theta:Q\rightarrow Q$ be the involutive automorphism defined by $\theta=-w_X\circ \tau$ and set $Q^\theta=\{\beta\in Q\,|\,\theta(\beta)=\beta\}$. By property \eqref{eq:GS2} of the definition of a generalized Satake diagram we have $\theta(\alpha_j)=\alpha_j$ for all $j\in X$.

We associate several subalgebras of $\Uq$ to the generalized Satake diagram $(X,\tau)$. Define $U^0_\theta=\qfield \langle K_\beta\,|\,\beta\in Q^\theta\rangle$. Since $Q^\theta$ is generated by the elements of the set $\{\alpha_i-\alpha_{\tau(i)},\alpha_j\,|\, i\in I\setminus X, j\in X\}$, we have that
\begin{align}\label{eq:U0theta}
  U^0_\theta=\qfield\langle K_j^{\pm 1}, K_i K_{\tau(i)}^{-1}\,|\,j\in X, i\in I\setminus X\rangle.
\end{align}
Define
\begin{align*}
  \cH_{X,\tau}=\qfield\langle F_j, E_j, K_\beta\,|\, j \in X, \beta\in Q^\theta\rangle.
\end{align*}
Let $\cH_{X,\tau}^\ge$ be the subalgebra of $\cH_{X,\tau}$ generated by $\{E_j,K_\beta\,|\,j\in X, \beta\in Q^\theta\}$.
Let $\bc=(c_i)_{i\in I\setminus X}\in (\qfield^\ast)^{I\setminus X}$ be a set of parameters satisfying the condition
\begin{align*}
   c_i=c_{\tau(i)} \quad \mbox{if} \quad  (\alpha_i,\theta(\alpha_i))=0.
\end{align*}
Let $\cB_\bc\subset \Uq$ be the subalgebra generated by $\cH_{X,\tau}$ and the elements
\begin{align}\label{eq:Bi-def}
  B_i=F_i - c_i T_{w_X}(E_{\tau(i)}) K_i^{-1} \qquad \mbox{for all $i\in I\setminus X$.}
\end{align}
We refer to $\cB_\bc$ as the (standard) \textit{quantum symmetric pair coideal subalgebra} (or QSP-subalgebra) corresponding to the generalized Satake diagram $(X,\tau)$ and the set of parameters $\bc$. The QSP-subalgebra $\cB_\bc$ is a right coideal subalgebra of $\Uq$, that is
\begin{align}\label{eq:kowB}
  \kow(\cB_\bc)\subset \cB_\bc\ot \Uq,
\end{align}
see \cite[Corollary 4.2]{a-Letzter99a}, \cite[Theorem 7.2]{MSRI-Letzter}, \cite[Proposition 5.2]{a-Kolb14}. We reformulate the proof of the coideal property, as this allows us to introduce notation crucial in the discussion of the quasi $K$-matrix in Section \ref{sec:quasiK}.
Indeed, $\cH_{X,\tau}$ is a Hopf subalgebra of $\Uq$, and the generators $B_i$ defined by Equation \eqref{eq:Bi-def} satisfy the relation
\begin{align}\label{eq:kow-Bi}
  \kow(B_i)=B_i\ot K_i^{-1} + 1 \ot F_i - c_i(1\ot K_i^{-1}) M_i
\end{align}
where
\begin{align}\label{eq:Mi-def}
M_i = q^{-(\theta(\alpha_i),\alpha_i)}(K_i^{-1}\ot 1)\big(\kow(T_{w_X}(E_{\tau(i)}))- T_{w_X}(E_{\tau(i)})\ot 1\big)
\end{align}
for all $i\in I\setminus X$. The coideal property \eqref{eq:kowB} now follows from the first statement of the following lemma. Both properties of the lemma will be used in Section \ref{sec:quasiK}.
%%%%%%%%%%%%%%%%%%%%%%%%%%%%%%%%%%%%%
\begin{lem}\label{lem:Mi}
  For any $i\in I\setminus X$ the element $M_i$ defined by \eqref{eq:Mi-def} has the following properties:
  \begin{enumerate}
    \item $M_i\in \cH_{X,\tau}^\ge \ot \ad_l(\cH_{X,\tau})(E_{\tau(i)})$;
    \item $ (\id \ot \kow)(M_i)= M_i\ot 1 + (1\ot K_i\ot 1)(\kow \ot \id)(M_i)$.
  \end{enumerate}  
\end{lem}
%%%%%%%%%%%%%%%%%%%%%%%%%%%%%%%%%%%%%
\begin{proof} (1) By Corollary \ref{cor:adH-gen}.(1) and Equation \eqref{eq:UTU1} we have $T_{w_X}(E_{\tau(i)})=\ad_l(h)(E_{\tau(i)})$ for some $h\in U_X^+$. Hence the relation $M_i\in \cH_{X,\tau}^\ge \ot \ad_l(\cH_{X,\tau})(E_{\tau(i)})$ follows from Equation \eqref{eq:kow-adhE}.\\
  (2) Using Equation \eqref{eq:Mi-def} and coassociativity we obtain
  \begin{align}
    (\id\ot \kow)(M_i)=&q^{-(\theta(\alpha_i),\alpha_i)}(K_i^{-1}\ot 1\ot 1)\cdot\nonumber\\
    &\cdot\big((\kow \ot \id)\circ\kow(T_{w_X}(E_{\tau(i)}))- T_{w_X}(E_{\tau(i)})\ot 1\ot 1\big)\label{eq:idkowMi}
  \end{align}
  Using $\kow (T_{w_X}(E_{\tau(i)}))=T_{w_X}(E_{\tau(i)})\ot 1 +q^{(\theta(\alpha_i),\alpha_i)}(K_i\ot 1)M_i$ we obtain
  \begin{align*}
    (\kow \ot \id)\circ&\kow(T_{w_X}(E_{\tau(i)}))- T_{w_X}(E_{\tau(i)})\ot 1\ot 1=\\
    =&\big(\kow(T_{w_X}(E_{\tau(i)}))- T_{w_X}(E_{\tau(i)})\ot 1\big)\ot 1
     +q^{(\theta(\alpha_i),\alpha_i)}(K_i\ot K_i\ot 1)(\kow\ot \id)(M_i).
  \end{align*}
  Inserting the above expression into Equation \eqref{eq:idkowMi} and using Equation \eqref{eq:Mi-def} again, we obtain the desired formula.
\end{proof}
%%%%%%%%%%%%%%%%%%%%%%%%%%%%%%%%%%%%%
Nonstandard QSP-subalgebras were introduced in the finite case by Letzter \cite[Variation 2, after (7.25)]{MSRI-Letzter} and in the Kac-Moody case in \cite[Definition 5.6]{a-Kolb14}. Conceptually, nonstandard QSP-subalgebras are best understood via twisting of $\cB_\bc$ by characters. Let $\chi:\cB_\bc\rightarrow \qfield$ be a character, i.e.~a unital $\qfield$-algebra homomorphism, and consider
\begin{align*}
  \cB_{\bc,\chi}=\{\chi(b_{(1)})b_{(2)}\,|\,b\in \cB_\bc\}.
\end{align*}
By the coassociativity of the coproduct the subalgebra $\cB_{\bc,\chi}\subset \Uq$ is a right coideal subalgebra. Characters of $\cB_\bc$ and the coideal subalgebras $\cB_{\bc,\chi}$ have been discussed in \cite[Section 10]{a-Kolb23}. Let $\bs=(s_i)_{i\in I\setminus X}\in \qfield^{I\setminus X}$, and assume that there exists a character $\chi^\bc_\bs:\cB_\bc\rightarrow \qfield$ such that
\begin{align}\label{eq:chics}
  \chi^\bc_\bs|_{\cH_{X,\tau}}=\vep|_{\cH_{X,\tau}}, \qquad \chi^\bc_\bs(B_i)=s_i \qquad \mbox{for all $i\in I\setminus X$.}
\end{align}
In this case it follows from Equation \eqref{eq:kow-Bi} and Lemma \ref{lem:Mi}.(1) that $\cB_{\bc,\bs}:=\cB_{\bc,\chi^\bc_\bs}$ is the subalgebra of $\Uq$ generated by $\cH_{X,\tau}$ and the elements
\begin{align*}
  B_i^{\bc, \bs}=F_i- c_i T_{w_X}(E_{\tau(i)}) K_i^{-1} + s_i K_i^{-1}.
\end{align*}
Moreover, in this case, the map
\begin{align}\label{eq:rho-cs}
  \rho_\bs^\bc:\cB_\bc\rightarrow \cB_{\bc,\bs}, \qquad b\mapsto \chi^\bc_\bs(b_{(1)})b_{(2)}
\end{align}
is an isomorphism of right $\Uq$-comodule algebras such that
\begin{align*}
  \rho_\bs^\bc|_{\cH_{X,\tau}}=\id_{\cH_{X,\tau}}, \qquad \rho_{\bs}^\bc(B_i)=B_i^{\bc, \bs},
\end{align*}
see \cite[Proposition 10.5]{a-Kolb23}. We refer to the right coideal subalgebras $\cB_{\bc,\bs}$ with $\bs\neq \mathbf{0}$ obtained in this way as nonstandard QSP-subalgebras. By \cite[Proposition 10.6, Remark 10.7]{a-Kolb23} the class of nonstandard QSP-subalgebras obtained in this way is slightly larger than the class considered in \cite[Definition 5.6]{a-Kolb14}. However,  
via the isomorphism \eqref{eq:rho-cs} all desirable properties of $\cB_{\bc,\bs}$ from \cite{a-Kolb14} also hold for this slightly larger class.
%%%%%%%%%%%%%%%%%%%%%%%%%%%%%%%%%%%%%%%%%%%%%%%%%
\subsection{The quantum horospherical subalgebra}
%%%%%%%%%%%%%%%%%%%%%%%%%%%%%%%%%%%%%%%%%%%%%%%%%
As in Section \ref{sec:QSP} let $(X,\tau)$ be a generalized Satake diagram. Let $\cA_{X,\tau}^-$ be the subalgebra of $\Uq$ generated by $\cH_{X,\tau}$ and $U^-$. We call $\cA_{X,\tau}^-$ the \textit{quantum horospherical subalgebra} corresponding to $(X,\tau)$. The quantum horospherical subalgebra $\cA_{X,\tau}^-$ is $\N$-graded via the degree function given by
\begin{align*}
  \deg(F_i)&=1 \qquad \mbox{for all $i\in I\setminus X$,}\\
  \deg(h)&=0 \qquad \mbox{for all $h\in \cH_{X,\tau}$.}
\end{align*}
For any $n\in \N$ we write $\cA_{X,\tau,n}^-$ to denote the homogeneous component of degree $n$. Note that $\cR^{-,r}_X$ with the grading defined after Equation \eqref{eq:deg-R} is a graded subalgebra of $\cA^-_{X,\tau}$ and that the multiplication maps
\begin{align}
  \cR_X^{-,r}\ot\cH_{X,\tau} &\rightarrow \cA^-_{X,\tau} \label{eq:A-triang-1},\\
  U^-\ot U^0_\theta\ot U^+_X &\rightarrow \cA_{X,\tau}^- \label{eq:A-triang}
\end{align}
are linear isomorphisms. The QSP-subalgebra $\cB_\bc$ is filtered by a degree function defined by
\begin{align*}
   \deg(B_i)&=1 \qquad \mbox{for all $i\in I\setminus X$,}\\
  \deg(h)&=0 \qquad \mbox{for all $h\in \cH_{X,\tau}$.}
\end{align*}
We write $\cF_m(\cB_\bc)$ for the $m$-th filtered component, and $\gr(\cB_\bc)=\gr_\cF(\cB_\bc)$ for the associated graded algebra.
The following Lemma is implicit in \cite[Section 7]{MSRI-Letzter} for $\gfrak$ of finite type and in \cite[Sections 6, 7]{a-Kolb14}. The Lemma is stated explicitly in \cite[Section 2.4]{a-KY21}, \cite[Theorem 7.2]{a-Kolb23}.
%%%%%%%%%%%%%%%%%%%%%%%%%%%%%%%%%%%%%%%%
\begin{lem}
  There exists an isomorphism of graded algebras
  \begin{align}\label{eq:varphi}
    \varphi:\gr(\cB_\bc)\rightarrow \cA_{X,\tau}^-
  \end{align}
  such that $\varphi(B_i)=F_i$ for all $i\in I\setminus X$ and $\varphi(h)=h$ for all $h\in \cH_{X,\tau}$.
\end{lem}
%%%%%%%%%%%%%%%%%%%%%%%%%%%%%%%%%%%%%%%%%%
\begin{rema}
  We follow the conventions of \cite{a-Kolb23}. In \cite{a-KY21} we used the symbol $\varphi$ to denote the inverse of the map \eqref{eq:varphi}.
\end{rema}
%%%%%%%%%%%%%%%%%%%%%%%%%%%%%%%%%%%%%%%%%%
\begin{rema}
  In \cite[Section 2]{a-Knop90} Knop introduced the notion of the horospherical subgroup associated to a $G$-variety. Horospherical subgroups appear as stabilizer subgroups of the horospherical variety associated to a $G$-variety which had previously been constructed by Popov \cite{a-Popov87}.

Assume that the Lie algebra $\gfrak$ is complex, finite dimensional, semisimple.
Let $G$ be the connected, simply connected complex affine algebraic group with Lie algebra $\gfrak$. Each Satake diagram $(X,\tau)$ determines an involutive automorphism $\theta=\theta(X,\tau)$ of $G$ and $\gfrak$. Let $K=G^\theta$ be the corresponding fixed subgroup with Lie algebra $\kfrak=\gfrak^\theta$. In this case $G/K$ is a symmetric variety, and (up to conjugation) the corresponding horospherical subgroup has Lie algebra
\begin{align*}
  \sfrak=(\lfrak_X\cap \kfrak) \oplus \rfrak_X^-
\end{align*}
where $\lfrak_X\subset \gfrak$ is the standard Levi factor corresponding to $X\subset I$, while $\rfrak_X^-$ is the nilradical of the corresponding negative standard parabolic. This shows that the algebra $\cA_{X,\tau}^-$ is a $q$-analogue of the universal enveloping algebra $U(\sfrak)$ and explains why we call $\cA_{X,\tau}^-$ the quantum horospherical subalgebra corresponding to the Satake diagram $(X,\tau)$.

In \cite[Section 2.3]{a-KY21} we referred to the algebra $\cA_{X,\tau}^-$ as the partial parabolic subalgebra corresponding to $(X,\tau)$. The name quantum horospherical subalgebra is preferable.
\end{rema}  
%%%%%%%%%%%%%%%%%%%%%%%%%%%%%%%%%%%%%%%%%%%%%%%%%
  \subsection{The Letzter map}
%%%%%%%%%%%%%%%%%%%%%%%%%%%%%%%%%%%%%%%%%%%%%%%%%
  Let $U^\poly$ denote the subalgebra of $\Uq$ generated by $\cA^-_{X,\tau}$ and the elements $\Etil_i=E_iK_i^{-1}$, $K_i^{-1}$ for all $i\in I\setminus X$. We have
  \begin{align}\label{eq:EtilF}
    q^{-(\alpha_i,\alpha_j)}\Etil_i F_j - F_j \Etil_i = \delta_{i,j}\frac{1-K_i^{-2}}{q_i-q_i^{-1}} \qquad \mbox{for all $i,j\in I$.}
  \end{align}  
  Define $\cL_{X}^\poly=\cL_{X}\cap U^\poly$. The triangular decomposition \eqref{eq:triang_PX} implies that the multiplication map gives rise to a linear isomorphism
  \begin{align}\label{eq:Upoly-triang}
     \cR_X^{-,r}\ot \cL_{X}^\poly \ot \cR_X^{+,r}\cong U^\poly.
  \end{align}
  Let $I_\tau$ be a set of representatives of the $\tau$-orbits in $I\setminus X$ and consider the polynomial ring $U^0_\tau=\qfield\langle K_i^{-1}\,|\,i\in I_\tau\rangle$. By construction, multiplication defines a linear isomorphism
  \begin{align}\label{eq:Lpoly-Hpoly}
    \cH_{X,\tau}\ot U^0_\tau\cong \cL_X^{\poly}.
  \end{align}
  Combining \eqref{eq:A-triang-1}, \eqref{eq:Upoly-triang} and \eqref{eq:Lpoly-Hpoly} we obtain a triangular decomposition
  \begin{align}\label{eq:Upoly-triang-2}
    \cA_{X,\tau}^-\ot U^0_\tau\ot \cR_{X}^{+,r}\cong U^\poly.
  \end{align}  
  Consider the ideal  $U^0_{\tau,+}$ of $U^0_\tau$ generated by $\{K_i^{-1}\,|\,i\in I_\tau\}$, and consider the augmentation ideal
  $\cR_{X,+}^{+,r}=\cR_X^{+,r}\cap\ker(\vep)$ of $\cR_X^{+,r}$. Let $\cJ_{X,\tau}$ be the left ideal of $U^\poly$ generated by $\cR_{X,+}^{+,r}$ and $U^0_{\tau,+}$. The triangular decomposition \eqref{eq:Upoly-triang-2} gives us a vector space decomposition
  \begin{align}\label{eq:Upoly-decomp}
     U^\poly = \cA^-_{X,\tau}\oplus \cJ_{X,\tau},
  \end{align}
  see also \cite[Section 2.5]{a-KY21}. Let $\psi:U^\poly\rightarrow \cA^-_{X,\tau}$ denote the $\qfield$-linear projection with respect to the direct sum decomposition \eqref{eq:Upoly-decomp}.
  By definition, $\cB_\bc$ is a subalgebra of $U^\poly$ and hence we can restrict the map $\psi$ to $\cB_\bc$. We call this restriction
  \begin{align}\label{eq:Letzter-def}
     \psi:\cB_\bc\rightarrow \cA^-_{X,\tau}
  \end{align}
  the \textit{Letzter map}. Recall that $\cB_\bc$ is a filtered algebra, and that $\cA^-_{X,\tau}$ is graded, and therefore also filtered. We summarize essential properties of the Letzter map given in \cite[Section 2.5]{a-KY21}. Recall the isomorphism $\varphi:\gr(\cB_\bc)\rightarrow \cA_{X,\tau}^-$ given by \eqref{eq:varphi}. 
%%%%%%%%%%%%%%%%%%%%%%%%%%%%%%%%%
  \begin{lem}\label{lem:Letzter-props}
    \cite[Lemmas 2.10 and 2.11]{a-KY21}
    The Letzter map $\psi:\cB_\bc\rightarrow \cA^-_{X,\tau}$ is a linear isomorphism of filtered vector spaces with associated graded map $\gr(\psi)=\varphi$. Moreover, the following hold:
    \begin{enumerate}
      \item[(i)] $\psi(ab)=\psi(a\psi(b))$ for all $a,b\in \cB_\bc$.
      \item[(ii)] $\psi(h_1bh_2) = h_1\psi(b) h_2$ for all $b\in \cB_\bc$, $h_1,h_2\in \cH_{X,\tau}$.
      \item[(iii)] $\psi(T_i(b))=T_i(\psi(b))$ for all $i\in X$.
      \item[(iv)] $\psi(ab)-\psi(a)\psi(b)\in \bigoplus_{0\le k< m+n}\cA^-_{X,\tau,k}$ for all $a\in \cF_m(\cB_\bc)$, $b\in \cF_n(\cB_\bc)$.
    \end{enumerate}
  \end{lem}   
%%%%%%%%%%%%%%%%%%%%%%%%%%%%%%%%%  
\subsection{The horospherical Heisenberg double}
%%%%%%%%%%%%%%%%%%%%%%%%%%%%%%%%%%%%%%%%%%%%%%%%%  
Let $\cI_{X,\tau}$ be the left ideal of $U^\poly$ generated by $U^0_{\tau,+}$. Note that $\cI_{X,\tau}\subset U^\poly$ is a two-sided ideal because the elements $K_i^{-1}$ and the $U^\poly$ generators $q$-commute. 
We call
\begin{align*}
  \cW_{X,\tau}:=U^\poly\big/ \cI_{X,\tau}
\end{align*}
the \textit{horospherical Heisenberg double} associated to the generalized Satake diagram $(X,\tau)$.
Equation \eqref{eq:Lpoly-Hpoly} implies that
\begin{align*}
  \cL_X^\poly\Big/ (\cI_{X,\tau}\cap \cL_X^\poly) \cong \cH_{X,\tau}
\end{align*}  
and it follows from the isomorphism \eqref{eq:Upoly-triang} that the multiplication map gives rise to a linear isomorphism
\begin{align}\label{eq:W-triang}
   \cR_X^{-,r}\ot \cH_{X,\tau} \ot \cR_X^{+,r} \cong \cW_{X,\tau}. 
\end{align}
The ideal $\cI_{X,\tau}\subset U^\poly$ satisfies $\kow(\cI_{X,\tau})\subset \cI_{X,\tau}\ot \cI_{X,\tau}$ and hence $\cI_{X,\tau}$ is a subbialgebra of $U^\poly$. In particular, $\cI_{X,\tau}$ is a right coideal of $U^\poly$ and hence $\cW_{X,\tau}$ is a right $U^\poly$-comodule algebra.
  \begin{rema}
     Note that $\cW_{X,\tau}$ is not a bialgebra, as the counit $\vep$ is not defined on $\cW_{X,\tau}$.
  \end{rema}
We can give a presentation of $\cW_{X,\tau}$ in terms of generators and relations.
%%%%%%%%%%%%%%%%%%%%%%%%%%%%%%%%%%%%%%%%%%%%%%%
  \begin{lem}
    The algebra $\cW_{X,\tau}$ is generated by its subalgebras $U^-, G^+, U^0_\theta$, subject only to the additional commutation relations
    \begin{align}
      K_\beta F_i = q^{-(\alpha_i,\beta)} F_i K_\beta, \qquad
      K_\beta E_i &= q^{(\alpha_i,\beta)} E_i K_\beta \qquad \mbox{for all $\beta\in Q^\theta, i\in I$,}\label{eq:Wrel1}\\
      q^{-(\alpha_i,\alpha_j)} \Etil_i F_j - F_j \Etil_i &= \frac{\delta_{i,j}}{q_i-q_i^{-1}} \begin{cases} 1- K_i^{-2} &\mbox{if $i\in X$,}\\ 1&\mbox{if $i\in I\setminus X$.}\end{cases} \label{eq:Wrel2}
    \end{align}   
  \end{lem}
%%%%%%%%%%%%%%%%%%%%%%%%%%%%%%%%5
  \begin{proof}
    Define $G^+_X=S(U^+_X)$. The triangular decomposition \eqref{eq:W-triang} and the decompositions \eqref{eq:triang_Upm} and $\cR_X^{+,r}\ot G_X^+\cong G^+$ imply that multiplication defines an isomorphism
    \begin{align*}
      m:U^-\ot U^0_\theta \ot G^+\rightarrow \cW_{X,\tau}.
    \end{align*}
    On the other hand, let $\cW_{X,\tau}'$ denote the algebra generated by $U^-$, $G^+$ and $U^0_\theta$ subject only to the relations \eqref{eq:Wrel1} and \eqref{eq:Wrel2}. The relations \eqref{eq:Wrel1}, \eqref{eq:Wrel2} also hold in $\cW_{X,\tau}$ by definition, and hence there is a surjective algebra homomorphism
    \begin{align*}
       \pi':\cW'_{X,\tau} \rightarrow \cW_{X,\tau}
    \end{align*}
    mapping the generators of $\cW'_{X,\tau}$ to the corresponding generators of $\cW_{X,\tau}$. Moreover, the relations \eqref{eq:Wrel1} and \eqref{eq:Wrel2} imply that the multiplication map
    \begin{align*}
      m':U^-\ot U^0_\theta \ot G^+\rightarrow \cW'_{X,\tau}
    \end{align*}
    is surjective. As $\pi'\circ m'=m$ is a linear isomorphism and $m'$ is surjective, it follows that $\pi'$ is an isomorphism. 
  \end{proof}  
%%%%%%%%%%%%%%%%%%%%%%%%%%%%%%%%%%%%%%%%%%%%%%%
\begin{rema}
  Let $U^{\le, \cop}$ be the Hopf algebra $U^\le=U^-U^0$ with the opposite coproduct. There is a Hopf pairing $\langle\,,\,\rangle:U^{\le,\cop}\ot U^\ge\rightarrow \qfield$, see \cite[Proposition 1.2.3]{b-Lusztig94}, also \eqref{eq:pairing}. With respect to the action 
  $f\lact e = e_{(1)} \langle f,e_{(2)} \rangle$ the algebra $U^\ge$ is a left $U^{\le,\cop}$-module algebra. Hence, we can form the smash product $U^\ge\# U^{\le,\cop}$, commonly referred to as Heisenberg double, see \cite[4.1.10]{b-Montg93}. 
  If $X=\emptyset$ and $\tau=\id$ then the algebra $\cW_{\emptyset,\id}$ is isomorphic to the subalgebra of $U^\ge\# U^{\le,\cop}$ generated by $\Etil_i, F_i$ for all $i\in I$. The algebra $\cW_{\emptyset,\id}$ is also isomorphic to Kashiwara's bosonic algebra $\mathscr{B}_q(\gfrak)$, which plays an important role in the study of crystal bases, see \cite[Section 3.3]{a-Kashiwara91}.
\end{rema}
%%%%%%%%%%%%%%%%%%%%%%%%%%%%%%%%%%%%%%%%%%%%%%%%  
  We use the gradings of $\cR_X^{-,r}$ and $\cR_X^{+,r}$ defined after Equation \eqref{eq:deg-R} to define vector space gradings on $U^\poly$ and $\cW_{X,\tau}$. More precisely, using the triangular decompositions \eqref{eq:Upoly-triang} and \eqref{eq:W-triang}, we set
  \begin{align}
    U^\poly_n&=\sum_{\ell=0}^n \cR_{X,\ell}^{-,r}\ot \cL_{X}^\poly \ot \cR_{X,n-\ell}^{+,r}, \label{eq:Un-def}\\
    \cW_{X,\tau,n}&= \sum_{\ell=0}^n \cR_{X,\ell}^{-,r}\ot \cH_{X,\tau} \ot \cR_{X,n-\ell}^{+,r}\label{eq:Wn-def}
  \end{align}  
  for all $n\in \N$. We thus obtain vector space gradings
  \begin{align}\label{eq:gradings}
    U^\poly= \bigoplus_{n\in \N} U^\poly_n,\qquad \cW_{X,\tau}=\bigoplus_{n\in \N} \cW_{X,\tau,n},
  \end{align}
and the canonical projection $\pi:U^\poly\rightarrow \cW_{X,\tau}$ is a graded linear map. 

We want to describe the right $U^\poly$-coaction on the graded component $\cW_{X,\tau,n}$. To this end, we introduce a second vector space grading of $U^\poly$. For any $n\in \N$ let $U^0_\tau[n]$ denote the degree $n$ component of the polynomial ring $U^0_\tau$ and define
\begin{align*}
  U^\poly[n]=\cR_X^{-,r} \cH_{X,\tau} U^0_\tau[n] \cR_X^{+,r}.
\end{align*}
The triangular decomposition \eqref{eq:Upoly-triang} and Equation \eqref{eq:Lpoly-Hpoly} imply that
\begin{align}\label{eq:Upoly-grading2}
   U^\poly = \bigoplus_{n\in \N} U^\poly[n].
\end{align}
Recall that $\cR_X^{\pm,r}$ are right coideals of $\Uq^\pm U^0$. By definition of the coproduct \eqref{eq:kow-def} we have
\begin{align}
  \kow(\cR_{X,\ell}^{+,r})\subset \sum_{m=0}^\ell \cR_{X,m}^{+,r}\ot \cH_{X,\tau} U^0_\tau[m] \cR_{X,\ell-m}^{+,r} ,\label{eq:kowRX+}\\
  \kow(\cR_{X,\ell}^{-,r})\subset \sum_{m=0}^\ell \cR_{X,m}^{-,r}\ot \cR_{X,\ell-m}^{-,r} \cH_{X,\tau} U^0_\tau[m], \label{eq:kowRX-}
\end{align}
and hence the definition \eqref{eq:Wn-def} of $\cW_{X,\tau,n}$ implies that
\begin{align}\label{eq:kowWn}
  \kow(\cW_{X,\tau,n}) \subset \sum_{\ell=0}^n \cW_{X,\tau,\ell} \ot U^\poly[\ell] \qquad \mbox{for all $n\in \N$.}
\end{align}  
The restriction $\pi|_{\cB_\bc}:\cB_\bc\rightarrow \cW_{X,\tau}$ is injective, because  by Lemma \ref{lem:Letzter-props} the Letzter map $\psi:\cB_\bc\rightarrow \cA^-_{X,\tau}$ is an injective map which factors through $\pi$ by construction.
%%%%%%%%%%%%%%%%%%%%%%%%%%%%%%%%%%%%%%%%%%%%
\begin{thm}\label{thm:Bc-graded}
  The image $\pi(\cB_\bc)$ is a graded subspace of $\cW_{X,\tau}$ with respect to the grading \eqref{eq:gradings}. 
\end{thm}  
%%%%%%%%%%%%%%%%%%%%%%%%%%%%%%%%%%%%%%%%%%%%
\begin{proof}
  Let $b\in \pi(\cB_\bc)\setminus \{0\}$ and write $b$ as a sum
  \begin{align*}
    b=\sum_{\ell=0}^n b_\ell \qquad \mbox{with $b_\ell\in \cW_{X,\tau,\ell}$ and $b_n\neq 0$}
  \end{align*}  
    for some $n\in \N$. By induction it suffices to show that $b_n\in\pi(\cB_\bc)$. Let $\pi_n:U^\poly\rightarrow U^{\poly}[n]$ be the projection with respect to the direct sum decomposition \eqref{eq:Upoly-grading2}. As $\pi(\cB_\bc)$ is a right $\Uq^\poly$-subcomodule of $\cW_{X,\tau}$, Equation \eqref{eq:kowWn} implies that
    \begin{align}\label{eq:Step}
      (\id \ot \pi_n)\circ \kow(b_n)=(\id\ot \pi_n)\circ \kow(b)\in \pi(\cB_\bc)\ot U^\poly[n].
    \end{align}
    On the other hand, Equations \eqref{eq:kowRX+} and \eqref{eq:kowRX-} imply that
    \begin{align*}
      (\id \ot (\vep \circ \pi_n))\circ \kow(b_n)=(\id \ot \vep)\circ \kow(b_n)=b_n.
    \end{align*}
    Hence, application of the counit $\vep$ to the second tensor factor of Equation \eqref{eq:Step} gives us $b_n\in \pi(\cB_\bc)$, as desired. 
\end{proof}  
%%%%%%%%%%%%%%%%%%%%%%%%%%%%%%%%%%%%%%%%%%%%%%
\section{Short star products}\label{sec:shortstar}
%%%%%%%%%%%%%%%%%%%%%%%%%%%%%%%%%%%%%%%%%%%%%%
We recall the notion of a short star product as developed by Etingof and Stryker in \cite{a-ES}. In \cite{a-ES} star products are only considered for commutative, graded Poisson algebras but the concept of a star-product in \cite{a-ES} also makes sense for non-commutative algebras and filtered deformations of this sort were studied by various authors before, see e.g. \cite{Rains22} for twisted homogeneous coordinate rings. 
We will see that the Letzter map is a quantization map in the sense of \cite{a-ES}. Our main goal is to show that the corresponding star product on the quantum horospherical algebra is short.
%%%%%%%%%%%%%%%%%%%%%%%%%%%%%%%%%%%%%%%%%%%%%%%
\subsection{Star products}\label{sec:star}
%%%%%%%%%%%%%%%%%%%%%%%%%%%%%%%%%%%%%%%%%%%%%%
Assume that $A$ is an $\N$-graded $\qfield$-algebra $A=\bigoplus_{j\in \N} A_j$. Throughout Section \ref{sec:star} we allow $\qfield$ to be an arbitrary field. For all $n, m \in \N$, we denote 
\[
A_{[n,m]} : = 
\begin{cases}
A_n \oplus \ldots \oplus A_m, &\mbox{if} \; \; n \leq m
\\
0, &\mbox{otherwise}
\end{cases}
\]
and $A_{<m} := A_{[0, m-1]}$, $A_{\le m}:= A_{[0,m]}$. 
%%%%%%%%%%%%%%%%%%%%%%%%%%%%%%%%%%%%%%%%%%%%%%%%%%%%%%%%%%%%%%%%
\begin{defi}\label{def:star-prod}
\cite[Section 5.1]{a-KY20}, \cite[Section 2.1]{a-ES}
Let $A$ be an $\N$-graded $\qfield$-algebra.
\begin{enumerate}
\item[(i)] A star product on $A$ is an associative bilinear operation $* : A \times A \to A$, $(a,b)\mapsto a\ast b$
such that
\begin{align}\label{eq:a*b-ab}
  a * b - ab \in A_{< m+ n} \qquad \mbox{for all $a \in A_m, b \in A_n$.}
\end{align}
\item[(ii)] A star product $*$ on $A$ is called short if 
\[
a * b \in A_{[|m-n|, m+n]} \qquad \mbox{for all $a \in A_m, b\in A_n$.} 
\]
\item[(iii)] A star product $\ast$ on $A$ is called $0$-equivariant
if
\[
a * h = ah \quad \mbox{and} \quad  h * a= h a \qquad \mbox{for all $h \in A_0, a \in A$.} 
\]
\item[(iv)] A star product $\ast$ on $A$ is called $\Z/2$-equivariant if 
\begin{equation}
\label{Z2cond}
a * b \in \bigoplus_{k=0}^{\lfloor \frac{m+n}{2} \rfloor} A_{m+n-2 k} \qquad \mbox{for all $a \in A_m, b \in A_n$.} 
\end{equation}
\end{enumerate} 
\end{defi}
%%%%%%%%%%%%%%%%%%%%%%%%%%%%%%%%%%%
\begin{rema} 
\hfill
  \begin{enumerate}
     \item Every short star product is 0-equivariant: if $a \in A_m$ and $h \in A_0$, then by the shortness assumption $a * h \in A_m$, and by \eqref{eq:a*b-ab}, $a * h = ah$; similarly $h* a = ha$. 
     \item Any $\Z/2$-equivariant $\ast$-product on the graded algebra $A$ is uniquely determined by a family of maps $C_k:A\ot A \rightarrow A$ of degree $-2k$ such that
     \begin{align*}
         a\ast b = ab + \sum_{k=1}^{\lfloor \frac{m+n}{2}\rfloor} C_k(a,b).
     \end{align*}
     The linear map $C_1:A\ot A \rightarrow A$ of degree $-2$ defines a Hochschild $2$-cocycle, that is, $a C_1(b,c) - C_1(ab,c) + C_1(a,bc) - C_1(a,b)c=0$ for all $a,b,c\in A$. If the graded algebra $A$ is commutative, then
     \[
         \{a,b\}:=C_1(a,b)-C_1(b,a) \qquad \mbox{for all $a, b\in A$}
     \]
     defines a Poisson structure on $A$ and we think of the algebra $(A, *)$ as of a quantization of the Poisson algebra $(A, \{.,.\})$.
     \item In \cite{a-ES} all star products are assumed to satisfy the condition \eqref{Z2cond} and the term $\Z/2$-equivariant is used in relation to filtered quantizations equipped with quantization maps, see \cite[Proposition 2.5]{a-ES}.
     \item Nondegenerate short star products on connected graded algebras are classified in terms nondegenerate twisted traces on filtered quantizations \cite[Proposition 3.3]{a-ES}. In this paper we are concerned with the case where $A$ is noncommutative and $\dim A_0 >1$.
  \end{enumerate}
\end{rema}  
%%%%%%%%%%%%%%%%%%%%%%%%%%%%%%%%%%%%%%%%5
If $*$ is a star product on an $\N$-graded algebra $A$ then $(A,\ast)$ is a filtered algebra with 
\begin{align} \label{eq:FA}
  \cF_m (A) := A_{\leq m} \qquad \mbox{ for all $m\in \N$.}
\end{align}
It follows from \eqref{eq:a*b-ab} that the associated graded algebra satisfies
\[
\gr_\cF (A, *) \cong A.
\]
Conversely, a filtered linear isomorphism between $A$ and a filtered algebra gives rise to a $*$-product on $A$. Here we consider $A$ as a filtered vector space with filtration given by \eqref{eq:FA}.
%%%%%%%%%%%%%%%%%%%%%%%
\begin{defi}{\cite[Section 2.2]{a-ES}} \label{def:quant-map}
   Let $A$ be an $\N$-graded algebra and $B=\bigcup_{m\in \N}\cF_m(B)$ a filtered algebra. A filtered  isomorphism of vector spaces $\phi:A\rightarrow B$ is called a quantization map if the associated graded map $\gr(\phi): A \rightarrow \gr_\cF(B)$ is an isomorphism of graded algebras.
\end{defi}
%%%%%%%%%%%%%%%%%%%%%%%%
\begin{lem}{\cite[Section 2.2]{a-ES}} \label{lem:quant-star}
  Let $A$ be an $\N$-graded algebra and $B$ a filtered algebra, and let $\phi:A\rightarrow B$ be a quantization map. Then 
  \begin{align}\label{eq:aastb}
a * b := \phi^{-1} (\phi(a) \phi(b)) \qquad\mbox{for all $a,b\in A$} 
\end{align}
defines a $\ast$-product on $A$.
\end{lem}
%%%%%%%%%%%%%%%%%%%%%%%
\begin{proof}
  As $B$ is an associative algebra, so is $(A,\ast)$. Moreover, $(A,\ast)$ is filtered with $\cF_m (A)= A_{\leq m}$ for all $m\in \N$ because $\phi$ is a filtered linear map. As $\gr(\phi)$ is an isomorphism of graded algebras we get $\gr_\cF(A,\ast)=A$ by \eqref{eq:aastb}. Hence $a\ast b-ab\in A_{<m+n}$ for all $a\in A_m, b\in A_n$.
\end{proof}
%%%%%%%%%%%%%%%%%%%%%%%
If the graded algebra $A$ is generated in degrees 0 and 1, then every star product algebra structure $(A,*)$ is also generated in degrees 0 and 1. The first two statements of the following Lemma were proved in \cite[Lemma 5.2]{a-KY20}.
\begin{lem} 
\label{lem:sprod-isom}
Let $A$ be an $\N$-graded $\qfield$-algebra generated in degrees 0 and 1. 
\begin{enumerate}
\item[(i)] Any 0-equivariant star product on $A$ is uniquely determined by the $\qfield$-linear map $\mu^L:A_1\to \End_{\qfield}(A)$, $f\mapsto \mu^L_f$ defined by
\[
\mu^L_f(a) = f * a  - f a \qquad \mbox{for all $f \in A_1, a \in A$.}
\]
\item[(ii)] If $\cR$ is a graded subalgebra of $A$ such that $A_0 \cR = \cR A_0 = A$, then every 0-equivariant star product on $A$ is uniquely determined by the $\qfield$-linear map $\mu^L : \cR_1 \to \Hom_{\qfield}(\cR,A)$, $f\mapsto \mu^L_f$ defined by
\begin{align}\label{eq:muLU-def}
   \mu^L_f(b) = f*b - fb \qquad \mbox{for all $f \in \cR_1, b \in \cR$.}
\end{align}
\item[(iii)] Retain the setting of (ii). Assume, moreover, that $A_0$ is a Hopf algebra, that $\cR$ is invariant under the right adjoint action of $A_0$, and that there exist $f_1,f_2,\dots,f_\ell\in \cR_1$ such that $\cR_1=\mathrm{span}_{\qfield}\{\ad_r(A_0)(f_i)\,|\,i=1,\dots,\ell\}$. Then every $0$-equivariant star-product on $A$ is uniquely determined by the $\qfield$-linear maps $\mu_{f_i}^L\in \Hom_{\qfield}(\cR,A)$ defined by \eqref{eq:muLU-def} for $i=1,\dots,\ell$.
\end{enumerate}
\end{lem}
%%%%%%%%%%%%%%%%%%%%%%%%%%%%%%%%%%%%%%
\begin{proof}[Proof of (iii)] 
  For any $0$-equivariant star product on $A$ and any $a\in A_0$, $b\in \cR$ we have
  \begin{align*}
    \mu^L_{\ad_r(a)(f_i)}(b) = S(a_{(1)}) \mu^L_{f_i} (\ad_r(S^{-1}(a_{(3)}))(b)) a_{(2)}.
  \end{align*}  
  Hence (iii) follows from (ii).
\end{proof}
%%%%%%%%%%%%%%%%%%%%%%%%%%%%%%%%%%%%%%
\subsection{The star product interpretation of QSP-subalgebras}\label{sec:star-Bc}
%%%%%%%%%%%%%%%%%%%%%%%%%%%%%%%%%%%%%%
From now on we assume the generalized Satake diagram $(X,\tau)$ to be fixed and drop the subscripts $X,\tau$ to simplify notation. Hence we will write $\cH$, $\cW$, $\cI$, $\cJ$ instead of $\cH_{X,\tau}$, $\cW_{X,\tau}$, $\cI_{X,\tau}$, $\cJ_{X,\tau}$ etc. Moreover, also dropping the superscript, we will denote the quantum horospherical subalgebra $\cA^-_{X,\tau}$ by $\cA$.

Lemma \ref{lem:Letzter-props} implies that the inverse of the Letzter map $\psi:\cB_\bc\rightarrow \cA$ is a quantization map in the sense of Definition \ref{def:quant-map}. By Lemma \ref{lem:quant-star}
\begin{align}\label{eq:ast-def}
  a\ast b = \psi(\psi^{-1}(a)\psi^{-1}(b)) \qquad \mbox{for all $a,b\in \cA$.}
\end{align}
defines a star product on $\cA$. It follows from Lemma \ref{lem:Letzter-props}.(ii) that this star product is 0-equivariant.
By construction, the Letzter map $\psi:\cB_\bc\rightarrow (\cA,\ast)$ is an isomorphism of $\qfield$-algebras such that $\psi|_{\cH}=\id_\cH$ and $\psi(B_i)=F_i$ for all $i\in I\setminus X$. We use the notation
\begin{align}\label{eq:mui-def}
  \mu_i^L(a)=F_i\ast a - F_ia, \qquad \mu_i^R(a)=a\ast F_i - aF_i \qquad \mbox{for all $i\in I\setminus X$, $a\in \cA$.}
\end{align}
 For $\cR=\cR_X^{-,r}$ the conditions of Lemma \ref{lem:sprod-isom}.(iii) are satisfied. Hence, the maps $\mu_i^L$ for $i\in I\setminus X$ determine the $0$-equivariant star product \eqref{eq:ast-def} uniquely. 
 The $0$-equivariance of the star-product \eqref{eq:ast-def} and \cite[Lemma 3.3]{a-KY21} imply that, if $u=ah$ with $h\in \cH$ and $a\in \cR_X^{-,r}$, then we have
\begin{align}\label{eq:muiL}
  \mu_i^L(u)=\mu_i^L(a)h = - c_i\frac{q^{-(\alpha_i,\theta(\alpha_i))}}{q_i-q_i^{-1}} K_{-\alpha_i-\theta(\alpha_i)} T_{w_X} \circ \partial^L_{\tau(i)}\circ T^{-1}_{w_X}(a)h 
\end{align}
for all $i\in I\setminus X$.
By Equation \eqref{eq:muiL} the map $\mu_i^L:\cA\rightarrow \cA$ is homogeneous of degree $-1$. We will show in Corollary \ref{cor:muiR} that $\mu_i^R$ is also homogeneous of degree $-1$, but we do not know this yet. By Equations \eqref{eq:muiL} and \eqref{eq:skew-der} the map $\mu^L_i:\cR_X^{-,r}\rightarrow\cA$ is a skew-derivation in the following sense
%%%%%%%%%%%%%%%%%%%%%%%%%%%%%%%%%%%%%%%%%%%%
\begin{align}\label{eq:muiL-skew}
  \mu_i^L(ab) = \mu_i^L(a) b + q^{(\alpha_i,\mu)} a \mu_i^L(b) \qquad \mbox{for all $a\in (\cR_X^{-,r})_{-\mu}$, $b\in \cR_X^{-,r}$.}
\end{align}  
%%%%%%%%%%%%%%%%%%%%%%%%%%%%%%%%%%%%%%%%%%%%
For later use we note the following invariance of $\mu_i^L$ and $\mu_i^R$ under the right adjoint action of $U^+_X$.
%%%%%%%%%%%%%%%%%%%%%%%%%%%%%%%%%%%%%%%%%%%%
\begin{lem}
  For all $i\in I\setminus X$, $j\in X$ and $a\in \cA$ we have
  \begin{align}
    \mu_i^L(\ad_r(E_j)(a)) &= q^{-(\alpha_i,\alpha_j)}\ad_r(E_j)(\mu_i^L(a)),\label{eq:muL-adr}\\
    \mu_i^R(\ad_r(E_j)(a)) &= \ad_r(E_j)(\mu_i^R(a)).\label{eq:muR-adr}
  \end{align}
\end{lem}  
%%%%%%%%%%%%%%%%%%%%%%%%%%%%%%%%%%%%%%%%%%%%
\begin{proof}
  By the $0$-equivariance of the star product, the algebra $(\cA,\ast)$ is a right $\cH$-module algebra under the right adjoint action. Hence, using $\ad_r(E_j)(F_i)=0$, we obtain
  \begin{align*}
    \ad_r(E_j)(F_i\ast a)&=\ad_r(K_j)(F_i)\ast \ad_r(E_j)(a)=q^{(\alpha_i,\alpha_j)}F_i\ast \ad_r(E_j)(a),\\
     \ad_r(E_j)(a\ast F_i)&=\ad_r(E_j)(a)\ast F_i
  \end{align*}
  which implies the relations \eqref{eq:muL-adr} and \eqref{eq:muR-adr}.
\end{proof}  
%%%%%%%%%%%%%%%%%%%%%%%%%%%%%%%%%%%%%%%%%%%%
For any $m\in \N$ we write
\begin{align*}
  \cA_1^{\ast m}=\underbrace{\cA_1\ast\cA_1\ast \cdots\ast \cA_1}_{\mbox{$m$ factors}}
\end{align*}
and we set $\cA_1^{\ast 0}=\cA_0$.
The fact that $\mu_i^L$ is homogeneous of degree $-1$ implies by induction that
\begin{align}\label{eq:Am-in-Aastm}
  \cA_m\subset \sum_{k=0}^{\lfloor m/2\rfloor} \cA_1^{\ast(m-2k)}.
\end{align}  
%%%%%%%%%%%%%%%%%%%%%%%%%%%%%%%%%%%%%%%%
\begin{lem}\label{lem:Z2-ast}
  The star product $\ast$ on $\cA$ is $\Z/2$-equivariant.
\end{lem}
%%%%%%%%%%%%%%%%%%%%%%%%%%%%%%%%%%%%%%%%
\begin{proof}
  Indeed, by \eqref{eq:Am-in-Aastm} we have
  \begin{align*}
      \cA_m\ast \cA_n\subset \sum_{k=0}^{\lfloor (m+n)/2\rfloor} \cA_1^{\ast(m+n-2k)}.
  \end{align*}
  This shows that the linear map $s:\cA\rightarrow \cA$ defined by $s(a)=(-1)^ma$ for all $a\in \cA_m$ is an algebra automorphism of the star product algebra $(\cA,\ast)$. 
\end{proof}  
%%%%%%%%%%%%%%%%%%%%%%%%%%%%%%%%%%%%%%%%

%%%%%%%%%%%%%%%%%%%%%%%%%%%%%%%%%%%%%%%%
\subsection{Shortness of the star product on $\cA$}\label{sec:short-star}
%%%%%%%%%%%%%%%%%%%%%%%%%%%%%%%%%%%%%%%%
Consider $V:=U^\poly/\cJ$ as a left $U^\poly$-module. Let
\begin{align*}
  \eta: V\rightarrow \cA
\end{align*}
be the linear isomorphism given by the direct sum decomposition \eqref{eq:Upoly-decomp} and let $v_0\in V$ denote the right coset of $1\in U^\poly$, in other words $\eta(v_0)=1$. The following properties of the projection $\psi:U^\poly\rightarrow \cA$ follow directly from the definitions
\begin{align}
  \psi(u)&=\eta(uv_0) \qquad \mbox{for all $u\in U^\poly$,}\label{eq:psiu}\\
  \psi^{-1}(a)v_0&=a v_0 \qquad \mbox{for all $a\in \cA$ where $\psi^{-1}(a)\in \cB_\bc$.}\label{eq:psiav}
\end{align}  
Hence, for any $a,b\in \cA$ we have
\begin{align}\label{eq:ast-eta}
  a\ast b = \psi(\psi^{-1}(a)\psi^{-1}(b))
  \stackrel{\eqref{eq:psiu}}{=}\eta(\psi^{-1}(a)\psi^{-1}(b)v_0)
  \stackrel{\eqref{eq:psiav}}{=}\eta(\psi^{-1}(a)b v_0).
\end{align}
We use this formula to spell out the degree shift effected by homogeneous elements in $\cR_X^{\pm, r}$.
%%%%%%%%%%%%%%%%%%%%%%%%%%%%%%%%%%%%%%
\begin{lem}\label{lem:deg-shift}
  Let $m,n\in \N$. For all $x\in \cR_{X,m}^{+,r}$ and $y\in \cR_{X,m}^{-,r}$ and $a\in \cA_n$ we have
  \begin{align}
    \eta(ya v_0)&\in \cA_{n+m},\label{eq:deg-shift-y}\\
    \eta(xa v_0)&\in \cA_{n-m}.\label{eq:deg-shift-x}
  \end{align}
  In particular, if $m>n$ then $\eta(xav_0)=0$.
\end{lem}
%%%%%%%%%%%%%%%%%%%%%%%%%%%%%%%%%%%%%%
\begin{proof}
  The relation \eqref{eq:deg-shift-y} follows directly from the inclusion $\cR^{-,r}_{X,m}\subset \cA_m$. It suffices to prove the relation \eqref{eq:deg-shift-x} for $m=1$. As left multiplication by elements in $\cH$ preserves degree, we may assume that $x=\Etil_i$ for some $i\in I\setminus X$. Using the decomposition \eqref{eq:A-triang} and the commutation relations \eqref{eq:EtilF} we see that
  \begin{align*}
  \Etil_i \cA_n \subset \cA_{n-1} + \cJ. 
  \end{align*}
  This implies that $\eta(\Etil_iav_0)\in \cA_{n-1}.$
\end{proof}
%%%%%%%%%%%%%%%%%%%%%%%%%%%%%%%%%%%%%%
Recall that $\cI$ is a two-sided ideal in $U^\poly$ with $\cI\subset \cJ$. As $\cW=U^\poly/\cI$, we may consider $V$ as a $\cW$-module. Recall also the definition \eqref{eq:Wn-def} of $\cW_m=\cW_{X,\tau,m}$ for $m\in \N$.
%%%%%%%%%%%%%%%%%%%%%%%%%%%%%%%%%%%%%%
\begin{lem}\label{lem:WA-short}
  Let $u\in \cW_m$ and $a\in \cA_n$ for some $m,n\in \N$ with $m\ge n$. Then
  \begin{align*}
     \eta(ua v_0)\in \bigoplus_{\ell=m-n}^{m+n} \cA_\ell.
  \end{align*}  
\end{lem}
%%%%%%%%%%%%%%%%%%%%%%%%%%%%%%%%%%%%%%
\begin{proof}
  By the definition of $\cW_m$ and linearity we may assume that $u\in \cR_{X,m-p}^{-,r} \cH \cR_{X,p}^{+,r}$ for some $0\le p\le m$. By Lemma \ref{lem:deg-shift} we obtain $\eta(uav_0)=0$ if $p>n$ and $\eta(uav_0)\in \cA_{m-2p+n}$. Hence we may assume that $p\le n$, and in this case the claim follows as $m-n\le m-2p+n\le m+n$.
\end{proof}  
%%%%%%%%%%%%%%%%%%%%%%%%%%%%%%%%%%%%%%
With these preparations we are ready to establish the main goal of this section.
%%%%%%%%%%%%%%%%%%%%%%%%%%%%%%%%%%%%%%
\begin{thm}\label{thm:short}
  The star product of $\cA$ defined by Equation \eqref{eq:ast-def} is short.
\end{thm}
%%%%%%%%%%%%%%%%%%%%%%%%%%%%%%%%%%%%%%
\begin{proof}
  Let $a\in \cA_m$ and $b\in \cA_n$. Recall that the maps $\mu_i^L$ are homogeneous of degree $-1$ for all $i\in I\setminus X$. Moreover, by Equation \eqref{eq:Am-in-Aastm} we have $a\in \sum_{k=0}^{\lfloor m/2 \rfloor} \cA_1^{\ast(m-2k)}$. These two facts together imply that
  \begin{align*}
     a\ast b \in \bigoplus_{\ell=n-m}^{n+m} \cA_\ell \qquad \mbox{if $m\le n$.} 
  \end{align*}
  Now assume that $m\ge n$. By Equation \eqref{eq:ast-eta} we have
  \begin{align*}
     a\ast b = \eta(\psi^{-1}(a) bv_0)=\eta(\pi(\psi^{-1}(a))bv_0)
  \end{align*}
  with $\psi^{-1}(a)\in \cB_\bc$. By Theorem \ref{thm:Bc-graded} we can write $\pi(\psi^{-1}(a))=\sum_\ell \beta_\ell$ with $\beta_\ell \in \pi(\cB_\bc)\cap \cW_\ell$. Recall that the projection map $\psi:U^\poly\rightarrow \cA$ is graded with respect to the grading \eqref{eq:gradings} on $U^\poly$. As $a\in \cA_m$ and $\psi|_{\cB_\bc}:\cB_\bc\rightarrow \cA$ is an isomorphism, we have $\beta_\ell=0$ for all $\ell\neq m$, that is $\pi(\psi^{-1}(a))\in \cW_m$. We now use Lemma \ref{lem:WA-short} to obtain
  \begin{align*}
     a\ast b \in \bigoplus_{\ell=m-n}^{m+n} \cA_\ell  
  \end{align*}
  which completes the proof in the case $m\ge n$.
  \end{proof}  
%%%%%%%%%%%%%%%%%%%%%%%%%%%%%%%%%%%%%%
  As an immediate consequence of Theorem \ref{thm:short} we obtain the outstanding homogeneity of the map $\mu^R_i$.
%%%%%%%%%%%%%%%%%%%%%%%%%%%%%%%%%%%%%%
\begin{cor}\label{cor:muiR}
  Let $i\in I\setminus X$. The map $\mu_i^R:\cA\rightarrow \cA$ defined by $\mu_i^R(a)=a\ast F_i-aF_i$ is homogeneous of degree $-1$.
\end{cor}
%%%%%%%%%%%%%%%%%%%%%%%%%%%%%%%%%%%%%
\begin{proof}
  Let $a\in \cA_n$. By the shortness of the star product we have $\mu_i^R(a)\in \cA_n\oplus \cA_{n-1}$. However, as the star product is $\Z/2$-equivariant by Lemma \ref{lem:Z2-ast}, we have
  \begin{align*}
    a\ast F_i\subset \cA_{n+1}\oplus \cA_{n-1}\oplus \cA_{n-3} \oplus \cdots.
  \end{align*}
Hence $\mu_i^R(a)\in \cA_{n-1}$.
\end{proof}
%%%%%%%%%%%%%%%%%%%%%%%%%%%%%%%%%%%%%
\begin{rema}
   It is noteworthy that the statement of Corollary \ref{cor:muiR} (together with the analogous statement for $\mu_i^L$) is equivalent to the shortness of the $\Z/2$-equivariant star product on the graded algebra $\cA$. 
\end{rema}
%%%%%%%%%%%%%%%%%%%%%%%%%%%%%%%%%%%%%
\section{The algebra anti-automorphism $\sigma_\tau$ for QSP-subalgebras}\label{sec:sigmatau}
%%%%%%%%%%%%%%%%%%%%%%%%%%%%%%%%%%%%%
As a first application of the shortness of the star product of $\cA$ we now show that the map $\sigma\circ \tau$ defines an algebra anti-automorphism of $(\cA,\ast)$ and hence of $\cB_\bc$. This will allow us to give new and straightforward proofs of the existence of the bar involution of $\cB_\bc$ and of \cite[Conjecture 2.7]{a-BalaKolb15}.
%%%%%%%%%%%%%%%%%%%%%%%%%%%%%%%%%%%%%
\subsection{The left and right lower order term maps}
%%%%%%%%%%%%%%%%%%%%%%%%%%%%%%%%%%%%%
%Recall the lower order term maps $\mu_i^L, \mu_i^R:\cA\rightarrow \cA$ defined for $i\in I\setminus X$ by \eqref{eq:mui-def}. 
We aim to understand the relation between $\mu_i^L$ and $\mu^R_i$. The following lemma is a first elementary instance of this relation.
%%%%%%%%%%%%%%%%%%%%%%%%%%%%
\begin{lem}\label{lem:muLR-1}
  For all $i, \ell \in I\setminus X$ we have
  \begin{align*}
     T_{w_X}^{-1}(\mu^L_\ell(T_{w_X}(F_i))) = \mu_i^R(T_{w_X}^{-1}(F_\ell))=-\delta_{\ell,\tau(i)}c_i \frac{q^{-(\alpha_i,\theta(\alpha_i))}}{q_i-q_i^{-1}} K_{\alpha_i+\theta(\alpha_i)}.
  \end{align*}
\end{lem}
%%%%%%%%%%%%%%%%%%%%%%%%%%%%
\begin{proof}
  We prove the first equality. By definition of $\mu^R_i$ we have
  \begin{align*}
    T_{w_X}^{-1}(F_\ell)\ast F_i=T_{w_X}^{-1}(F_{\ell}) F_i + \mu^R_i(T_{w_X}^{-1}(F_\ell)).
  \end{align*}
  On the other hand, by Lemma \ref{lem:Letzter-props}.(iii) the map $T_{w_X}$ is an algebra automorphism of $(\cA,\ast)$. Hence we have
  \begin{align*}
    T_{w_X}^{-1}(F_\ell)\ast F_i&= T_{w_X}^{-1}(F_\ell \ast T_{w_X}(F_i))=T_{w_X}^{-1}(F_\ell T_{w_X}(F_i)) + T_{w_X}^{-1}(\mu_\ell^L(T_{w_X}(F_i))).
  \end{align*}
  Comparison of the two equations proves the first equality. The second equality now follows from the explicit expression for $\mu_\ell^L$ given by Equation \eqref{eq:muiL}. 
\end{proof}  
%%%%%%%%%%%%%%%%%%%%%%%%%%%%
Recall from Section \ref{sec:star-Bc} and Corollary \ref{cor:muiR} that $\mu_i^L$ and $\mu_i^R$ are homogeneous maps of degree $-1$. Equation \eqref{eq:muiL} gives us an explicit description of the map $\mu_i^L$. The second statement of the following proposition provides a main ingredient used to find a similar description for $\mu_i^R$.
%%%%%%%%%%%%%%%%%%%%%%%%%%%%%
\begin{prop}\label{prop:muR-muL}
  The maps $\mu_\ell^L, \mu_i^R:\cA\rightarrow \cA$ for $i,\ell\in I\setminus X$ satisfy the following relations:
   \begin{enumerate}
    \item  $\mu_i^R \circ \mu_\ell^L = \mu_\ell^L\circ \mu_i^R$,
    \item $\mu_i^R \circ T_{w_X}^{-1}\circ \mu_\ell^L\circ T_{w_X} = T_{w_X}^{-1}\circ\mu_\ell^L\circ T_{w_X}\circ \mu_i^R $.
  \end{enumerate}  
\end{prop}  
%%%%%%%%%%%%%%%%%%%%%%%%%%%%%%%%%%%%%
\begin{proof}
  Let $a\in \cA_n$ for some $n\in \N$. Using the associativity of the star product, we obtain two expressions for $F_\ell\ast a\ast F_i$, namely
  \begin{align*}
    (F_\ell\ast a)\ast F_i&=(F_\ell \ast a)F_i+\mu^R_i(F_\ell\ast a)\\
    &=F_\ell a F_i +\mu^L_\ell(a)F_i + \mu^R_i(F_\ell a)+\mu^R_i(\mu^L_\ell(a)),\\
     F_\ell\ast (a\ast F_i)&=F_\ell (a\ast F_i)+\mu^L_\ell(a \ast F_i)\\
    &=F_\ell a F_i +F_\ell \mu^R_i(a) + \mu^L_\ell(a F_i)+\mu^L_\ell(\mu^R_i(a)).
  \end{align*}  
  By Equation \eqref{eq:muiL} the map $\mu^L_\ell$ is homogeneous of degree $-1$, and by Corollary \ref{cor:muiR} so is the map $\mu^R_i$. Using this, we obtain the relation  $\mu_i^R (\mu_\ell^L(a)) = \mu_\ell^L(\mu_i^R (a))$ by comparison of the components in $\cA_{n-2}$ in the above expressions for $F_\ell\ast a\ast F_i$. This proves the first formula in the Proposition. The second formula is obtained analogously starting from $T_{w_X}^{-1}(F_\ell)\ast a \ast F_i$. Indeed, we obtain
    \begin{align}
    (T_{w_X}^{-1}(F_\ell)\ast a)\ast F_i=&(T_{w_X}^{-1}(F_\ell) \ast a)F_i+\mu^R_i(T_{w_X}^{-1}(F_\ell)\ast a)\nonumber\\
      =&T_{w_X}^{-1}(F_\ell) a F_i +T_{w_X}^{-1}(\mu^L_\ell(T_{w_X}(a)))F_i + \mu^R_i(T_{w_X}^{-1}(F_\ell) a)\label{eq:associative-1}\\
      &+\mu^R_i(T_{w_X}^{-1}(\mu^L_\ell(T_{w_X}(a)))),\nonumber\\
     T_{w_X}^{-1}(F_\ell)\ast (a\ast F_i)=&T_{w_X}^{-1}(F_\ell) (a\ast F_i)+T_{w_X}^{-1}(\mu^L_\ell (T_{w_X}(a \ast F_i)))\nonumber\\
     =&T_{w_X}^{-1}(F_\ell) a F_i +T_{w_X}^{-1}(F_\ell) \mu^R_i(a) + T_{w_X}^{-1}(\mu^L_\ell (T_{w_X}(a F_i)))\label{eq:associative-2}\\
     &+T_{w_X}^{-1}(\mu^L_\ell(T_{w_X}(\mu^R_i(a))))\nonumber
    \end{align}
    and again comparison of components in $\cA_{n-2}$ gives the desired formula.
\end{proof}
%%%%%%%%%%%%%%%%%%%%%%%%%%%%%%%%%%%%%%
\subsection{Transforming $\mu_i^L$ from left to right}
%%%%%%%%%%%%%%%%%%%%%%%%%%%%%%%%%%%%%%
Recall the $\qfield$-linear involutive algebra anti-automorphism $\sigma:\Uq\rightarrow \Uq$ defined by \eqref{eq:sigma-def}. For brevity we write
\begin{align*}
  \sigma\tau=\sigma \circ \tau:\Uq\rightarrow \Uq.
\end{align*}
The algebra anti-automorphisms $\sigma$ and $\sigma\tau$ restrict to the quantum horospherical subalgebra $\cA$. By \cite[37.2.4]{b-Lusztig94} we know that $\sigma \circ T_i \circ \sigma = T_i^{-1}$ for all $i\in I$ and therefore
\begin{align}\label{eq:sigmaTwX}
  \sigma\tau\circ T_{w_X} = T_{w_X}^{-1}\circ\sigma\tau.
\end{align}
Hence, Proposition \ref{prop:R-UTU} implies that
\begin{align}\label{eq:sigmaRX-}
  \sigma\tau(\cR_X^{-,r})= \sigma(\cR_X^{-,r})=\sigma(U^-\cap T_{w_X}(U^-))=U^-\cap T_{w_X}^{-1}(U^-)=T_{w_X}^{-1}(\cR_X^{-,r}).
\end{align}
Moreover, the skew-derivation properties \eqref{eq:skew-der}, \eqref{eq:skew-der-+} of $\partial_i^L$ and $\partial_i^R$ imply the relation $\sigma\circ \partial_i^L\circ \sigma=\partial_i^R$ on $U^\pm$ and hence
\begin{align}\label{eq:sigma-tau-partial}
  \sigma\tau\circ \partial_i^L\circ \sigma\tau=\partial_{\tau(i)}^R.
\end{align}  
For all $i\in I\setminus X$ define
\begin{align}\label{eq:mutil-def}
  \mutil_i^R=\sigma\tau \circ \mu_{\tau(i)}^L\circ \sigma\tau: \cA \rightarrow \cA.
\end{align}
We will show in Theorem \ref{thm:muiR} that $\mutil_i^R=\mu_i^R$. The following lemma serves as preparation.
%%%%%%%%%%%%%%%%%%%%%%%%%%%%%%%%%%%%%
\begin{lem}\label{lem:mutilRi}
  (1) Let $i\in I\setminus X$ and $a=hv$ with $h\in \cH$ and $v\in T_{w_X}^{-1}(\cR_X^{-,r})$. Then
  \begin{align}\label{eq:mutilRi}
    \mutil_i^R(a)=h\mutil_i^R(v)= -c_{\tau(i)} \frac{q^{-(\alpha_i,\theta(\alpha_i))}}{q_i-q_i^{-1}} h T^{-1}_{w_X} \circ \partial^R_{\tau(i)}\circ T_{w_X}(v)K_{\alpha_i+\theta(\alpha_i)}.
  \end{align}
  (2) For all $i\in I\setminus X$, $j\in X$ and $a\in \cA$ we have
  \begin{align}\label{eq:adr-mutil}
    \ad_r(E_j)(\mutil_i^R(a)) = \mutil_i^R(\ad_r(E_j)a).
  \end{align}  
\end{lem}
%%%%%%%%%%%%%%%%%%%%%%%%%%%%%%%%%%%%%%
\begin{proof}
  (1)  By Equation \eqref{eq:sigmaRX-} we have $\sigma \tau(v)\in \cR_X^{-,r}$. Hence we may apply Equation \eqref{eq:muiL} to obtain
  \begin{align*}
    \mutil_i^R(a)&=\sigma \tau \circ\mu_{\tau(i)}^L(\sigma \tau(v)\sigma \tau(h))\\
    &= - c_{\tau(i)}\frac{q^{-(\alpha_i,\theta(\alpha_i))}}{q_i-q_i^{-1}}\sigma \tau\Big( K_{-\alpha_{\tau(i)}-\theta(\alpha_{\tau(i)})} T_{w_X} \circ \partial^L_{i}\circ T^{-1}_{w_X}(\sigma \tau(v))\sigma \tau(h)\Big).
  \end{align*}
  Now Equation \eqref{eq:mutilRi} follows from Equations \eqref{eq:sigmaTwX} and \eqref{eq:sigma-tau-partial}.\\
  (2) For all $a\in \cA$ we have
  \begin{align}\label{eq:sigma-adr}
     \sigma \tau(\ad_r(E_j)(a))=\sigma \tau\big(K_j^{-1}(aE_j-E_ja)\big)= -K_{\tau(j)}\ad_r(E_{\tau(j)})(\sigma \tau(a))K_{\tau(j)}.
  \end{align}
  Moreover, note that
  \begin{align}\label{eq:Kj-left}
     \mu_i^L(K_j a)=q^{(\alpha_i,\alpha_j)}K_j\mu_i^L(a) \qquad \mbox{for all $j\in X$.}
  \end{align}  
  Hence we obtain
  \begin{align*}
     \mutil_i^R(\ad_r(E_j)(a))&\stackrel{\eqref{eq:sigma-adr}}{=} -\sigma \tau\circ \mu_{\tau(i)}^L
     \Big( K_{\tau(j)}\ad_r(E_{\tau(j)})(\sigma \tau(a))K_{\tau(j)} \Big)\\
     &\stackrel{\eqref{eq:Kj-left}}{=} -q^{(\alpha_i,\alpha_j)}\sigma \tau\Big(K_{\tau(j)} \mu_{\tau(i)}^L \big(\ad_r(E_{\tau(j)})(\sigma \tau(a))\big) K_{\tau(j)}\Big)\\
     &\stackrel{\eqref{eq:muL-adr}}{=} -K_j^{-1}\sigma \tau\Big( \ad_r(E_{\tau(j)})\big(\mu_{\tau(i)}^L(\sigma \tau(v))\big)\Big) K_j^{-1}\\
     &\stackrel{\eqref{eq:sigma-adr}}{=} \ad_r(E_j)\big(\sigma \tau\circ \mu_{\tau(i)}^L\circ \sigma \tau(a)\big) 
  \end{align*}  
  which completes the proof of Equation \eqref{eq:adr-mutil}.  
\end{proof}  
%%%%%%%%%%%%%%%%%%%%%%%%%%%%%%%%%%%%%%
\subsection{An explicit expression for $\mu_i^R$}
%%%%%%%%%%%%%%%%%%%%%%%%%%%%%%%%%%%%%%
The following technical lemma will simplify the proof of Theorem \ref{thm:muiR}.
%%%%%%%%%%%%%%%%%%%%%%%%%%%%%%%%%%%%%%
\begin{lem}\label{lem:adr}
  For any $n\in \N$ we have
  \begin{align}\label{eq:TwXR}
    T_{w_X}^{-1}(\cR_{X,n+1}^{-,r})\subset \sum_{\ell\in I\setminus X} U_X^0\ad_r(U_X^+)\big(T_{w_X}^{-1}(F_\ell \cR_{X,n}^{-,r})\big).
  \end{align}  
  \end{lem}  
%%%%%%%%%%%%%%%%%%%%%%%%%%%%%%%%%%%%%%%%  
\begin{proof}
  For any $\ell\in I\setminus X$ the element $F_\ell$ is a highest weight vector for the right adjoint action of $U_X$. Hence, in view of Corollary \ref{cor:adH-gen}.(4), we know that
  \begin{align}\label{eq:RXn+1}
      \sum_{\ell\in I\setminus X} \ad_r(U_X^-)(F_\ell \cR_{X,n}^{-,r})=\cR_{X,n+1}^{-,r}.
  \end{align}
  Now observe that for any $u\in \cR_X^{-,r}$ and $j\in X $ we have
  \begin{align}\label{eq:adrEjTwX-1}
     \ad_r(E_j)(T_{w_X}^{-1}(u)) = - K_j^{-1} T_{w_X}^{-1}\big(\ad_r(F_{\tau(j)})(u)\big) K_j^{-1}.
  \end{align}  
   Indeed, using the relation $T_{w_X}(E_j)=-F_{\tau(j)}K_{\tau(j)}$, see e.g.~\cite[Lemma 3.4]{a-Kolb14}, we obtain
  \begin{align*}
    \ad_r(E_j)(T_{w_X}^{-1}(u)) &= - K_j^{-1} E_j T_{w_X}^{-1}(u) + K_j^{-1} T_{w_X}^{-1}(u)E_j\\
    &=K_j^{-1} T_{w_X}^{-1}\big([u, T_{w_X}(E_j)]\big)\\
    &= K_j^{-1} T_{w_X}^{-1}\big( F_{\tau(j)} K_{\tau(j)} u - u F_{\tau(j)} K_{\tau(j)}\big)\\
    &=- K_j^{-1} T_{w_X}^{-1}\Big( \ad_r(F_{\tau(j)})(u)\Big) K_j^{-1}.
  \end{align*}
  Applying $T_{w_X}^{-1}$ to Equation \eqref{eq:RXn+1} and using Equation \eqref{eq:adrEjTwX-1} gives the inclusion \eqref{eq:TwXR}.
  \end{proof}  
%%%%%%%%%%%%%%%%%%%%%%%%%%%%%%%%%%%%%%
  Recall the definition of the map $\mutil_i^R:\cA\rightarrow \cA$ from Equation \eqref{eq:mutil-def} and the explicit description of $\mutil^R_i$ given in Lemma \ref{lem:mutilRi}.
  By definition of $\mutil_i^R$, Equation \eqref{eq:muiL-skew} gives us
  \begin{align}\label{eq:mutilR-skew}
      \mutil_i^R(v'v) =  v' \mutil_i^R(v)  +  q^{(\alpha_i,\nu)} \mutil_i^R(v')v \qquad \mbox{for all $v\in \big(T_{w_X}^{-1}(\cR_X^{-,r})\big)_{-\nu}$, $v' \in T_{w_X}^{-1}(\cR_X^{-,r})$.}
  \end{align}  
We are now ready to give the desired explicit expression for $\mu_i^R$.
%%%%%%%%%%%%%%%%%%%%%%%%%%%%%%%%%%%%%
\begin{thm}\label{thm:muiR}
  For all $i\in I\setminus X$ we have $\mu_i^R=\mutil_i^R$.
\end{thm}  
%%%%%%%%%%%%%%%%%%%%%%%%%%%%%%%%%%%%%
\begin{proof}
  We verify the equation
  \begin{align}\label{eq:muRi}
    \mu_i^R(a)=\mutil_i^R(a)
  \end{align}
  for $a\in \cA_n$ by induction on $n$. For $a\in \cA_0=\cH$ we have $\mu_i^R(a)=0=\mutil_i^R(a)$ by the 0-equivariance of the star product. Assume now that Equation \eqref{eq:muRi} holds true for all $a\in \cA_{\le n}$. By the left $\cH$-linearity of $\mu_i^R$ and $\mutil_i^R$ it suffices to prove Equation \eqref{eq:muRi} for all $a\in T_{w_X}^{-1}(\cR^{-,r}_{X,n+1})$.
Moreover, by Lemma \ref{lem:adr} and Equations \eqref{eq:muR-adr}, \eqref{eq:adr-mutil}, it suffices to verify Equation \eqref{eq:muRi} for $a=T_{w_X}^{-1}(F_{\ell})v$ where $v\in T_{w_X}^{-1}(\cR^{-,r}_{X,n})$.

    Subtracting Equation \eqref{eq:associative-2} from Equation \eqref{eq:associative-1} for $a=v$ and using the second statement of Proposition \ref{prop:muR-muL}, we obtain
    \begin{align}\label{eq:step1}
       \mu_i^R(T_{w_X}^{-1}(F_\ell) v)&=T_{w_X}^{-1}(F_\ell) \mu^R_i(v) + T_{w_X}^{-1}(\mu^L_\ell (T_{w_X}(v F_i))) - T_{w_X}^{-1}(\mu^L_\ell(T_{w_X}(v)))F_i. 
    \end{align}  
    By Equation \eqref{eq:muiL-skew} for any $v\in (T_{w_X}^{-1}(\cR_X^{-,r}))_{-\nu}$ the relation
    \begin{align*}
      \mu^L_\ell(T_{w_X}(v F_i))=\mu^L_\ell(T_{w_X}(v)) T_{w_X}(F_i) + q^{(\alpha_\ell,w_X(\nu))} T_{w_X}(v)\mu^L_\ell(T_{w_X}(F_i))
    \end{align*}
    holds. Inserting this into Equation \eqref{eq:step1} and using Lemma \ref{lem:muLR-1} we obtain
    \begin{align*}
      \mu_i^R(T_{w_X}^{-1}(F_\ell) v)&=T_{w_X}^{-1}(F_\ell) \mu^R_i(v) +  q^{(\alpha_\ell,w_X(\nu))} v \mu^R_i(T_{w_X}^{-1}(F_\ell))\\
      &=T_{w_X}^{-1}(F_\ell) \mu^R_i(v) +  q^{(\alpha_{i},\nu)}  \mu^R_i(T_{w_X}^{-1}(F_\ell)) v.
    \end{align*}
    By induction hypothesis and the skew derivation property \eqref{eq:mutilR-skew} for $\mutil_i^R$, the above equation implies that $\mu_i^R(T_{w_X}^{-1}(F_\ell) v)=\mutil_i^R(T_{w_X}^{-1}(F_\ell) v)$ which completes the proof.
\end{proof}  
%%%%%%%%%%%%%%%%%%%%%%%%%%%%%%%%%%%%%%
\begin{thm}\label{thm:sigmatau-anti}
  The map $\sigma\tau:\cA\rightarrow \cA$ is an algebra anti-automorphism of $(\cA,\ast)$.
\end{thm}
%%%%%%%%%%%%%%%%%%%%%%%%%%%%%%%%%%%%%%
\begin{proof}
  As the algebra $(\cA,\ast)$ is generated by $\cH\cup\{F_\ell\,|\,\ell\in I\setminus X\}$, it suffices to check the relation 
  \begin{align}\label{eq:sigmatau-anti}
    \sigma\tau(b\ast a)=\sigma\tau(a)\ast \sigma\tau(b)
  \end{align}
  for $b\in \cA_0=\cH$ and for $b=F_\ell$ with $\ell\in I\setminus X$ and for all $a\in \cA$. For $b\in \cH$ relation \eqref{eq:sigmatau-anti} follows from the $0$-equivariance of the star product and from the fact that $\sigma\tau$ is an algebra anti-automorphism of $\cA$ with the undeformed multiplication. For $\ell\in I\setminus X$ Theorem \ref{thm:muiR} gives us
  \begin{align*}
      \sigma\tau(F_\ell\ast a)&=\sigma\tau(F_\ell a) + \sigma\tau(\mu_\ell^L(a)))\\
      &=\sigma\tau(a)F_{\tau(\ell)} + \sigma\tau\circ \mu^L_\ell\circ\sigma\tau(\sigma\tau(a))\\
      &=\sigma\tau(a)F_{\tau(\ell)} + \mu_{\tau(\ell)}^R(\sigma\tau(a))\\
      &=\sigma\tau(a)\ast \sigma\tau(F_{\ell})
  \end{align*}
  which proves Equation \eqref{eq:sigmatau-anti} also for $b=F_\ell$.
\end{proof}
%%%%%%%%%%%%%%%%%%%%%%%%%%%%%%%%%%%%%%
We now use the Letzter map $\psi:\cB_\bc\rightarrow \cA$ to transport the algebra anti-automorphism $\sigma\tau=\sigma\circ\tau$ of $(\cA,\ast)$ to the QSP-subalgebra $\cB_\bc$. 
%%%%%%%%%%%%%%%%%%%%%%%%%%%%%%%%%%%%%%
\begin{cor}\label{cor:sigmatauB}
  The linear map $\sigma_\tau:\cB_\bc\rightarrow \cB_\bc$ defined by $\sigma_\tau=\psi^{-1} \circ \sigma\circ \tau\circ \psi$ is an algebra anti-automorphism which satisfies $\sigma_\tau|_\cH=\sigma\circ \tau|_\cH$ and $\sigma_\tau(B_i)=B_{\tau(i)}$ for all $i\in I\setminus X$.
\end{cor}
%%%%%%%%%%%%%%%%%%%%%%%%%%%%%%%%%%%%%%
\begin{proof}
  This follows from Theorem \ref{thm:sigmatau-anti} and the fact that the Letzter map $\psi:\cB_\bc\rightarrow (\cA,\ast)$ is an isomorphism of algebras which satisfies $\psi(B_i)=F_i$ for all $i\in I\setminus X$.
\end{proof}  
%%%%%%%%%%%%%%%%%%%%%%%%%%%%%%%%%%%%%%
%%%%%%%%%%%%%%%%%%%%%%%%%%%%%%%%%%%%%%
\subsection{The bar involution for $\cB_\bc$}\label{sec:Bbar}
%%%%%%%%%%%%%%%%%%%%%%%%%%%%%%%%%%%%%%
The existence of the bar involution for $\cB_\bc$ for most general parameters was established in \cite{a-Kolb22} building on the improvement \cite[Section 7]{a-AppelVlaar22} of the formulation of the quasi $K$-matrix in \cite[Section 6]{a-BalaKolb19}. Using the anti-automorphism $\sigma_\tau$ from Corollary \ref{cor:sigmatauB} we can now give a new proof of the existence of the bar involution for $\cB_\bc$, which does not rely on the prior construction of the quasi $K$-matrix.

Recall that the Lusztig braid group automorphisms $T_i$ for $i\in I$ are given in four variants $T_{i,e}', T''_{i,e}$ for $e\in \{\pm 1\}$, see \cite[37.1]{b-Lusztig94}, and in our conventions $T_i=T''_{i,1}$. In the sequel we will make use of the relations
\begin{align}
  \sigma \circ T''_{i,e}&=T'_{i,-e}\circ \sigma, \label{eq:sTTs}\\
  \obar \circ T''_{i,e} &= T''_{i,-e} \circ \obar \label{eq:obTTob}
\end{align}
which hold for all $i\in I$, $e\in \{\pm 1\}$, see \cite[37.2.4]{b-Lusztig94}. Recall that $X\subset I$ is a subset of finite type. Let $\Phi_X\subset Q$ be the corresponding finite root system, and let 
\begin{align*}
  \rho_X=\frac{1}{2} \sum_{\beta\in \Phi_X\cap Q^+} \beta, \qquad 
    \rho_X^\vee=\frac{1}{2} \sum_{\beta\in \Phi_X\cap Q^+} \beta^\vee
\end{align*}
be the corresponding half sum of positive roots and coroots, respectively. We will use the relation
\begin{align}\label{eq:T21}
  T''_{w_X,e}(E_i) = (-1)^{2\alpha_i(\rho_X^\vee)} q^{e(2\rho_X,\alpha_i)} T'_{w_X,e}(E_i) \qquad \mbox{for all $i\in I\setminus X$, $e\in \{\pm 1\}$}
\end{align}
which was established in the proof of \cite[Lemma 2.9]{a-BalaKolb15} building on \cite[37.2.4]{b-Lusztig94}. Recall that $\theta=-w_X\circ\tau$ on $Q$.
%%%%%%%%%%%%%%%%%%%%%%%%%%%%%%%%%%%%%%%
\begin{lem}\label{lem:BBt}
  For any $i\in I\setminus X$ we have $(\sigma\circ \tau\circ \obar)(B_i)=B_{\tau(i)}$ if and only if 
  \begin{align}\label{eq:Bbar-condition}
    \overline{c_i}=(-1)^{2\alpha_i(\rho_X^\vee)} q^{(2\rho_X-\theta(\alpha_i),\alpha_i)} c_{\tau(i)}.
  \end{align}  
\end{lem}
%%%%%%%%%%%%%%%%%%%%%%%%%%%%%%%%%%%%%%%
\begin{proof}
    For $i\in I\setminus X$ we calculate
    \begin{align*}
      (\sigma\circ \tau\circ \obar)(B_i) &\stackrel{\phantom{\eqref{eq:obTTob}}}{=} F_{\tau(i)} - \overline{c_i} K_{\tau(i)}^{-1}(\sigma\circ \tau\circ \obar\circ T_{w_X})(E_{\tau(i)})\\
      &\stackrel{\phantom{\eqref{eq:obTTob}}}{=} F_{\tau(i)} - \overline{c_i} K_{\tau(i)}^{-1}\sigma(\overline{T_{w_X}(E_{i})})\\
      &\stackrel{\eqref{eq:obTTob}}{=}F_{\tau(i)} - \overline{c_i} K_{\tau(i)}^{-1} \sigma({T''_{w_X,-1}(E_{i})}))\\
      &\stackrel{\eqref{eq:sTTs}}{=}F_{\tau(i)} - \overline{c_i} K_{\tau(i)}^{-1} T'_{w_X,1}(E_{i})\\
      &\stackrel{\eqref{eq:T21}}{=}F_{\tau(i)} - \overline{c_i} (-1)^{2\alpha_i(\rho_X^\vee)} q^{-(2\rho_X,\alpha_i)-(\alpha_{\tau(i)}, w_X(\alpha_i))} T_{w_X}(E_{i})K_{\tau(i)}^{-1}
    \end{align*}
    which proves the desired equivalence.
\end{proof}
%%%%%%%%%%%%%%%%%%%%%%%%%%%%%%%%%%%%%%
With this preparation we can give a quasi $K$-matrix free proof of the existence of the bar involution on $\cB_\bc$.
%%%%%%%%%%%%%%%%%%%%%%%%%%%%%%%%%%%%%%
\begin{cor}\label{cor:barB}
  Assume that $\overline{c_i}=(-1)^{2\alpha_i(\rho_X^\vee)} q^{(2\rho_X-\theta(\alpha_i),\alpha_i)} c_{\tau(i)}$ for all $i\in I\setminus X$. Then the map
  \begin{align}
      \obar^B:=\sigma_\tau \circ \sigma\circ \tau\circ \obar: \cB_\bc\rightarrow \cB_\bc, \quad b\mapsto \overline{b}^B
  \end{align}
  defines a $k$-algebra automorphism of $\cB_\bc$ such that $\obar^B|_\cH=\obar|_\cH$ and $\overline{B_i}^B=B_i$ for all $i\in I\setminus X$.
\end{cor}
%%%%%%%%%%%%%%%%%%%%%%%%%%%%%%%%%%%%%%%
\begin{proof}
   By Corollary \ref{cor:sigmatauB} and Lemma \ref{lem:BBt} the map $\sigma_\tau\circ \sigma\circ\tau\circ \obar$ defines a $k$-algebra automorphism of $\cB_\bc$ with $\sigma_\tau\circ \sigma\circ\tau\circ \obar(B_i)=B_i$ for all $i\in I\setminus X$. As $\sigma_\tau|_\cH=\sigma\circ\tau|_\cH$ we obtain $\obar^B|_\cH=\obar|_\cH$.  
\end{proof}
%%%%%%%%%%%%%%%%%%%%%%%%%%%%%%%%%%%%%%
\subsection{A noteworthy consequence of Theorem \ref{thm:muiR}}\label{sec:noteworthy}
%%%%%%%%%%%%%%%%%%%%%%%%%%%%%%%%%%%%%
As a further application of Theorem \ref{thm:muiR} we give a new proof of \cite[Conjecture 2.7]{a-BalaKolb15}.
%%%%%%%%%%%%%%%%%%%%%%%%%%%%%%%%%%%%%
  \begin{prop}\label{prop:fundamental}
  For all $i\in I\setminus X$ we have
  \begin{align*}
     \sigma \circ \tau(\partial^L_i(T_{w_X}^{-1}(F_i)))=\partial^L_i(T_{w_X}^{-1}(F_i)).
  \end{align*}  
\end{prop}  
%%%%%%%%%%%%%%%%%%%%%%%%%%%%%%%%%%%%%%%
\begin{proof}
  By Equation \eqref{eq:muiL} we have
  \begin{align}
    F_{\tau(i)}\ast F_i &= F_{\tau(i)} F_i + \mu_{\tau(i)}^L(F_i)\nonumber\\
    &=F_{\tau(i)}F_i  -c_{\tau(i)}\frac{q^{-(\alpha_i,\theta(\alpha_i))}}{q_i-q_i^{-1}} K_{-\alpha_{\tau(i)}-\theta(\alpha_{\tau(i)})} T_{w_X} \circ \partial^L_{i}\circ T^{-1}_{w_X}(F_i).\label{eq:FFL}
  \end{align}
  On the other hand, Theorem \ref{thm:muiR} and Equation \eqref{eq:mutilRi} give us
   \begin{align}
     F_{\tau(i)}\ast F_i &=F_{\tau(i)}F_i + \mu_i^R(F_{\tau(i)})\nonumber\\
     &=F_{\tau(i)}F_i  -c_{\tau(i)}\frac{q^{-(\alpha_i,\theta(\alpha_i))}}{q_i-q_i^{-1}}  T_{w_X}^{-1} \circ \partial^R_{\tau(i)}\circ T_{w_X}(F_{\tau(i)})K_{\alpha_{i}+\theta(\alpha_{i})}.\label{eq:FFR}
   \end{align}
   Comparison of Equation \eqref{eq:FFL} and Equation \eqref{eq:FFR} gives us
   \begin{align}\label{eq:main-step}
    K_{-\alpha_{\tau(i)}-\theta(\alpha_{\tau(i)})} T_{w_X} \circ \partial^L_{i}\circ T^{-1}_{w_X}(F_i)=  T_{w_X}^{-1} \circ \partial^R_{\tau(i)}\circ T_{w_X}(F_{\tau(i)})K_{\alpha_{i}+\theta(\alpha_{i})}.  
   \end{align}
   As $w_X(\alpha_i)-\alpha_i=w_X(\alpha_{\tau(i)})-\alpha_{\tau(i)}$ both $\partial^L_{i}\circ T^{-1}_{w_X}(F_i)$ and $\partial^R_{\tau(i)}\circ T_{w_X}(F_{\tau(i)})$ lie in $U^-_{-(w_X(\alpha_i)-\alpha_i)}\subset U^-_X$. Now, for any monomial $f=F_{j_1}F_{j_2}\dots F_{j_k} \in U^-_{-\alpha}$ with $\alpha=\sum_{n=1}^k\alpha_{j_n}\in Q_X^+$ we have
   \begin{align*}
     T_{w_X}(f)&=(-1)^k K_{\tau(j_1)}^{-1}E_{\tau(j_1)} K_{\tau(j_2)}^{-1}E_{\tau(j_2)}\dots  K_{\tau(j_k)}^{-1}E_{\tau(j_k)}\\
     &=q^{\frac{1}{2}((\alpha,\alpha)-\sum_{n=1}^k(\alpha_{j_n},\alpha_{j_n}))} K_{-\tau(\alpha)} \omega \circ \tau(f),\\
     T_{w_X}^{-1}(f)&=(-1)^k E_{\tau(j_1)}K_{\tau(j_1)} E_{\tau(j_2)}K_{\tau(j_2)}\dots E_{\tau(j_k)}K_{\tau(j_k)}\\
     &= q^{\frac{1}{2}((\alpha,\alpha)-\sum_{n=1}^k(\alpha_{j_n},\alpha_{j_n}))} \omega \circ \tau(f)  K_{\tau(\alpha)}.
   \end{align*}  
   Hence property \eqref{eq:main-step} is equivalent to
   \begin{align*}
      K_i K_{\tau(i)}^{-1}\partial_i^L (T_{w_X}^{-1}(F_i)) =   \partial^R_{\tau(i)}( T_{w_X}(F_{\tau(i)})) K_i K_{\tau(i)}^{-1}.
   \end{align*}  
   As $(\alpha_i-\alpha_{\tau(i)},w_X(\alpha_i)-\alpha_i)=0$ we obtain
   \begin{align*}
      \partial_i^L (T_{w_X}^{-1}(F_i)) =   \partial^R_{\tau(i)}( T_{w_X}(F_{\tau(i)})).
   \end{align*}
   Now the relation $\sigma\circ \tau\big( \partial^R_{\tau(i)}( T_{w_X}(F_{\tau(i)}))\big) = \sigma\big( \partial^R_{i}( T_{w_X}(F_{i}))\big)=\partial^L_i(T_{w_X}^{-1}(F_i))$ implies the claim of the proposition.
\end{proof}  
%%%%%%%%%%%%%%%%%%%%%%%%%%%%%%%%%%%%%%%
Recall that $\omega\circ \partial^L_i=-\partial^L_i\circ \omega$ where on the right hand side $\partial^L_i:U^+\rightarrow U^+$ while on the left hand side $\partial^L_i:U^-\rightarrow U^-$. Moreover, $\tau \circ\omega=\omega\circ \tau$, $\sigma \circ \omega=\omega \circ \sigma$ and
\begin{align*}
  \omega \circ T_{w_X}^{-1}(u) = q_{X,\mu} T_{w_X}^{-1}\circ \omega(u) \qquad \mbox{for all $u\in (\cR_X^-)_{\-\mu}$}
\end{align*}
where $q_{X,\mu}\in \qfield$ is a scalar factor which only depends on $X$ and the weight of $u$, see \cite[37.2.4]{b-Lusztig94}. Hence, applying $\omega$ and $\sigma\circ \omega$ to the formula in Proposition \ref{prop:fundamental} we obtain the following result.
%%%%%%%%%%%%%%%%%%%%%%%%%%%%%%%%%%%%%%%
\begin{cor}\label{cor:fundamental}
  For all $i\in I\setminus X$ the following relations hold:
  \begin{align}
    \sigma \circ \tau(\partial^L_i(T_{w_X}^{-1}(E_i))) &=\partial^L_i(T_{w_X}^{-1}(E_i)),\label{eq:fund1}\\
    \sigma \circ \tau(\partial^R_i(T_{w_X}(E_i))) &=\partial^R_i(T_{w_X}(E_i)).\label{eq:fund2}
  \end{align}  
\end{cor}
%%%%%%%%%%%%%%%%%%%%%%%%%%%%%%%%%%%%%%%%
\begin{rema}
  Relation \eqref{eq:fund2} is the statement of \cite[Conjecture 2.7]{a-BalaKolb15}. We refer back to Section \ref{sec:intro-fundamental} in the Introduction for the historical context and relevance of this relation.
\end{rema}
%%%%%%%%%%%%%%%%%%%%%%%%%%%%%%%%%%%%%%%
\section{The quasi $K$-matrix for QSP-subalgebras}\label{sec:quasiK}
%%%%%%%%%%%%%%%%%%%%%%%%%%%%%%%%%%%%%
Our final application of the star-product interpretation of QSP-subalgebras is the non-inductive construction of the quasi $K$-matrix, which will be presented in Section \ref{sec:RtoK}. Sections \ref{sec:relative-pairing} -- \ref{sec:alt-mu} provide preparatory results, which hold in more generality.
%%%%%%%%%%%%%%%%%%%%%%%%%%%%%%%%%%%%%
\subsection{The relative skew-Hopf pairing}\label{sec:relative-pairing}
%%%%%%%%%%%%%%%%%%%%%%%%%%%%%%%%%%%%%
Recall that $\Uq^\ge=U^0 U^+$ and $\Uq^\le=U^-U^0$ are Hopf subalgebras of $\Uq$.
By \cite[Proposition 1.2.3]{b-Lusztig94} there exists a $\qfield$-bilinear pairing
\begin{align}\label{eq:pairing}
  \langle\,,\,\rangle: U^{\le}\otimes U^\ge \rightarrow \qfield
\end{align}
such that for all $x,x'\in U^\ge, y,y'\in U^\le$ and $i,j\in I$ the following relations hold:
\begin{align}
\left<y, xx' \right>&=\left<\Delta(y), x'\otimes x \right>, &  \left<yy', x \right>&=\left<y\otimes y', \Delta(x)\right>, \label{eq:form-kow} \\
\left<K_i,K_j \right>&=q^{-(\alpha_i,\alpha_j)}, &  \left<F_i, E_j\right>&=\delta_{ij} \frac{-1}{q_i-q_i^{-1}},  \nonumber\\
\left<K_j,E_i\right>&=0,  &  \left<F_i, K_j\right>&=0.  \nonumber
\end{align}  
Here we follow the conventions used in the finite case in \cite[6.12]{b-Jantzen96}. The restriction of $\langle \, ,\, \rangle$ to $U^-_{-\mu}\ot U^+_\mu$ is non-degenerate for any $\mu\in Q^+$, while the restriction of  $\langle \, ,\, \rangle$ to $U^-_{-\mu}\ot U^+_\nu$ vanishes if $\mu\neq \nu$. As the Chevalley involution $\omega$ is an algebra homomorphism, coalgebra anti-homomorphism, the defining relations of the pairing imply that
\begin{align}\label{eq:pairing-omega}
  \langle y,x\rangle=\langle \omega(x),\omega(y)\rangle \qquad \mbox{for all $x\in U^+, u\in U^-$,}
\end{align}
see also \cite[Lemma 6.16]{b-Jantzen96}. The pairing is also invariant under the algebra anti-automorphism $\sigma$, that is 
\begin{align}\label{eq:pairing-sigma}
  \langle y,x\rangle=\langle \sigma(y),\sigma(x)\rangle \qquad \mbox{for all $x\in U^+, u\in U^-$,}
\end{align}
see \cite[Lemma 1.2.8]{b-Lusztig94}. The relations \eqref{eq:cop-partial-L}-\eqref{eq:cop-partial2-R} and \eqref{eq:form-kow} imply that
\begin{align}
  \langle F_i y,x\rangle &= \langle F_i, E_i \rangle \langle y, \partial_i^L(x)\rangle,\nonumber\\
  \langle y F_i,x\rangle &= \langle F_i, E_i \rangle \langle y, \partial_i^R(x)\rangle,\nonumber\\
   \langle y,E_ix\rangle &= \langle F_i, E_i \rangle \langle \partial_i^L(y),x\rangle,\label{eq:yEx}\\
  \langle y ,x E_i\rangle &= \langle F_i, E_i \rangle \langle \partial_i^R(y), x\rangle\label{eq:yxE}
\end{align}  
for all $x\in U^+, y\in U^-$. 

Throughout Sections \ref{sec:relative-pairing}, \ref{sec:relative-quasiR} let $X\subset I$ be any proper subset. Recall the definition of the Levi factor $\cL_X$ given in \eqref{eq:LX-def} and set $\cL_X^\ge=\cL_X\cap \Uq^\ge$ and $\cL_X^\le=\cL_X\cap \Uq^\le$.
The pairing \eqref{eq:pairing} is compatible with the triangular decompositions given in \eqref{eq:triang_Upm}.
%%%%%%%%%%%%%%%%%%%%%%%%%%%%%%%%%%%%%%%%%%%%%%%
\begin{prop}\label{prop:rhrh}
  The pairing \eqref{eq:pairing} satisfies the relation
  \begin{align}\label{eq:rhrh}
    \langle r^- h^-, r^+ h^+ \rangle = \langle r^-,r^+ \rangle \langle h^-,h^+ \rangle
  \end{align}
  for all $r^-\in \cR_X^{-,r}$,  $r^+\in \cR_X^{+,l}$, $h^-\in \cL_X^\le$ and $h^+\in \cL_X^\ge$. In particular, the restriction of the pairing \eqref{eq:pairing} to $\cR_X^{-,r}\ot \cR_X^{+,l}$ is non-degenerate. 
  Moreover, the property
  \begin{align}\label{eq:lr-adjoint}
     \langle y,\ad_l(h)(x)\rangle=\langle \ad_r(h)(y),x\rangle
  \end{align}
  holds for all $y\in \cR_X^{-,r}, x\in \cR_X^{+,l}$ and $h\in \cL_X$.
\end{prop}  
%%%%%%%%%%%%%%%%%%%%%%%%%%%%%%%%%%%%%%%%%%%%%%%%%%%%%%%%%
\begin{proof}
  For any $r^-\in \cR^{-,r}_X$ and any $h^+\in U_X^+$ we have $\langle r^-, h^+\rangle=\vep(r^-)\vep(h^+)$ for weight reasons. As $\cR_X^{-,r}$ is a right coideal of $U^{\le}$, we have
\begin{align}\label{eq:rrh1}
  \langle r^-,r^+ h^+\rangle = \langle r^-_{(1)},h^+\rangle \langle r^-_{(2)},r^+\rangle=\vep(h^+)\langle r^{-},r^+\rangle.
\end{align}
Using the fact that $\cR^{+,l}_X$ is a left coideal in the second equation below, we obtain
\begin{align*}
  \langle r^- h^-, r^+ h^+ \rangle & \stackrel{\phantom{\eqref{eq:rrh1}}}{=}\langle r^-, r^+_{(1)}h^+_{(1)}\rangle \langle h^-, r^+_{(2)}h^+_{(2)} \rangle\\
  &\stackrel{\phantom{\eqref{eq:rrh1}}}{=}\langle r^-, r^+ h^+_{(1)}\rangle \langle h^-, h^+_{(2)}\rangle \\
  & \stackrel{\eqref{eq:rrh1}}{=}\langle r^-,r^+ \rangle \langle h^-,h^+ \rangle
\end{align*}
which proves Equation \eqref{eq:rhrh}.
Now the second isomorphism in \eqref{eq:triang_Upm} and the non-degeneracy of the pairing \eqref{eq:pairing} on $U^-\ot U^+$ imply that there exists $r^+\in \cR^+_X$ such that
$\langle r^-,r^+\rangle \neq 0$. This and the invariance properties \eqref{eq:pairing-omega}, \eqref{eq:omegaR} imply that the restriction of the pairing \eqref{eq:pairing} to $\cR^{-,r}_X\otimes \cR^{+,l}_X$ is non-degenerate.

  It suffices to verify Equation \eqref{eq:lr-adjoint} for $h=E_j, F_j, K_i^{\pm 1}$ for $i\in I, j\in X$. Indeed, if the equation holds for $h_1, h_2\in \cL_X$ and all $x\in \cR_X^+,y\in \cR_X^-$ then it also holds for $h=h_1h_2$. For $h=K_i^{\pm 1}$ the equation holds for weight reasons. For $j\in X$ consider $h=E_j$ and let $x\in (\cR_X^{+,l})_\mu$ and $y\in (\cR_X^{-,r})_{-\mu-\alpha_j}$. Using Equations \eqref{eq:yEx}, \eqref{eq:yxE} we calculate
  \begin{align}
    \langle y, \ad_l(E_j)(x)\rangle &=\langle y, E_jx - q^{(\alpha_j,\mu)}x E_j\rangle\nonumber\\
    &=\langle F_j, E_j \rangle \langle \partial^L_j(y)-q^{(\alpha_j,\mu)}\partial^R_j(y),x\rangle\nonumber\\
    &=- \frac{1}{q_j-q_j^{-1}}\langle  \partial^L_j(y)K_j- q^{(\alpha_j,\mu)}\partial^R_j(y)K_j^{-1},x\rangle\nonumber\\
    &=  - q^{(\alpha_j,\mu)}\langle E_j y- y E_j, x \rangle\nonumber\\
    &= \langle -K_j^{-1} E_j y+ K_j^{-1}y E_j, x \rangle\nonumber\\
    &= \langle \ad_r(E_j)(y), x \rangle. \label{eq:Ei-adj}
  \end{align}
  The relation
   \begin{align*}
    \langle y, \ad_l(F_j)(x)\rangle &= \langle \ad_r(F_j)(y), x \rangle
   \end{align*}
   for $j\in X$ can be verified by an analogous calculation or, alternatively, it follows from \eqref{eq:pairing-omega} and \eqref{eq:Ei-adj}.
\end{proof}
%%%%%%%%%%%%%%%%%%%%%%%%%%%%%%%%%%%%%%%%%%%%%%%
There are two right actions of $\cL_X^\ge $ on $\cR_X^{-,r}$. On the one hand, $\cL_X^\ge$ acts on $\cR_X^{-,r}$ by the right adjoint action. On the other hand, as $\cR_X^{-,r}\subset U^\le$ is a right coideal subalgebra, we can define an action $\cR_X^{-,r}\ot \cL_X^\ge \rightarrow \cR_X^-$, $a\ot u\mapsto a\ract u$ by
\begin{align*}
  a \ract u = a_{(1)}\langle a_{(2)},u\rangle \qquad \mbox{for all } a\in \cR_X^{-,r}, u\in \cL_X^\ge.
\end{align*}
By the following lemma, the two actions coincide.
%%%%%%%%%%%%%%%%%%%%%%%%%%%%%%%%%%%%%%%%
\begin{lem}\label{lem:adrua}
  For all $a\in \cR_X^{-,r}$ and $u\in \cL_X^\ge$ we have
  \begin{align}\label{eq:2actions}
     \ad_r(u)(a)=a_{(1)}\langle a_{(2)},u\rangle.
  \end{align}
\end{lem}
%%%%%%%%%%%%%%%%%%%%%%%%%%%%%%%%%%%%%%
\begin{proof}
  As we are comparing two right actions, it suffices to verify Equation \eqref{eq:2actions} for $u=K_\mu$ with $\mu\in Q$ and for $u=E_j$ with $j\in X$.
  For $\nu, \mu\in Q$ and $a\in (\cR_X^{-,r})_{-\mu}$ we have
  \begin{align*}
    a\ract K_\nu=a_{(1)}\langle a_{(2)},K_\nu\rangle
       \stackrel{\eqref{eq:cop-partial2-L}}{=} a\langle K_{-\mu},K_{\nu} \rangle = q^{(\mu,\nu)} a = \ad_r(K_\nu)(a).
  \end{align*}
  For $u=E_j$ with $j\in X$ and $a\in (\cR_X^{-,r})_{-\mu}$ we calculate
  \begin{align*}
    a\ract E_j=a_{(1)}\langle a_{(2)},E_j\rangle 
    \stackrel{\eqref{eq:cop-partial2-L}}{=} \partial_j^L(a)\langle F_j K_{-(\mu-\alpha_j)}, E_j\rangle= \partial_j^L(a)\langle F_j, E_j\rangle=\frac{-\partial_j^L(a)}{q_j-q_j^{-1}}.
  \end{align*}
  On the other hand, the relation $\partial^R_j(a)=0$ for $j\in X$ implies that
  \begin{align*}
    \ad_r(E_j)(a)=-K_j^{-1}(E_j a-a E_j)\stackrel{\eqref{eq:EyyE}}{=}\frac{-\partial_j^L(a)}{q_j-q_j^{-1}}.
  \end{align*}
  Comparison with the previous equation yields $a\ract E_j=\ad_r(E_j)(a)$, as desired.
\end{proof}
%%%%%%%%%%%%%%%%%%%%%%%%%%%%%%%%%%%%%%
\begin{cor}\label{cor:RL}
  For all $a\in \cR_X^{-,r}$ and $m\in \cL_X^\ge$ we have
  \begin{align}
     a m = m_{(1)} a_{(1)}\langle a_{(2)},m_{(2)}\rangle.
  \end{align}  
\end{cor}
%%%%%%%%%%%%%%%%%%%%%%%%%%%%%%%%%%%%%%
\begin{proof}
  By the previous lemma we have
  \begin{align*}
      a m = m_{(1)}\ad_r(m_{(2)})(a) = m_{(1)} (a\ract m_{(2)})
  \end{align*}
  which proves the corollary.
\end{proof}
%%%%%%%%%%%%%%%%%%%%%%%%%%%%%%%%%%%%%%%%%%%%%%%
\subsection{The relative quasi $R$-matrix}\label{sec:relative-quasiR}
%%%%%%%%%%%%%%%%%%%%%%%%%%%%%%%%%%%%%%%%%%%%%%%
For $\mu\in Q^+$ set $r(\mu)=\dim(U^+_\mu)$ and let $\{u^+_{\mu,i}\,|\,i=1,\dots,r(\mu)\}$ be a basis of $U^+_\mu$ with dual basis $\{u^-_{\mu,i}\,|\,i=1,\dots,r(\mu)\}\subset U^-_{-\mu}$. Define $\Theta_\mu=\sum_{i=1}^{r(\mu)} u^-_{\mu,i}\ot u^+_{\mu,i}$. The quasi $R$-matrix of $\Uq$ is defined by
\begin{align}\label{eq:Theta-def}
  \Theta=\sum_{\mu\in Q^+} \Theta_\mu
\end{align}
as an element of the algebra $\mathscr{U}^{(2)}$, see \eqref{eq:scrUn}.
By \cite[Theorem 4.1.2]{b-Lusztig94}, the quasi $R$-matrix satisfies the relation
\begin{align}\label{eq:Theta-intertwiner}
  \kow(\ubar)\cdot \Theta=\Theta\cdot (\obar \ot \obar)\circ\kow(u) \qquad \mbox{for all $u\in \Uq$,}
\end{align}  
see also \cite[Lemma 7.1]{b-Jantzen96}. Define $\overline{\Theta}=(\obar\ot\obar)(\Theta)$. By construction, $\Theta$ is invertible in $\mathscr{U}^{(2)}$, and $\Theta$ is uniquely determined by the intertwiner property \eqref{eq:Theta-intertwiner} and the normalization $\Theta_0=1\ot 1$, see \cite[Theorem 4.1.2]{b-Lusztig94}. These two properties together imply that $\Theta^{-1}=\overline{\Theta}$.

For any subset $X\subset I$ let $\Theta_X$ denote the quasi $R$-matrix of $\uqgX\subset \uqg$ and denote its inverse by $\overline{\Theta}_X=(\obar\ot\obar)(\Theta_X)$.

Recall Proposition \ref{prop:rhrh}. For any $\mu\in Q^+$ set $s(\mu)=\dim((\cR_X^{+,l})_\mu)=\dim((\cR_X^{-,r})_{-\mu})$ and let $\{e_{\mu,i}\,|\,i=1\dots,s(\mu)\}$ be a basis of $(\cR_X^{+,l})_\mu$ with dual basis  $\{f_{\mu,i}\,|\,i=1\dots,s(\mu)\}\subset(\cR_X^{-,r})_{-\mu}$ for the pairing \eqref{eq:pairing}. Define $\Lambda_{X,\mu}=\sum_{i=1}^{s(\mu)} f_{\mu,i}\ot e_{\mu,i}$ and set
\begin{align*}
  \Lambda_X=\sum_{\mu\in Q^+} \Lambda_{X,\mu}\in \mathscr{U}^{(2)}.
\end{align*}  
We refer to $\Lambda_X$ as the relative quasi $R$-matrix corresponding to the subset $X\subset I$.
%%%%%%%%%%%%%%%%%%%%%%%%%%%%%%%%%%%%%%%%%%%%%%%%%%%%%%%%
\begin{lem}\label{lem:rel-quasi}
\hfill
  \begin{enumerate}
    \item $\Lambda_X=\Theta \overline{\Theta}_X$.
    \item For all $h\in \cL_X$ the relation $\kow(h)\Lambda_X=\Lambda_X \kow(h)$
      holds in $\mathscr{U}^{(2)}$.
  \end{enumerate}  
\end{lem}  
%%%%%%%%%%%%%%%%%%%%%%%%%%%%%%%%%%%%%%%%%%%%%%%%%%%%%%%%
\begin{proof}
  (1) The isomorphisms \eqref{eq:triang_Upm} together with Proposition \ref{prop:rhrh} imply that $\Lambda_X\Theta_X=\Theta$. We obtain the desired formula by multiplication by $\overline{\Theta}_X$ from the right.\\
  (2) This statement follows from Equation \eqref{eq:Theta-intertwiner}.
\end{proof}  
%%%%%%%%%%%%%%%%%%%%%%%%%%%%%%%%%%%%%%%%%%%%%%%%%%%%%%%%

%%%%%%%%%%%%%%%%%%%%%%%%%%%%%%%%%%%%%
\subsection{An alternative description of $\mu_i^L$}\label{sec:alt-mu}
%%%%%%%%%%%%%%%%%%%%%%%%%%%%%%%%%%%%%
The calculations in this section are valid for any pair $(X,\tau)$ where $X\subset I$ is of finite type and $\tau:I\rightarrow I$ is a diagram automorphism satisfying properties (1), (2) of a generalized Satake diagram given in Section \ref{sec:QSP}. In this generality, we take \eqref{eq:muiL} as the definition of $\mu^L_i$ and we set $\mu_i^R=\sigma \tau\circ \mu^L_{\tau(i)}\circ \sigma \tau$ for $i\in I\setminus X$. In the following we rewrite formula \eqref{eq:muiL} for $\mu_i^L(a)$ where $a\in \cR_X^{-,r}$. We first consider the case that $a\in \cR_{X,1}^{-,r}$ is of degree 1.
%%%%%%%%%%%%%%%%%%%%%%%%%%%%%%%%%%%%%
\begin{lem}\label{lem:TETE-left}
  For all $a\in \cR_{X,1}^{-,r}$ and all $i\in I\setminus X$ the following relation holds:
  \begin{align*}
      T_{w_X}(E_i)_{(1)} \langle a, T_{w_X}(E_i)_{(2)}\rangle = -\frac{1}{q_i-q_i^{-1}} K_{w_X(\alpha_i)} T_{w_X} \circ \partial^L_i \circ T_{w_X}^{-1}(a).
  \end{align*}  
\end{lem}
%%%%%%%%%%%%%%%%%%%%%%%%%%%%%%%%%%%%%
\begin{proof}
  Recall that we may consider $T_{w_X}$ as an element of the algebra $\mathscr{U}=\End({\mathcal For})$ and that as such $T_{w_X}(E_i)=T_{w_X}\cdot E_i\cdot T_{w_X}^{-1}$ and
  \begin{align*}
     \kow(T_{w_X})=(T_{w_X}\ot T_{w_X})\cdot \Theta_X^{-1},
  \end{align*}
  see for example \cite[Lemma 3.9]{a-BalaKolb19}. Hence
  \begin{align*}
    \kow(T_{w_X}(E_i))&=(T_{w_X}\ot T_{w_X})\big( \Theta_X^{-1}(E_i\ot 1 + K_i\ot E_i)\Theta_X\big)\\
    &=T_{w_X}(E_i)\ot 1 + T_{w_X}\ot T_{w_X}\big(\Theta_X^{-1}(K_i\ot E_i) \Theta_X\big).
  \end{align*}
  Writing formally $\Theta_X=\Theta_{X,1}\ot \Theta_{X,2}$ and $\Theta_X^{-1}=\Theta^-_{X,1}\ot \Theta^-_{X,2}$ we obtain
  \begin{align*}
    & T_{w_X} \ot T_{w_X}(\Theta_X^{-1}(K_i\ot E_i)\Theta_X)=T_{w_X}(\Theta^-_{X,1} K_i \Theta_{X,1})\ot T_{w_X}(\Theta^-_{X,2} E_i \Theta_{X,2})\\
    &=T_{w_X}(\Theta^-_{X,1} K_i \Theta_{X,1})\ot \underbrace{T_{w_X}(\Theta^-_{X,2}) (T_{w_X}(\Theta_{X,2}))_{(2)}}_{\in \cH^\le} \underbrace{\ad_l\big(S^{-1}\big((T_{w_X}(\Theta_{X,2}))_{(1)}\big)(T_{w_X}(E_i))}_{\in \cR_X^{+,l}}.
  \end{align*} 
  As $\kow(T_{w_X}(E_i))\in U^\ge \ot \cR_X^{+,l}$ the first factor in the second tensor leg must be trivial, and we obtain
  \begin{align*}
    T_{w_X}\ot T_{w_X}\big(\Theta_X^{-1} (K_i\ot E_i) \Theta_X\big)
   &=T_{w_X}(K_i \Theta_{X,1}) \ot \ad_l\big(S^{-1}(T_{w_X}(\Theta_{X,2}))\big)(T_{w_X}(E_i)).
  \end{align*}
  For $a\in \cR_{X,1}^{-,r}$ we hence obtain
  \begin{align}
    T_{w_X}(E_i)_{(1)}\langle a,T_{w_X}(E_i)_{(2)}\rangle
    &=T_{w_X}(K_i \Theta_{X,1}) \langle a, \ad_l\big(S^{-1}(T_{w_X}(\Theta_{X,2}))\big)(T_{w_X}(E_i))\rangle \nonumber\\
     &=T_{w_X}(K_i \Theta_{X,1}) \langle T_{w_X}^{-1}\big(\ad_r(S^{-1}\circ T_{w_X}(\Theta_{X,2}))(a)\big),E_i\rangle \label{eq:sig-om-tau}
  \end{align}
  where we used Equation \eqref{eq:lr-adjoint} and \cite[Proposition 38.2.1 a)]{b-Lusztig94} in the final step. Note that $S^{-1}\circ T_{w_X}:\Uq_X\rightarrow\Uq_X$ is an algebra anti-automorphism defined by 
  \begin{align*}
     S^{-1}\circ T_{w_X}(E_j)=F_{\tau(j)}, \quad  S^{-1}\circ T_{w_X}(F_j)=E_{\tau(j)}, \quad   S^{-1}\circ T_{w_X}(K_j)=K_{\tau(j)}
  \end{align*}
  for all $j\in X$. Now observe that for $u\in\cR^{-,r}_{X}$ and $j\in X$ we have
  \begin{align*}
    T_{w_X}^{-1}\big(\ad_r(F_{\tau(j)})(u)\big)&=T_{w_X}^{-1}\big(-F_{\tau(j)}K_{\tau(j)}u K_{\tau(j)}^{-1}+uF_{\tau(j)}\big)\\
    &=[E_j,T^{-1}_{w_X}(u)]K_j\\
    &=\langle F_j,E_j\rangle \partial_j^R(T_{w_X}^{-1}(u)).
  \end{align*}
  Inserting this into \eqref{eq:sig-om-tau} we obtain inductively
  \begin{align*}
    T_{w_X}(E_i)_{(1)}\langle a,  T_{w_X}(E_i)_{(2)}\rangle &\stackrel{\eqref{eq:yxE}}{=} T_{w_X}(K_i \Theta_{X,1}) \langle T_{w_X}^{-1}(a), E_i\Theta_{X,2}\rangle\\
    &\stackrel{\eqref{eq:yEx}}{=} T_{w_X}(K_i \Theta_{X,1})\langle F_i, E_i \rangle \langle \partial_i^L T_{w_X}^{-1}(a),\Theta_{X,2}\rangle\\
    &\stackrel{\phantom{\eqref{eq:yEx}}}{=} -\frac{1}{q_i-q_i^{-1}} K_{w_X(\alpha_i)} T_{w_X} \circ\partial^L_i\circ T_{w_X}^{-1}(a)
  \end{align*}
  as desired. 
\end{proof}  
%%%%%%%%%%%%%%%%%%%%%%%%%%%%%%%%%%%%%%
To extend Lemma \ref{lem:TETE-left} to higher degree, recall the element
$M_i$ defined by \eqref{eq:Mi-def}. It is convenient to write the element $M_i$ for $i\in I$ formally as
\begin{align}\label{eq:Mi-notation}
  M_i=\sum_{j} m_{i,1}^j\ot m_{i,2}^j.
\end{align}
With this notation we can now give the desired alternative description of the element $\mu_i^L(a)$ for $a\in \cR_X^{-,r}$.
%%%%%%%%%%%%%%%%%%%%%%%%%%%%%%%%%%%%%
\begin{thm}\label{thm:muiLa-new}
  For all $a\in \cR_X^{-,r}$ and all $i\in I\setminus X$ the following relation holds:
  \begin{align}\label{eq:muiLa-new}
     \mu_i^L(a) = c_i \sum_j m^j_{i,1} a_{(1)} \langle a_{(2)},m^j_{i,2}\rangle.
  \end{align}  
\end{thm}
%%%%%%%%%%%%%%%%%%%%%%%%%%%%%%%%%%%%%
\begin{proof}
  If $a\in \qfield$ then $\mu_i^L(a)=0$, and the right hand side of Equation \eqref{eq:muiLa-new} also vanishes by Lemma \ref{lem:Mi}.(1). For $a\in \cR_{X,1}^{-,r}$ we may assume that $a=\ad_r(h)(F_\ell)$ for some $h\in \cH$ and $\ell\in I\setminus X$. Then, similarly to Equation \eqref{eq:kow-adhE}, we have
  \begin{align}\label{eq:kowa1}
    \kow(a)= \ad_r(h_{(2)})(F_\ell) \ot S(h_{(1)}) K_{\ell}^{-1} h_{(3)} + 1 \ot \ad_r(h)(F_\ell).
  \end{align}
  Using this, Equation \eqref{eq:muiL}, and Lemmas \ref{lem:TETE-left} and \ref{lem:Mi}.(1) we obtain
  \begin{align*}
    \mu_i^L(a) &\stackrel{\phantom{\eqref{eq:kowa1}}}{=} c_i q^{-(\alpha_i,\theta(\alpha_i))} K_{-\alpha_i} T_{w_X}(E_{\tau(i)})_{(1)}\langle a, T_{w_X}(E_{\tau(i)})_{(2)} \rangle\\
    &\stackrel{\eqref{eq:kowa1}}{=} c_i \sum_j  m^j_{i,1} a_{(1)} \langle a_{(2)},m^j_{i,2}\rangle
  \end{align*}
  which completes the proof of Equation \eqref{eq:muiLa-new} for $a\in \cR_{X,1}^{-,r}$.\\
  For the general case, consider the map $\mutil^L_i:\cR_X^{-,r}\rightarrow \cA$ defined by
  \begin{align*}
    \mutil^L_{i}(a)= c_i \sum_j m^j_{i,1} a_{(1)} \langle a_{(2)},m^j_{i,2}\rangle.
  \end{align*}  
  As the algebra $\cR_X^{-,r}$ is generated in degree one, it suffices to establish the skew derivation property \eqref{eq:muiL-skew} also for the map $\mutil^L_i$. To this end, let  $a\in (\cR_X^{-,r})_{-\mu}$ and $b\in \cR_X^{-,r}$ and use the second statement of Lemma \ref{lem:Mi} to obtain 
  \begin{align*}
    \mutil_i^L(ab) &= c_i\sum_j m^j_{i,1} a_{(1)}b_{(1)} \langle a_{(2)}b_{(2)},m^j_{i,2}\rangle\\
    &= c_i\sum_j m^j_{i,1} a_{(1)}b \langle a_{(2)},m^j_{i,2}\rangle +
       c_i\sum_j (m^j_{i,1})_{(1)} a_{(1)} \langle a_{(2)},K_i(m^j_{i,1})_{(2)} \rangle b_{(1)} \langle b_{(2)},m^j_{i,2}\rangle.
  \end{align*}
  By Corollary \ref{cor:RL} and Lemma \ref{lem:adrua} this can be rewritten as
  \begin{align*}
    \mutil_i^L(ab)&=  c_i\sum_j m^j_{i,1} a_{(1)}b \langle a_{(2)},m^j_{i,2}\rangle +
    c_i K_i^{-1}a K_i\sum_j m^j_{i,1} b_{(1)} \langle b_{(2)},m^j_{i,2}\rangle\\
    &=\mutil_i^L(a) b + q^{(\alpha_i,\mu)} a \mutil^L_i(b).
  \end{align*}
  Hence, $\mutil_i^L$ does indeed satisfy the skew-derivation property \eqref{eq:muiL-skew}.
\end{proof}
%%%%%%%%%%%%%%%%%%%%%%%%%%%%%%%%%%%%%
Equation \eqref{eq:muiLa-new} can be compactly rewritten in terms of the quasi $R$-matrix.
%%%%%%%%%%%%%%%%%%%%%%%%%%%%%%%%%%%%%
\begin{cor}\label{cor:muiTheta}
  For all $i\in I\setminus X$ the following relations hold:
  \begin{align}
    (\mu_i^L\ot \id)(\Theta) &= c_i M_i\cdot\Theta,\label{eq:muiL-Theta}\\
    (\mu_i^R\ot \id)(\Theta) &=c_{\tau(i)} \Theta\cdot (\sigma \ot \sigma)(M_i).\label{eq:muiR-Theta}
  \end{align}  
\end{cor}
%%%%%%%%%%%%%%%%%%%%%%%%%%%%%%%%%%%%%
\begin{proof}
  We first prove Equation \eqref{eq:muiL-Theta}. Recall Lemma \eqref{lem:rel-quasi}. As $\mu_i^L$ is right $\cH$-linear, it suffices to prove that the relative quasi $R$-matrix $\Lambda_X=\Theta \overline{\Theta}_X$ satisfies the relation
  \begin{align}
     (\mu_i^L\ot \id)(\Lambda_X) &= c_i M_i\cdot\Lambda_X.\label{eq:muiL-Lambda}
  \end{align}
   Suppressing one summation, we write the relative quasi $R$-matrix formally as $\Lambda_X=\sum_{\mu\in Q^+} f_\mu\ot e_\mu$.
  Using Theorem \ref{thm:muiLa-new} and the fact that $\cR_X^{-,r}$ is a right coideal subalgebra of $U^\le$ we may calculate for any $a\in \cR_X^{-,r}$ as follows  
  \begin{align*}
    \mu_i^L(a) &= c_i\sum_j m_{i,1}^j a_{(1)}\langle a_{(2)},m_{i,2}^j\rangle\\
    &= c_i\sum_{j,\mu} m_{i,1}^j f_\mu \langle a_{(1)},e_\mu\rangle \langle a_{(2)},m_{i,2}^j\rangle\\
    &=c_i \sum_{j,\mu} m_{i,1}^j f_\mu \langle a, m_{i,2}^j e_\mu\rangle.
  \end{align*}  
  As $m_{i,2}^j\in \cR_X^{+,l}$ we obtain
  \begin{align*}
    \sum_{\mu'}\mu_i^L(f_{\mu'})\ot e_{\mu'} = c_i\sum_{j,\mu,\mu'} m_{i,1}^j f_\mu \langle f_{\mu'}, m_{i,2}^j e_\mu\rangle\ot e_{\mu'} =c_i \sum_{j,\mu} m_{i,1}^j f_\mu \ot m_{i,2}^j e_\mu
  \end{align*}  
  which proves Equation \eqref{eq:muiL-Lambda}.

  To prove Equation \eqref{eq:muiR-Theta}, recall from Theorem \ref{thm:muiR} that $\mu_i^R=\sigma\tau \circ \mu_{\tau(i)}^L\circ \sigma\tau$ and from Equation \eqref{eq:pairing-sigma} that $(\sigma\circ\sigma)(\Theta)=\Theta$. Hence, for any $i\in I\setminus X$, we obtain
  \begin{align*}
    (\mu_i^R\ot \id)(\Theta) &\stackrel{\phantom{\eqref{eq:muiL-Theta}}}{=} (\mu_i^R\ot \id)\big((\sigma\tau\ot\sigma\tau)(\Theta)\big)\\
 &\stackrel{\phantom{\eqref{eq:muiL-Theta}}}{=}(\sigma\tau\ot\sigma\tau)\circ (\mu_{\tau(i)}^L \ot \id)(\Theta)\\
    &\stackrel{\eqref{eq:muiL-Theta}}{=} (\sigma\tau\ot \sigma\tau)\big(c_{\tau(i)}M_{\tau(i)}\Theta\big)\\
     &\stackrel{\phantom{\eqref{eq:muiL-Theta}}}{=} c_{\tau(i)} \Theta\cdot\big((\sigma\ot \sigma)(M_i)\big)
  \end{align*}
  which gives the desired formula.
\end{proof}  
%%%%%%%%%%%%%%%%%%%%%%%%%%%%%%%%%%%%%
\subsection{The pullback of the quasi $R$-matrix}\label{sec:RtoK}
%%%%%%%%%%%%%%%%%%%%%%%%%%%%%%%%%%%%%
We now return to the setting of generalized Satake diagrams $(X,\tau)$ with corresponding QSP-subalgebra $\cB_\bc$ depending on a set of parameters $\bc$. Recall that the Letzter map $\psi:\cB_\bc\rightarrow \cA$ defined by \eqref{eq:Letzter-def} is a linear isomorphism. Recall the quasi $R$-matrix $\Theta\in \prod_{\mu\in Q^+} U^-_{-\mu}\ot U^+_\mu$ defined by \eqref{eq:Theta-def} and set
\begin{align}\label{eq:TB-def}
  \Theta^B=(\psi^{-1}\ot \id)(\Theta) \in \prod_{\mu\in Q^+} \cB_\bc\ot U^+_\mu.
\end{align}
Similarly to the quasi $R$-matrix, we write $\Theta^B$ as an infinite sum $\Theta^B=\sum_{\mu\in Q^+} \Theta^B_\mu$ with $\Theta^B_\mu\in \cB_\bc\ot U^+_\mu$. Note that by construction $\Theta^B_0=1\ot 1$.
We now show that $\Theta^B$ satisfies the defining property of the tensor quasi $K$-matrix for $\cB_\bc$. Recall that we consider $\prod_{\mu\in Q^+} \cB_\bc\ot U^+_\mu$ as a subalgebra of $\mathscr{U}^{(2)}$. Also recall the algebra anti-automorphism $\sigma_\tau$ of $\cB_\bc$ defined in Corollary \ref{cor:sigmatauB}.
%%%%%%%%%%%%%%%%%%%%%%%%%%%%%%%%%%%%%
\begin{thm}\label{thm:intertwiner2}
  For all $b\in \cB_\bc$ the relation
  \begin{align}\label{eq:intertwiner2}
    \kow(\sigma_\tau(b)) \cdot\Theta^B = \Theta^B\cdot (\sigma_\tau\ot(\sigma\circ\tau))\circ \kow(b)
  \end{align}
  holds in $\mathscr{U}^{(2)}$.
\end{thm}  
%%%%%%%%%%%%%%%%%%%%%%%%%%%%%%%%%%%%%
\begin{proof}
  Both $\kow\circ \sigma_\tau$ and $(\sigma_\tau\ot(\sigma\circ\tau))\circ \kow$ are algebra anti-homomorphisms $\cB_\bc\rightarrow \cB_\bc\ot \Uq$. Hence, it suffices to verify the relation \eqref{eq:intertwiner2} for $b=h\in \cH$ and for $b=B_{\tau(i)}$ where $i\in I\setminus X$. For $h\in \cH$ the $0$-equivariance of the star product implies that
  \begin{align*}
    \Theta^B\cdot (\sigma_\tau\ot(\sigma\circ\tau))\circ \kow(h)& \stackrel{\phantom{\eqref{eq:sigmakow}}}{=}\Theta^B \cdot(\sigma \ot \sigma)\circ \kow(\tau(h))\\
    &\stackrel{\eqref{eq:sigmakow}}{=}(\psi^{-1}\ot \id)\left(\Theta \cdot (\barx\ot \barx) \circ \kow(\sigma \circ \tau(\overline{h}))\right)\\
    &\stackrel{\eqref{eq:Theta-intertwiner}}{=}(\psi^{-1} \ot \id) \left( \kow(\sigma\circ \tau(h))\cdot\Theta \right)\\
    &\stackrel{\phantom{\eqref{eq:sigmakow}}}{=}\kow(\sigma_\tau(h))\cdot \Theta^B
  \end{align*}
  which proves \eqref{eq:intertwiner2} for $b=h\in \cH$.

  To prove Equation \eqref{eq:intertwiner2} for $b=B_{\tau(i)}$ where $i\in I\setminus X$, note that the intertwiner property \eqref{eq:Theta-intertwiner} of $\Theta$ gives us
  \begin{align*}
     (F_i\ot K_i^{-1} + 1 \ot F_i)\cdot\Theta = \Theta \cdot(F_i\ot K_i+1\ot F_i).
\end{align*}
Application of the inverse of the Letzter map to the first tensor factor then gives
\begin{align}
  (B_i\ot K_i^{-1} + 1 \ot F_i)&\cdot\Theta^B - (1\ot K_i^{-1})(\psi^{-1}\circ \mu_i^L\ot \id)(\Theta) \label{eq:kowB-intertwine}\\
  &= \Theta^B \cdot(B_i\ot K_i+1\ot F_i) - (\psi^{-1}\circ \mu_i^R\ot \id)(\Theta)(1\ot K_i).\nonumber
\end{align}
By Corollary \ref{cor:muiTheta}, Lemma \ref{lem:Mi}.(1) and the $\cH_{X,\tau}$-bilinearity of the Letzter map we have
\begin{align*}
  (\psi^{-1}\circ \mu_i^L\ot \id)(\Theta)&=c_i M_i\cdot \Theta^B,\\
  (\psi^{-1}\circ \mu_i^R\ot \id)(\Theta)&=c_{\tau(i)} \Theta^B \cdot(\sigma\ot\sigma)(M_i).
\end{align*}
Inserting these equations into Equation \eqref{eq:kowB-intertwine} we obtain
\begin{align*}
  \big(B_i\ot K_i^{-1} + 1 \ot F_i &- c_i(1\ot K_i^{-1})M_i\big)\cdot\Theta^B =\\
  &\Theta^B \cdot\big(B_i\ot K_i + 1 \ot F_i - c_{\tau(i)}(\sigma\ot \sigma)(M_i)(1\ot K_i)\big).
\end{align*}
This equation can be rewritten as
\begin{align*}
  \kow\big(\sigma_\tau(B_{\tau(i)})\big)\cdot\Theta^B =\Theta^B\cdot(\sigma_\tau\ot(\sigma\circ\tau))\circ \kow(B_{\tau(i)})
\end{align*}
which is Equation \eqref{eq:intertwiner2} for $b=B_{\tau(i)}$.
\end{proof}  
%%%%%%%%%%%%%%%%%%%%%%%%%%%%%%%%%%%%%%
We refer to the element $\Theta^B$ defined by \eqref{eq:TB-def} as the tensor quasi $K$-matrix corresponding to the QSP-subalgebra $\cB_\bc$. The tensor quasi $K$-matrix is uniquely determined by the intertwiner property \eqref{eq:intertwiner2} and the normalization $\Theta^B_0=1\ot 1$. The proof of the following result is an adaption of \cite[Proposition 3.10, Remark 3.11]{a-Kolb20} to the present setting.
%%%%%%%%%%%%%%%%%%%%%%%%%%%%%%%%%%%%%
\begin{prop}\label{prop:quasiK-unique}
  Let $\Gamma=\sum_{\mu\in Q^+}\Gamma_\mu \in \prod_{\mu\in Q^+} \Uq\ot \Uq^+_\mu$ be an element with $\Gamma_0=1\ot 1$ satisfying the relation
  \begin{align}\label{eq:intertwiner-Gamma}
    \kow(\sigma_\tau(b)) \cdot\Gamma = \Gamma\cdot (\sigma_\tau\ot(\sigma\circ\tau))\circ \kow(b) \qquad \mbox{for all $b\in \cB_\bc$}
  \end{align}
  in $\mathscr{U}^{(2)}$. Then $\Gamma=\Theta^B$.
\end{prop}
%%%%%%%%%%%%%%%%%%%%%%%%%%%%%%%%%%%%%%
\begin{proof}
    For $i\in I\setminus X$ and $b=B_{\tau(i)}$ the relation \eqref{eq:intertwiner-Gamma} can be rewritten as
    \begin{align}\label{eq:int-Gam-explicit}
        \big(B_i\ot K_i^{-1} + 1 \ot F_i &- c_i(1\ot K_i^{-1})M_i\big)\cdot\Gamma =\\
  &\Gamma \cdot\big(B_i\ot K_i + 1 \ot F_i - c_{\tau(i)}(\sigma\ot \sigma)(M_i)(1\ot K_i)\big).\nonumber
    \end{align}
    This relation also holds for $i\in X$ if we set $B_i=F_i$ and $M_i=0$ in this case. The relations \eqref{eq:int-Gam-explicit} for $i\in I$ can be rewritten as 
    \begin{align}
        (1\ot \partial^L_i)(\Gamma) &= -(q_i-q_i^{-1}) (B_i\ot 1 - c_i M_i)\cdot \Gamma,\label{eq:partialL-Gam}\\
        (1\ot \partial^R_i)(\Gamma) &= -(q_i-q_i^{-1}) \Gamma\cdot (B_i\ot 1 - c_{\tau(i)}(\sigma\ot \sigma)(M_i)).\nonumber
    \end{align}
    By \cite[Lemma 1.2.15]{b-Lusztig94} the map 
    \begin{align*}
      \partial_\mu^L:U_\mu^+\rightarrow \bigoplus_{i\in I}U^+_{\mu-\alpha_i}, \qquad u\mapsto \sum_{i\in I} \partial^L_i(u)
    \end{align*}
    is injective for any $\mu>0$. By induction over $\mu\in Q^+$, using \eqref{eq:partialL-Gam} and the injectivity of $\partial^L_\mu$, we obtain that $\Gamma_\mu$ is uniquely determined and that $\Gamma_\mu\in \cB_\bc\ot U^+$ for all $\mu\in Q^+$. The uniqueness together with Theorem \ref{thm:intertwiner2} then imply that $\Gamma=\Theta^B$.
\end{proof}
%%%%%%%%%%%%%%%%%%%%%%%%%%%%%%%%%%%%%
  We now deduce from Theorem \ref{thm:intertwiner2} the existence of the ordinary quasi $K$-matrix, as proposed in \cite[2.3]{a-BaoWang18a} and established in \cite[Theorem 6.10]{a-BalaKolb19}, \cite[Theorem 7.3]{a-AppelVlaar22}, in the reformulation of \cite[Theorem 3.16]{a-WangZhang23}. Define $\Xfrak=(\vep\ot \id)(\Theta^B)\in \scrU$. By \eqref{eq:TB-def}, the element $\Xfrak$ can be written as an infinite sum $\Xfrak=\sum_{\mu\in Q^+}\Xfrak_\mu\in \prod_{\mu\in Q^+}U^+_\mu$.
%%%%%%%%%%%%%%%%%%%%%%%%%%%%%%%%%%%%%
\begin{cor}\label{cor:X-intertwine}
   The element $\Xfrak=(\vep\ot 1)(\Theta^B)=\sum_{\mu\in Q^+}\Xfrak_\mu$ satisfies the relation $\Xfrak_0=1$ and
   \begin{align}\label{eq:X-intertwine}
       \sigma_\tau(b) \Xfrak = \Xfrak (\sigma\circ \tau)(b) \qquad \mbox{for all $b\in \cB_\bc$}.
   \end{align}
\end{cor}
%%%%%%%%%%%%%%%%%%%%%%%%%%%%%%%%%%%%%
\begin{proof}
  The intertwiner relation \eqref{eq:X-intertwine} is obtained by applying $\vep\ot \id$ to the intertwiner relation \eqref{eq:intertwiner2}.
\end{proof}
%%%%%%%%%%%%%%%%%%%%%%%%%%%%%%%%%%%%%
The tensor quasi $K$-matrix $\Theta^B$ can be expressed in terms of the ordinary quasi $K$-matrix $\Xfrak$. This is how the tensor quasi $K$-matrix was originally defined in \cite[Definition 3.1]{a-BaoWang18a}. The calculation in the proof of the following result is an adaption of the calculation in the proof of \cite[Proposition 3.2]{a-BaoWang18a} to the present setting. Note that the element $\Xfrak=\sum_{\mu\in Q^+} \Xfrak_\mu$ is invertible in $\scrU$ as $\Xfrak_0=1$.
%%%%%%%%%%%%%%%%%%%%%%%%%%%%%%%%%%%%%%
\begin{cor}
    The relation $\Theta^B= \kow(\Xfrak)\cdot \Theta\cdot(\Xfrak^{-1}\ot 1)$ holds in $\scrU^{(2)}$.
\end{cor}
%%%%%%%%%%%%%%%%%%%%%%%%%%%%%%%%%%%%%%
\begin{proof}
    Set $\Gamma=\kow(\Xfrak)\cdot \Theta\cdot(\Xfrak^{-1}\ot 1)$. We can write $\Gamma=\sum_{\mu\in Q^+}\Gamma_\mu$ with $\Gamma_\mu\in \Uq\ot U^+_\mu$. By construction we have $\Gamma_0=1\ot 1$. For any $b\in \cB_\bc$ we calculate
    \begin{align*}
        \kow(\sigma_\tau(b)) \cdot\Gamma &  \stackrel{\phantom{\eqref{eq:X-intertwine}}}{=}\kow(\sigma_\tau(b)) \cdot  \kow(\Xfrak)\cdot \Theta\cdot(\Xfrak^{-1}\ot 1)\\
        &\stackrel{\eqref{eq:X-intertwine}}{=} \kow(\Xfrak) \cdot \kow(\sigma\circ \tau(b))\cdot \Theta \cdot (\Xfrak^{-1}\ot 1)\\
        &\stackrel{\eqref{eq:Theta-intertwiner}}{=} \kow(\Xfrak) \cdot \Theta \cdot (\obar\ot \obar)(\kow(\overline{\sigma\circ \tau(b)})\cdot (\Xfrak^{-1}\ot 1)\\
        &\stackrel{\eqref{eq:sigmakow}}{=} \kow(\Xfrak) \cdot \Theta \cdot (\sigma\ot \sigma)(\kow(\tau(b))\cdot (\Xfrak^{-1}\ot 1)\\
        &\stackrel{\eqref{eq:X-intertwine}}{=} \kow(\Xfrak) \cdot \Theta \cdot (\Xfrak^{-1}\ot 1)\cdot (\sigma_\tau\ot (\sigma\circ \tau))(\kow(b))\\
        &  \stackrel{\phantom{\eqref{eq:X-intertwine}}}{=}\Gamma \cdot (\sigma_\tau\ot (\sigma\circ \tau))(\kow(b)).
    \end{align*}
    Now Proposition \ref{prop:quasiK-unique} implies that $\Gamma=\Theta^B$.
\end{proof}
%%%%%%%%%%%%%%%%%%%%%%%%%%%%%%%%%%%%%%%
\begin{rema}
  Assume that the parameters $\bc=(c_i)_{i\in I\setminus X}\in (\qfield^\times)^{I\setminus X}$ satisfy condition \eqref{eq:Bbar-condition} and hence $\cB_{\bc}$ has a bar involution $\obar^B$ as discussed in Section \ref{sec:Bbar}. In this case the quasi $K$-matrices $\Theta^B$ and $\Xfrak$ coincide with the quasi $K$-matrices considered in \cite{a-BaoWang18a}, \cite{a-BalaKolb19}. Indeed, in this case the intertwiner properties
  \begin{align*}
    \kow(\overline{b}^B) \cdot\Theta^B = \Theta^B\cdot (\obar^B \ot \obar)\circ \kow(b), \qquad \mbox{and} \qquad \overline{b}^B \Xfrak = \Xfrak \overline{b}
  \end{align*}
  follow from Theorem \ref{thm:intertwiner2}, Corollary \ref{cor:X-intertwine} and the relation $\obar^B=\sigma_\tau\circ \sigma\circ\tau\circ \obar$ which holds for such a choice of parameters.
\end{rema}
%%%%%%%%%%%%%%%%%%%%%%%%%%%%%%%%%%%%%%%
Finally, we comment on the quasi $K$-matrix for nonstandard quantum symmetric pairs as discussed at the end of Section \ref{sec:QSP}. Let $\bs=(s_i)_{i\in I\setminus X}\in \qfield^{I\setminus X}$ and let $\chi^\bc_\bs:\cB_\bc\rightarrow \qfield$ be a character satisfying \eqref{eq:chics}. As explained in Section \ref{sec:QSP}, the character $\chi^\bc_\bs$ gives rise to an isomorphism of right $\Uq$-comodule algebras $\rho=\rho^\bc_\bs:\cB_\bc\rightarrow \cB_{\bc,\bs}$. We define a $\qfield$-algebra anti-automorphism $\sigma_\tau^{\bc,\bs}:\cB_{\bc, \bs}\rightarrow \cB_{\bc, \bs}$ by
\begin{align*}
  \sigma_\tau^{\bc,\bs}
= \rho\circ \sigma_\tau\circ \rho^{-1}.
\end{align*}
Moreover, we define $\Theta^B_{\bc,\bs}=(\rho \ot \id)(\Theta^B)$ and $\Xfrak_{\bc,\bs}=(\vep\ot \id)(\Theta^B_{\bc,\bs})$. These elements are the quasi $K$-matrices for nonstandard quantum symmetric pairs. The following corollary is a reformulation of \cite[Proposition 3.5]{a-DobKol19} in the setting of the present paper.
%%%%%%%%%%%%%%%%%%%%%%%%%%%%%%%%%%%%%%%
\begin{cor}
   For all $b\in \cB_{\bc, \bs}$ the relations
   \begin{align*}
       \kow(\sigma_\tau^{\bc,\bs}(b))\cdot \Theta^B_{\bc,\bs} = \Theta^B_{\bc,\bs}\cdot \big(\sigma_\tau^{\bc,\bs}\ot (\sigma\circ \tau)\big)(\kow(b)) \qquad \mbox{and}\qquad 
       \sigma_\tau^{\bc,\bs}(b)\Xfrak_{\bc,\bs} = \Xfrak_{\bc,\bs}(\sigma\circ \tau)(b)
   \end{align*}
   hold.
\end{cor}
%%%%%%%%%%%%%%%%%%%%%%%%%%%%%%%%%%%%%%%
\begin{proof}
    The first relation follows from Theorem \ref{thm:intertwiner2} by application of $\rho$ to the first tensor factor, using the fact that $\rho$ is a isomorphism of right $\Uq$-comodules. The second relation then follows by application of $\vep\ot \id$ to the first relation.
\end{proof}
%%%%%%%%%%%%%%%%%%%%%%%%%%%%%%%%%%%%%%%
\providecommand{\bysame}{\leavevmode\hbox to3em{\hrulefill}\thinspace}
\providecommand{\MR}{\relax\ifhmode\unskip\space\fi MR }
\providecommand{\MRhref}[2]{%
  \href{http://www.ams.org/mathscinet-getitem?mr=#1}{#2}
}
\providecommand{\href}[2]{#2}

%%%%%%%%%%%%%%%%%%%%%%%%%%%%%%%%%%%%%%%
\end{document}